\documentclass[11pt, a4paper, reqno]{amsart}
\usepackage{graphicx,psfrag}
\usepackage[psamsfonts]{amssymb}
\usepackage{pdfsync}
\usepackage{amscd}
\usepackage{amsmath}
\usepackage{mathtools}
\usepackage{mathrsfs}
\usepackage{stmaryrd}
\usepackage{comment}
\usepackage{xfrac}
\usepackage[margin=1in]{geometry}
\usepackage{comment}
\usepackage[shortlabels]{enumitem}
\usepackage{stmaryrd}
\usepackage[lowtilde]{url}
\usepackage{hyperref}
\usepackage{tikz-cd}
\usepackage{bbm}
\usepackage{dsfont}
\usepackage{wasysym}

\allowdisplaybreaks

\theoremstyle{plain}
    \newtheorem{theorem}{Theorem}[subsection]
    \newtheorem{lemma}[theorem]{Lemma}
    \newtheorem{corollary}[theorem]{Corollary}
    \newtheorem{proposition}[theorem]{Proposition}
    \newtheorem{conjecture}[theorem]{Conjecture}
    
\theoremstyle{definition}
    \newtheorem{definition}[theorem]{Definition}
    
    \newtheorem{example}[theorem]{Example}
    \newtheorem{notation}[theorem]{Notation}

    \newtheorem*{thank}{Acknowledgements}
\theoremstyle{remark}
    \newtheorem{remark}[theorem]{Remark}

\numberwithin{equation}{subsection}

\newcommand{\wt}[1]{\widetilde{#1}}
\newcommand{\EE}{\mathcal{E}}
\newcommand{\ZZ}{\mathbb{Z}}
\newcommand{\ztwo}{\ZZ/2\ZZ}

\newcommand{\R}{\mathbb{R}}
\newcommand{\C}{\mathbb{C}}
\newcommand{\QQ}{\mathbb{Q}}

\newcommand{\OO}{\mathcal{O}}
\newcommand{\II}{\mathscr{I}}

\newcommand{\LL}{\mathcal{L}}
\newcommand{\JJ}{\mathcal{J}}
\newcommand{\m}{\mathfrak{m}}

\newcommand{\IY}{\II_Y}
\newcommand{\JY}{\JJ_Y}
\newcommand{\hJY}{\hJJ_Y}

\newcommand{\hJJ}{\widehat{\JJ}}
\newcommand{\hstr}{\widehat{\str}}

\newcommand{\A}{\mathcal{A}}
\newcommand{\E}{\mathcal{E}}
\newcommand{\F}{\mathcal{F}}
\newcommand{\M}{\mathcal{M}}
\newcommand{\N}{\mathcal{N}}
\renewcommand{\P}{\mathcal{P}}
\newcommand{\PJ}{\P_\JJ}
\newcommand{\Pdm}{\P_{\D,\M}}
\newcommand{\Pmr}{\P_{\M_r}}
\newcommand{\bPdm}{\overline{\P}_{\D,\M}}
\newcommand{\Lhb}{\LL_\hb}
\newcommand{\Ehb}{\EE_\hb}
\newcommand{\PE}{\mathbb{P}(E)}

\newcommand{\Hom}{\operatorname{Hom}}

\newcommand{\Id}{\operatorname{Id}}

\newcommand{\codim}{\operatorname{codim}_\C}

\newcommand{\hb}{\hbar}
\newcommand{\Spec}{\operatorname{Spec}}
\newcommand{\Mor}{\operatorname{Mor}}
\newcommand{\ann}{\operatorname{Ann}}
\newcommand{\tr}{\operatorname{tr}}
\newcommand{\per}{\operatorname{Per}}
\newcommand{\hdr}{H_{\mathrm{DR}}}
\newcommand{\hone}{\hdr^1}
\newcommand{\htwo}{\hdr^2}
\newcommand{\hdrx}{\hdr^2(X)\series}
\newcommand{\sym}{\omega}
\newcommand{\omtry}{\Omega^{\geq 1}_Y}
\newcommand{\omtrx}{\Omega^{\geq 1}_X}
\newcommand{\aty}{At(\Q,Y)}
\newcommand{\V}{\mathcal{V}}
\newcommand{\vv}{\mathfrak{v}}
\newcommand{\xx}{\mathfrak{x}}
\newcommand{\cm}{\C^\times}
\newcommand{\kk}{\mathbbm{k}}
\newcommand{\kks}{\underline{\mathbbm{k}}}
\newcommand{\KKs}{\underline{\mathbbm{K}}}
\newcommand{\kkm}{\kk^{\times}}
\newcommand{\gm}{\mathbb{G}_m}

\newcommand{\pws}{\mathbb{K}}
\newcommand{\Q}{\str_\hb}
\newcommand{\tQ}{\wt{\str}_\hb}
\newcommand{\str}{\mathfrak{S}}
\newcommand{\strx}{\str_X}
\newcommand{\stry}{\str_Y}
\newcommand{\isom}{\xrightarrow{\sim}}

\newcommand{\quan}{\mathcal{Q}(X,\sym)}
\newcommand{\qyone}{\mathcal{Q}_1(X,\sym,Y)}
\newcommand{\qyonek}{\mathcal{Q}^{\gm}_1(X,\sym,Y)}
\newcommand{\qyr}{\mathcal{Q}_r(X,\sym,Y)}
\newcommand{\qyrk}{\mathcal{Q}^{\gm}_r(X,\sym,Y)}

\newcommand{\lb}{\llbracket}
\newcommand{\rb}{\rrbracket}
\newcommand{\epi}{\twoheadrightarrow}
\newcommand{\inj}{\hookrightarrow}
\newcommand{\lab}{\langle}
\newcommand{\rab}{\rangle}
\newcommand{\series}{\llbracket \hb \rrbracket}
\newcommand{\cpws}{\C\series}

\newcommand{\sympa}{\operatorname{Symp}\A}
\newcommand{\HH}{\operatorname{H}}
\newcommand{\sympah}{\lab \sympa, \HH \rab}
\newcommand{\T}{\mathcal{T}}
\newcommand{\TT}{\mathscr{T}}
\newcommand{\tatp}{\TT^+(\Q, Y)}
\newcommand{\tat}{\TT(\Q,Y)}

\newcommand{\send}{\mathcal{E}nd}
\newcommand{\pend}{\send^0}
\newcommand{\vu}{\varUpsilon}

\newcommand{\hten}{\widehat{\otimes}}

\newcommand{\gh}{\lab G,\lieh \rab}
\newcommand{\ghzero}{\lab G_0,\lieh_0 \rab}
\newcommand{\ghone}{\lab G_1,\lieh_1 \rab}
\newcommand{\ghtwo}{\lab G_2,\lieh_2 \rab}
\newcommand{\loc}{\operatorname{Loc}}
\newcommand{\X}{X}

\newcommand{\hinv}{\tfrac{1}{\hb}}
\newcommand{\hinvd}{\hb\pb\D}
\newcommand{\pd}{\varphi_\D}
\newcommand{\pdm}{\varphi_\M}
\newcommand{\pddm}{\varphi_{\D,\M}}
\newcommand{\Pd}{\Phi_\D}
\newcommand{\Pdpd}{\lab \Pd, \pd \rab}

\newcommand{\Conn}{\operatorname{Conn}^{\gm}_Y}

\newcommand{\hinvj}{\hb\pb\JJ}
\newcommand{\dd}{\prescript{\D}{}{\delta}}
\newcommand{\dm}{\prescript{\M}{}{\delta}}
\newcommand{\fd}{\prescript{\D}{}f}
\newcommand{\fm}{\prescript{\M}{}f}
\newcommand{\tP}{\prescript{t}{}{\bPdm}}
\newcommand{\tPJ}{\prescript{t}{}{\PJ}}
\newcommand{\tPP}{\prescript{t}{}{\P}}

\newcommand{\hck}{\lab \kkm, \kk \rab}
\newcommand{\hcK}{\lab \pws^\times, \pws \rab}
\newcommand{\gzero}{^{\geq 0}}
\newcommand{\gone}{^{\geq 1}}

\newcommand{\aeh}{\alpha_\hb}

\newcommand{\hcg}{\mathfrak{G}}
\newcommand{\hcG}{\mathcal{G}}
\newcommand{\hcggj}{\lab \hcG_\JJ, \hcg_\JJ \rab}
\newcommand{\hcgj}{\hcg_\JJ}
\newcommand{\hcGj}{\hcG_\JJ}

\newcommand{\der}{{\operatorname{Der}}}
\newcommand{\aut}{{\operatorname{Aut}}}
\newcommand{\autd}{\aut(\D)}

\newcommand{\autdj}{\aut(\D)_\JJ}
\newcommand{\derdj}{\der(\D)_\JJ}
\newcommand{\autder}{\lab \autd, \der(\D)  \rab}
\newcommand{\autderj}{\lab \autdj, \derdj \rab}
\newcommand{\derdm}{\der(\D,\M)}
\newcommand{\autdm}{\aut(\D,\M)}
\newcommand{\autderdm}{\lab \autdm, \derdm \rab}
\newcommand{\bautdm}{\overline{\aut}(\D,\M)}
\newcommand{\bderdm}{\overline{\der}(\D,\M)}
\newcommand{\bautderdm}{\lab \bautdm, \bderdm \rab}
\newcommand{\autpdm}{\aut_{\PJ}(\bPdm)}
\newcommand{\autdmr}{\aut(\D,\M_r)}
\newcommand{\derdmr}{\der(\D,\M_r)}
\newcommand{\autderdmr}{\lab \autdmr, \derdmr \rab}

\newcommand{\derehb}{\der(\Q,\Ehb)}
\newcommand{\GLr}{GL(r)}
\newcommand{\PGLr}{PGL(r)}
\newcommand{\pgl}{\mathfrak{pgl}}

\newcommand{\GLh}{GL_\hb(r)}
\newcommand{\glh}{\mathfrak{gl}_\hb(r)}
\newcommand{\glgl}{\lab \GLh, \glh  \rab}
\newcommand{\PGLh}{PGL_\hb(r)}
\newcommand{\pglh}{\pgl_\hb(r)}
\newcommand{\pglpgl}{\lab \PGLh, \pglh  \rab}

\newcommand{\Ppgl}{\P^0_{\PGLh}}

\newcommand{\ep}{\varepsilon}
\newcommand{\epder}{\ep_{\der}}
\newcommand{\epaut}{\ep_{\aut}}

\newcommand{\ugh}{\mathcal{U}\lieg_\hb}
\newcommand{\ughh}{\widehat{\mathcal{U}}\lieg_\hb}
\newcommand{\mh}{\mu^*_\hb}
\newcommand{\mhk}{\mu_\liek}

\newcommand{\Aut}{\operatorname{Aut}}

\newcommand{\lie}{\operatorname{Lie}}
\newcommand{\lieg}{\mathfrak{g}}

\newcommand{\liek}{\mathfrak{k}}
\newcommand{\lieh}{\mathfrak{h}}

\newcommand{\liep}{\mathfrak{p}}
\newcommand{\lieq}{\mathfrak{q}}

\newcommand{\lieu}{\mathfrak{u}}
\newcommand{\liel}{\mathfrak{l}}

\newcommand{\ld}{\lambda}

\newcommand{\gk}{(\lieg,K)}

\newcommand{\pb}{^{-1}}

\newcommand{\Ug}{\mathcal{U}\lieg}

\newcommand{\Ker}{\operatorname{Ker}}

\newcommand{\D}{\mathcal{D}}

\newcommand{\cl}{\overline}

\newcommand{\SLR}{SL(2, \R)}

\newcommand{\slr}{\mathfrak{sl}(2, \R)}
\newcommand{\slc}{\mathfrak{sl}(2, \C)}

\newcommand{\PP}{\mathbb{P}}

\newcommand{\GR}{G_\R}
\newcommand{\Gad}{G_{ad}}
\newcommand{\KR}{K_\R}

\newcommand{\RN}[1]{%
	\textup{\uppercase\expandafter{\romannumeral#1}}%
}

\newcommand{\xreg}{\wt{X}^{reg}}

\newcommand{\Image}{\operatorname{Im}}

\begin{document}

\title{Equivariant deformation quantization and coadjoint orbit method}

\author{Naichung Conan Leung}

\address[Naichung Conan Leung]{The Institute of Mathematical Sciences 
	and Department of Mathematics 
	The Chinese University of Hong Kong 
	Shatin, N.T.
	Hong Kong}
\email{leung@ims.cuhk.edu.hk}

\author{Shilin Yu}

\address[Shilin Yu]{Department of Mathematics,
	Texas A\&M University
	College Station, TX 
	USA}
\email{yus@math.tamu.edu}


\setlength{\abovedisplayskip}{7pt}
\setlength{\belowdisplayskip}{7pt}

\maketitle

\begin{abstract}
	The purpose of this paper is to apply deformation quantization to the study of the coadjoint orbit method in the case of real reductive groups. We first prove some general results on the existence of equivariant deformation quantization of vector bundles on closed Lagrangian subvarieties, which lie in smooth symplectic varieties with Hamiltonian group actions. Then we apply them to orbit method and construct nontrivial irreducible Harish-Chandra modules for certain coadjoint orbits. Our examples include new geometric construction of representations associated to certain orbits of real exceptional Lie groups.
\end{abstract}

\tableofcontents

\section{Introduction}\label{sec:intro}

\subsection{History}

The philosophy of coadjoint orbit method, as proposed by Kirillov and Kostant, suggested that unitary representations of a Lie group $\GR$ are closely related with coadjoint orbits in the dual $\lieg^*_\R$ of the Lie algebra $\lieg_\R$ of $\GR$. It was observed by Kirillov, Kostant and Sorieau that these coadjoint orbits are naturally equipped with symplectic structures induced by the Lie bracket of $\lieg_\R$. Therefore the path from coadjoint orbits to representations is an analogue of the general idea of \emph{quantization} of a classical mechanical system to a quantum mechanical system, even though there is no universal mathematical definition of this term. In the case of nilpotent Lie groups (\cite{Kirillov_nil}), solvable Lie groups of type I (\cite{AK}) and compact Lie groups, it is clear what `quantization' should mean: the unitary representations can be constructed as geometric quantization of the orbits (cf. \cite{Kirillov_quan}).

When $\GR$ is a noncompact real reductive Lie group, however, the picture becomes more complicated and mysterious. Geometric quantization still attaches representations to semisimple orbits, but it is invalid in general for nilpotent orbits due to lack of good invariant polarizations. Despite this, huge amount of efforts and progress have been made to attach representations to nilpotent orbits. Motivated by the work and conjectures of Arthur \cite{Arthur1, Arthur2}, Barbasch and Vogan \cite{BarbaschVogan} proposed that certain \emph{unipotent representations} should be attached to nilpotent orbits. These representations are conjectured to be the building blocks of the unitary dual of $\GR$, in such a way that any irreducible unitary representation of $\GR$ can be constructed from the unipotent representations (of subgroups) through the three standard procedures: taking parabolic induction, cohomological induction and complementary series. 

 One special feature of the reductive case is a marvelous correspondence discovered by Kostant, Sekiguchi \cite{Sekiguchi} and later strengthened by Vergne \cite{Vergne}. Fix a maximal compact subgroup $\KR$ of $\GR$ with Lie algebra $\liek_\R$. For simplicity of the statements, we assume here that $\GR$ is contained in the set of real points of a connected complex reductive Lie group $G$ with Lie algebra $\lieg = (\lieg_\R)_\C$. More precise conditions on the group $\GR$ will be addressed in the main body of the paper. Let $K \subset G$ be the complexification of $\KR$ with Lie algebra $\liek=(\liek_\R)_\C$. Then the Kostant-Sekiguchi-Vergne (KSV) correspondence states that there is a bijection between the set of nilpotent $\GR$-orbits in $\lieg_\R^*$ and the set of $K$-orbits in $\liep^* := (\lieg/\liek)^*$, and furthermore, diffeomorphisms between them. Then it was noticed by Vogan \cite{Vogan_AV} that the KSV correspondence has a bearing on orbit method. He observed that a $K$-orbit $\OO_\liep$ in $\liep^*$ is included as a closed Lagrangian subvariety in a $G$-orbit $\OO$ in $\lieg^*$, which has a natural algebraic symplectic form by the construction of Kirillov-Kostant-Sorieau. After analyzing the symplectic geometric information, Vogan conjectured that certain \emph{admissible} $K$-equivariant vector bundles on $\OO_\liep$, which roughly are square roots of the canonical bundle of $\OO_\liep$, should correspond to unitarizable irreducible $(\lieg,K)$-modules. The details of the definitions and the conjecture will be given in  \S\,\ref{subsec:adm}  and Conjecture \ref{conj:vogan}. To give a slightly more precise account, denote by $\codim(\partial \OO_\liep, \cl{\OO}_\liep)$ the complex codimension of the boundary $\partial \cl{\OO}_\liep = \overline{\OO}_\liep - \OO_\liep$ in the closure $\cl{\OO}_\liep$ of $\cl{\OO}_\liep$ in $\liep^*$. Then Vogan's conjecture can be summarized as follows.
 
 \begin{conjecture}\label{conj:intro}
 	Assume $\codim(\partial \OO_\liep, \cl{\OO}_\liep) \geq 2$. Then for any admissible vector bundle $E$ over $\OO_\liep$, there is an irreducible unitary representation $\pi$ of $\GR$, such that its space of $K$-finite vectors is isomorphic to the space of algebraic sections of $E$ as $K$-representations. 
 \end{conjecture}
 
 A nilpotent $K$-orbit which admits at least one admissible vector bundle is called an \emph{admissible} orbit. Schwartz \cite{Schwartz} showed that Vogan's notion of admissibilty of nilpotent $K$-orbits is equivalent to Duflo's version for nilpotent $\GR$-orbits in the original setting of orbit method (\cite{Duflo}). Therefore Vogan's conjecture can be regarded as a reformulation of the original orbit conjecture.
  
 For complex reductive groups, the admissible condition means that the vector bundles are flat. In this case,  Barbasch and Vogan  \cite{BarbaschVogan} provided a recipe to attach so-called \emph{special unipotent representations} to certain flat vector bundles on \emph{special nilpotent orbits} in the sense of Lusztig (\cite{Lusztig}). The case of real groups were discussed in \cite[Chapter 27]{ABV}. For both complex and real classical groups, the dual pair correspondence discovered by Howe \cite{Howe} has been proven a powerful tool to construct unipotent representations. See for example \cite{Li, HuangLi, Trapa, Barbasch1, LokeMa, He, MaSunZhu}. 

\subsection{Summary of method}

The current paper is part of the expedition to a uniform construction of quantizations of coadjoint orbits via the intrinsic (symplectic) geometry of the orbits, for both classical and exceptional Lie groups. Our method is based on the recent work of Losev \cite{Losev3, Losev2}, in which a new geometric approach to the orbit method for complex reductive Lie groups has been given using deformation quantization. Besides geometric quantization, deformation quantization is another mainstream formalization of quantization (\cite{BFFLS1, BFFLS2, DeWilde, Fedosov, Deligne}), which focuses on the algebras of observables of a physical system (Heisenberg picture), while geometric quantization emphasizes the spaces of states (Schr\"odinger picture). Many work have been done in this direction, including Fedosov's work \cite{Fedosov} in the context of $C^\infty$-symplectic manifolds and the work of Nest and Tsygan \cite{NestTsygan} in the context of holomorphic symplectic manifolds. The existence of deformation quantization of general Poisson manifolds were proven by the marvelous work of Kontsevich (\cite{Kontsevich}). The most relevant work for our setting is that of Bezrukavnikov and Kaledin \cite{BeKaledin}, in which they considered deformation quantization of any smooth algebraic symplectic variety $X$. By definition, a \emph{(formal) deformation quantization of $X$} is a deformation of the structure sheaf $\strx$ of $X$ into a sheaf of noncommutative algebras $\Q$ over the algebra $\cpws$ of formal power series, such that the first-order piece of the deformed associative product is determined by the Poisson bracket (see \S\,\ref{subsec:Fedosov}). Given a complex reductive Lie group $G$, any coadjoint $G$-orbit $\OO$ is naturally a smooth algebraic symplectic variety as similar to the case of real groups. Therefore it is natural to wonder if the construction in \cite{BeKaledin} can be applied to orbit method.

The main difficulty lies in the following difference between the algebraic/holomorphic case and the $C^\infty$-case: in the $C^\infty$-case the quantization always exists, while the existence in the algebraic case is only ensured when the symplectic variety satisfies certain admissibility condition (see \eqref{eq:adm}). The coadjoint orbits, however, are in general far from fulfilling this condition. On the other hand, the closure $\cl{\OO}$ of $\OO$ in $\lieg^*$, even though singular, has nice birational geometric properties. Its normalization $X$ has the so-called \emph{symplectic singularities} in the sense of Beauville \cite{Beauville1} (\cite{Beauville2, Hinich, Panyushev}). It was shown that $X$ admits certain partial resolutions $\wt{X}$, so-called  \emph{$\QQ$-factorial terminalization} (see, e.g., \cite{Fu, FJLS, Namikawa1, Namikawa2, Namikawa3, BCHM}), such that its smooth locus $\xreg$ does satisfy the admissibility condition (Prop. \ref{prop:adm}). 

Another issue is that deformation quantization produces algebras over the formal power series ring $\cpws$ instead of the polynomial ring $\C[\hb]$, and therefore one can not evaluate the algebras at $\hb = 1$ directly. There are no general results on convergence of the deformed product.  On the other hand, there is the $\cm$-action on any nilpotent orbit $\OO$, which is the square of the rescaling action on $\lieg^*$, such that the symplectic form has degree $2$ under this action. We call any such symplectic variety as a \emph{graded symplectic variety}. In this case,  Losev observed that one can consider \emph{graded quantization} $\Q$ of $\strx$, which admits a $\cm$-action lifting the $\cm$-action  on $\OO$ (or rather $\xreg$), such that the parameter $\hb$ has positive weight. Then the space of $\cm$-finite vectors in the global sections of such $\Q$ is a desired algebra over $\C[\hb]$ and the issue of convergence disappears. Following the two observations above, Losev constructed quantizations of $\xreg$ which restricts to quantizations of $\OO$. It leads to filtered deformations of the algebra $\C[\OO]$ of regular functions on $\OO$. Moreover, these deformations are Diximier algebras in the sense of Vogan \cite{Vogan_Dixmier}. In particular, they are Harish-Chandra modules for the complex group $G$ (regarded as real group). In the case of special rigid orbits, Losev's representations coincide with the special unipotent representations of Barbasch-Vogan (see \cite[\S\,4.3]{Losev3}).

In this paper, we extend Losev's method to study Vogan's orbit conjecture for real reductive Lie groups. An additional ingredient is deformation quantization of Lagrangian subvarieties \cite{BGKP, BC}. Suppose $X$ is a general smooth symplectic variety and $Y$ is a closed smooth Lagrangian subvariety of $X$. Given a (deformation) quantization $\Q$ of $X$ and a vector bundle $E$ over $Y$, a \emph{quantization of $E$} is a sheaf $\Ehb$ of $\Q$-modules, flat over $\cpws$, with an isomorphism $\Ehb / \hb \Ehb \cong (i_Y)_*E$ as $\strx$-modules, where $i_Y: Y \inj X$ denotes the inclusion. Such an quantization has the features of both geometric quantization (state spaces) and deformation quantization (observables). The papers \cite{BGKP, BC} gave sufficient and necessary conditions on $E$ and $\Q$ for the existence of $\Ehb$. For our purpose, we need to promote their constructions into certain equivariant versions. Similar to Losev's construction, we assume that $X$ is a graded smooth symplectic variety. Suppose that the Lagrangian subvariety $Y$ is preserved by the $\cm$-action and that the $\cm$-action lifts to an action on the vector bundle $E$. In this case, we say that $E$ or the pair $(Y,E) $ is \emph{graded}. If the $\cm$-action lifts to a quantization $\Ehb$ of $E$ compatible with the $\cm$-action on $\Q$, then we say that $\Ehb$ with the $\cm$-action is a \emph{graded quantization} of the pair $(Y, E)$. 

One of the main results in this paper is that graded quantization exists under certain natural assumptions. To give a brief account, we need the convenient language of Picard algebroids, cf. \cite[Defn. 2.1.3]{BB}. Namely, a Picard algebroid on $Y$ is a Lie algebroid $\TT$ which fits into a short exact sequence
   \[  0 \to \stry \to \TT \to \T_Y \to 0, \]
and the anchor map $\TT \to \T_Y$ and the inclusion $\stry \inj \TT$ are parts of the data. A standard example is the sheaf $\TT(L)$ of first-order algebraic differential operators acting on a line bundle $L$ over $Y$. However, the main feature of Picard algebroids is that they do not necessarily come from line bundles. One can define a vector space structure on the set of isomorphism classes of Picard algebroids over $Y$, such that the scalar multiplication $n \cdot \TT(L)$ by any integer $n$ of $\TT(L)$ is the Picard algebroid $\TT(L ^{\otimes n})$ of the tensor line bundle $L^{\otimes n}$. When $Y$ admits a $\cm$-action, it is natural to consider Picard algebroids with $\cm$-action, which we call \emph{graded Picard algebroids}.

According to \cite{BC} (cf. \S\,\ref{sec:quan_lag}), by descending to the first-order part, a usual (ungraded) quantization $\Q$ of $X$ defines a Picard algebroid $\tatp$ over $Y$. Moreover, the $\Q$-module structure on a quantization $\Ehb$ of $E$ gives rise to an action of $\tatp$ on $E$, which splits into two pieces of information: 

\begin{enumerate}[label=(D\arabic*)]
	\item  \label{data1}
	  an isomorphism $\tatp \to \frac{1}{r} \TT(\det(E))$ of Picard algebroids, where $\det(E)$ is the determinant line bundle of $E$ and $r$ is the rank of $E$, and
	\item \label{data2}
	  an algebriac flat connection on the projectivization $\PE$ of $E$.
\end{enumerate}

If we are in the graded case, all the related Picard algebroids are naturally graded and so it is natural to require that the isomorphism in \ref{data1} is $\cm$-equivariant. The flat connection in \ref{data2} also needs to be $\cm$-equivariant. In order to reconstruct the full quantization, however, we need some additional mild assumption on this flat connection. The flat connection lifts the infinitesimal action of $\C$ on $Y$ to a Lie algebra $\C$-action on $\PE$. On the other hand, the $\cm$-action on $E$ differentiates into a $\C$-action on $\PE$. We say that the flat connection on $\PE$ is \emph{strong} if these two $\kk$-actions coincide. If this happens, Theorem \ref{thm:lag_vec_graded} states that one can construct a unique graded quantization $\Ehb$ from the data \ref{data1} and \ref{data2}, which we summarize briefly below.

\begin{theorem}\label{thm:graded_main}
	Let $(X,\sym)$ be a graded smooth symplectic variety and let $\Q$ be any graded quantization of $(X,\sym)$. Suppose $Y$ is a closed Lagrangian subvariety in $X$ preserved by the $\cm$-action and $E$ is a graded vector bundle or rank $k$ over $Y$. Assume that the projectivization $\PE$ of $E$ admits a strong flat algebraic connection. Then any $\cm$-equivariant isomorphism $\tatp \to \frac{1}{r} \TT(\det(E))$ of graded Picard algebroids gives rise to a unique graded quantization $\Ehb$ of $E$, and vice versa.
\end{theorem}

We remark that, In the general ungraded setting of \cite{BGKP,BC}, the set of (equivalence classes of) all possible quantizations is too big for applications in orbit method: it admits a free action of an infinite dimensional group (cf. Thm \ref{thm:lag}, (2), Thm \ref{thm:lag_vec}, (2)). At each order of $\hb$, there are different choices how one can lift the quantization to the next stage. On the other hand, the main feature of graded quantizations considered in Theorem \ref{thm:graded_main} is that they are determined completely by their zeroth and first-order information (cf. Thm \ref{thm:quan_lag_graded}, Thm \ref{thm:lag_vec_graded}). This is a great advantage when we consider equivariant quantizations.

Suppose now that there is an algebraic group $K$ acting on $Y$ and $E$. Natural questions are whether the $K$-action lifts to a $K$-action on $\Ehb$ and whether the lifting is unique. First observe that, if such lift exists, then all the data in \ref{data1} and \ref{data2} have to be all $K$-equivariant. We discover that, if one only considers graded quantization with $K$-action that commute with the $\cm$-action, then these assumptions are enough for existence and uniqueness of the $K$-action (Prop \ref{prop:quan_lag_equiv}). Here requiring the $K$-action to commute with the $\cm$-action is crucial. The slogan is: the more restrictive the requirements are, the easier the lifting of the equivariant structure can be constructed.  In fact, this philosophy is adopted in various parts of the paper. Another example, which is also one of the main results, concerns relationship between Hamiltonian group actions and quantizations. If the symplectic variety $X$ is equipped with a Hamiltonian $G$-action, we can consider an \emph{even quantization} $\Q$ of $X$ in the sense of Losev (Defn \ref{defn:even}), which is  a graded quantization equipped with a $\ztwo$-action commuting with the $\cm$-action. The $\ztwo$-action is generated by an involutory anti-automorphism $\epsilon: \Q \to \Q$ that is identity modulo $\hb$ and sends $\hb$ to $-\hb$. Then, under the assumption that the $\cm$-action on the space $\Gamma(X, \strx)$ of regular functions on $X$ has no negative weights, the $G$-action on $X$ can be lifted uniquely to a $G$-action on $\Q$, which commutes with the $\cm\times\ztwo$ (Cor \ref{cor:equiv}). Moreover, the classical moment map can also be lifted  to a \emph{symmetrized quantized moment map} $\mh: \lieg \to \Gamma(X, \Q)$ (Defn \ref{defn:quan_moment}, Defn \ref{defn:quan_moment_symm}) which is compatible with both the even structure and the $G$-action (Prop \ref{prop:moment}). In this case we say that $\Q$ is a \emph{symmetrized Hamiltonian quantization} of $X$. We want to point out that existence of quantized moment map has been known in the $C^\infty$-setting (e.g., \cite{GuttRawnsley, Xu}) since long time ago, but all these results require additional assumptions on the topology of the manifold or the group. Our results handle much broader situations and therefore are of independent interest. For instance, in our setting the group $G$ can be disconnected and does not necessarily have to be reductive (see Remark \ref{rmk:ham}). 

Assume now that a subgroup $K \subset G$ preserves the Lagrangian subvariety $Y$ and acts on the vector bundle $E$. The $K$-action differentiates to a Lie algebra $\liek$-action on $\Ehb$. On the other hand, the restriction of the quantized moment map $\mh$ to $\liek$ also gives a second $\liek$-action. If these two $\liek$-actions coincide, we call such a quantization with the $K$-action as a \emph{Hamiltonian quantization} of $(Y,E)$. Note that in this case, the mere $K$-equivariant structures on the Picard algebroids in \ref{data1} are not enough. We need Lie algebra maps from $\liek$ into the global sections of the Picard algebroids, whose adjoint actions coincide with the differential of the $K$-action. Such structures are called \emph{strong $\liek$-actions} on the Picard algebroids. For instance, if $L$ is a $K$-equivariant line bundle then $\TT(L)$ carries a natural strong $\liek$-action. The strong $\liek$-actions are also compatible with scalar multiplication of Picard algebroids. We also need additional assumptions on the $K$-actions on $\PE$. Our next result is the existence of Hamiltonian quantization summarized below (see \S\,\ref{subsec:moment_lag} and Theorem \ref{thm:lag_equiv}). 

\begin{theorem}\label{thm:equiv_main}
	Let $X$, $Y$ be as in Theorem \ref{thm:graded_main}. In addition, assume that $X$ admits a Hamiltonian action of an algebraic group $G$ and $\Q$ is a symmetrized Hamiltonian quantization of $X$. Let $E$ be a graded $K$-equivariant vector bundle of rank $r$ over $Y$,  such that $\PP(E)$ is equipped with a strong flat connection. Then there exists a graded Hamiltonian quantization $\Ehb$ of $E$ equipped with a $K$-equivariant structure $\aeh$, if and only if the following conditions are satisfied:
	\begin{enumerate}
		\item 
		the vector fields generated by the $K$-action on $\PP(E)$ is horizontal with respect to the strong flat connection;
		\item 
		there exists an isomorphism $\frac{1}{2} \TT(K_Y) \isom \frac{1}{r} \TT(\det(E))$ of graded Picard algebroids which intertwines the natural strong $\liek$-action on $\frac{1}{2} \TT(K_Y)$ and the strong $\liek$-action on $\frac{1}{r} \TT(\det(E))$ induced by $\alpha$.
	\end{enumerate}
	In this case, the set of isomorphism classes of Hamiltonian quantization of $E$ is in bijection with the set of isomorphisms in (2).
\end{theorem}


We call such a quantization as a \emph{Hamiltonian quantization} of the pair $(Y,E)$. An immediate consequence is a $\gk$-module structure on the $\cm$-finite part of global sections of $\Ehb$ restricted evaluated at $\hb=1$. Note that all the results above are still valid if $\C$ and $\cm$ are replaced by $\kk$ and $\gm=\kkm$ for any field $\kk$ of characteristic zero. In the main body of the paper we will state our results for general $\kk$ until \S\,\ref{sec:quan_orbit}.

In apply \S\,\ref{sec:quan_orbit} we apply the general results to the the study of Conjecture \ref{conj:intro}. Given any $K$-orbit $Y=\OO_\liep$ from the conjecture, it is naturally embedded in a $G$-orbit $X=\OO$. We first take Losev's even quantization $\Q$ of $\OO$. It turns out that that the natural $\cm$-action on $\OO$ not only preserves $\OO_\liep$, but also lifts to an action on any admissible vector bundle $E$ in Vogan's conjecture, which commutes with the $K$-action (which is unique up to a degree shift, see \S \ref{subsec:grading}). Note that square root of $K_Y$ show up in both Vogan's conjecture and Theorem \ref{thm:equiv_main}. In fact, we will check that the $\cm$-action and the $K$-action satisfy all the assumptions of Theorem \ref{thm:equiv_main}. Therefore there exists a unique graded Hamiltonian quantization $\Ehb$ with $K$-equivariant structure. Let $\M_\hb$ be the $\cm$-finite part of $\Gamma(\OO_\liep, \Ehb)$. Then by the definition of Hamiltonian quantization, the module $\M:=\M_\hb|_{\hb = 1}$ is naturally a $(\lieg,K)$-module, where the $\lieg$-module structure is induced by the quantized moment map $\mh$ and the $K$-module structure comes from the $K$-action on $\Ehb$. We prove the following result (cf. Theorem \ref{thm:hc1} and Proposition \ref{prop:nonzero}). 

\begin{theorem}\label{thm:gk_main}
	The $(\lieg,K)$-module $\M$ constructed above is an admissible Harish-Chandra module of $\GR$. Moreover, if $\codim(\partial \OO_\liep, \cl{\OO}_\liep) \geq 3$, then the associated variety of the Harish-Chandra module $M$ is $\cl{\OO}_\liep$. 
\end{theorem}

 In \S\,\ref{subsec:list}, we list all admissible orbits of real exceptional groups satisfying the condition on codimension in Theorem \ref{thm:gk_main}.  One interesting feature of the Harish-Chandra module $\M$ we built is a canonical good filtration coming from the grading on $\M_\hb$. In \S\,\ref{subsec:spec}, we relate this filtration to another  filtration in the work of Schmid and Vilonen \cite{SchmidVilonen}.

\subsection{Outline}

Here is the structure of the paper. In \S\,\ref{sec:prelim} we review basic knowledge about nilpotent coadjoint orbits and state Vogan's version of the orbit conjecture. In \S\,\ref{sec:quan} we review the work of Berzukavnikov-Kaledin and Losev on the algebraic and graded deformation quantization. We present our results on equivariant quantization and quantized moment maps in \S\,\ref{subsec:equiv} and \S\,\ref{subsec:moment}. In \S\,\ref{sec:quan_lag}, we review quantization of Lagrangian subvarieties and state its grade versions. The proofs will appear in \S\,\ref{subsec:graded_lag_proof} and \S\,\ref{subsec:quan_lag_vec_proof}, after we review formal geometry and the language of Harish-Chandra torsors after \cite{BeKaledin} and \cite{BGKP} in \S\,\ref{subsec:hcpair} and \S\,\ref{subsec:hctorsor}. In \S\,\ref{sec:lag_group} we consider equivariant quantization of Lagrangian subvarieties (Prop \ref{prop:quan_lag_equiv}) and its compatibility with quantized moment map (Theorem \ref{thm:lag_equiv}). Finally, we apply all the previous results to construct quantization of nilpotent orbits in \S\,\ref{sec:quan_orbit}.

\begin{notation}
	We denote by $\kk$ a base field of characteristic zero and $\pws = \kk\series$. We write $\otimes = \otimes_\kk$ for tensor over $\kk$ and $\widehat{\otimes}$ for completed tensor. We denote by $\gm=\kkm$ the multiplicative group. We usually use $\str_X$ to denote the structure sheaves of a variety $X$. We denote by $\GLh = GL(r,\pws)$ resp. $\PGLh=PGL(r,\pws)$ the $\pws$-valued points of the group $\GLr$ resp. $\PGLr$ respectively and by $\glh$ resp. $\pglh$ their Lie algebras.
\end{notation}

\begin{thank}
	The paper would have never appeared without the inspiration from the work of Ivan Losev. We appreciate his answers to numerous questions about his work.  We are grateful to Baohua Fu for explanation of his and relevant work on singularities of nilpotent orbit closures. We would like to thank Jeffrey Adams, Dan Barbasch, Daniel Wong and David Vogan for teaching us about representation theory and unipotent representations. The second author is also indebted to Binyong Sun for the hospitality during his visit to the Chinese Academy of Sciences and stimulating discussions throughout the project. The work of the first author described in this paper was substantially supported by grants from the Research Grants Council of the Hong Kong Special Administrative Region, China (Project No. CUHK14303516) and a CUHK direct grant (Project No. CUHK 4053289). The second author was partially supported by the US National Science Foundation 1564398. 
\end{thank}

\section{Preliminaries on orbit method}\label{sec:prelim}

\subsection{Nilpotent orbits} \label{subsec:nil}

We review some basic facts about nilpotent orbits, whose details can be found in \cite[\S\,6]{Vogan_AV} and \cite{CM}. Let $\lieg$ be a complex reductive Lie algebra and $G$ be a complex reductive Lie group with Lie algebra $\lieg$. Choose a non-degenerate symmetric bilinear  Then the non-degenerate Killing form identifies $\lieg$ with $\lieg^*$ as $G$-representations. Therefore it is the same to study adjoint $G$-orbits in $\lieg$ as to study coadjoint $G$-orbits in $\lieg^*$.  Recall that a triple $\{ H, X ,Y \}$ in $\lieg$ is said to be a \emph{standard triple} if it satisfies the relations:
\begin{equation}\label{eq:stdtriple}
[H, X] = 2X, \quad [H,Y] = -2Y, \quad [X,Y]=H.
\end{equation}
We call the element $H$ (resp. $X$, $Y$) the \emph{neutral} (resp. \emph{nilpositive}, \emph{nilnegative}) element. It is known that $H$ must be a semisimple element. An standard example is $\slc$, which is spanned by
\begin{equation}
H_0 = \begin{pmatrix}
1 & 0 \\
0 & -1
\end{pmatrix},   \quad
X_0 = \begin{pmatrix}
0 & 1 \\
0 & 0
\end{pmatrix}, \quad
Y_0 = \begin{pmatrix}
0 & 0 \\
1 & 0
\end{pmatrix}.
\end{equation}
Then $\{ H_0, X_0, Y_0 \}$ satisfies the relations \eqref{eq:stdtriple} and is therefore a standard triple of $\slc$. We see that a standard triple $\{ H, X, Y\}$ in $\lieg$ is equivalent to a (complex) Lie algebra homomorphism $\phi : \slc \to \lieg$ satisfying
\[  \phi (H_0) = H, \quad \phi(X_0) = X, \quad \phi(Y_0) = Y.  \]

\begin{theorem}[Jacobson-Morozov]\label{thm:JM}
	If $X$ is a nonzero nilpotent element of $\lieg$, then it is the nilpositive element of some standard triple.
\end{theorem}

\begin{theorem}[Kostant]\label{thm:Kos}
	Any two standard triples $\{ H, X, Y \}$ and $\{ H', X, Y' \}$ with the same nilpositive element are conjugate by an element of $\Gad^X$, the centralizer of $X$ in the adjoint group $\Gad$ of $\lieg$.
\end{theorem}

Theorem \ref{thm:JM} and Theorem \ref{thm:Kos} combined imply that there is a bijection from the set of $G$-conjugacy classes of standard triples in $\lieg$ to the set of  nilpotent adjoint $G$-orbits in $\lieg$, which sends $\{ H ,X , Y\}$ to the orbit $\OO_X = G \cdot X \subset \lieg$.

Let $\GR$ be a real reductive Lie group in the sense of \cite[Defn 6.1]{Vogan_AV}. Let $\lieg_\R$ be the Lie algebra of $\GR$ and let $\lieg$ be the complexification of $\lieg_\R$. Let $\KR$ be a maximal compact subgroup of $\GR$ with Lie algebra $\liek_\R$. The corresponding Cartan involution is denoted by $\theta$. The associated Cartan decomposition is $\lieg_\R = \liek_\R \oplus \liep_\R$, with the complexification $\lieg = \liek \oplus \liep$. If $\{ h_\theta, x_\theta, y_\theta \}$ is a standard triple in $\lieg$ such that
\begin{equation}
\theta(h_\theta) = h_\theta, \quad \theta (x_\theta) = -x_\theta, \quad \theta (y_\theta) = -y_\theta,
\end{equation}
that is, $ h_\theta \in \liek$, $x_\theta, y_\theta \in \liep$, then we say that $\{ h_\theta, x_\theta, y_\theta \}$ is a \emph{normal (standard) triple}. We have the follwoing analogue of Theorem \ref{thm:JM} and \ref{thm:Kos} (\cite{KostantRallis}).

\begin{theorem}\label{thm:JM_K}
	Any nonzero nilpotent element $X \in \liep$ is the nilpositive element of a normal triple.
\end{theorem}

\begin{theorem}\label{thm:Kos_K}
	Any two normal triples $\{ h_\theta, x_\theta, y_\theta \}$ and $\{ h'_\theta, x'_\theta, y'_\theta \}$ with the same nilpositive element are conjugate by an element of $K^X_{ad}$, the centralizer of $X$ in the adjoint group $K_{ad}$ of the complex Lie algebra $\liek$.
\end{theorem}

In other words, there is a bijection from the set of $K$-conjugacy classes of normal triples in $\lieg$ to the set of nilpotent orbits in $\liep$, which sends $\{ h_\theta, x_\theta, y_\theta \}$ to the orbit $\OO_\liep = K_{ad} \cdot x_\theta \subset \liep \cong \liep^*$. Just like the case of standard triples, we can also repackage normal triples in terms of Lie algebra homomorphisms. Choose the Cartan involution $\theta_0$ of $\slc$ defined by
\begin{equation}
\theta_0(Z) = - ^t Z, \quad Z \in \slc,
\end{equation}  
where $^t Z$ denotes the transpose of matrices. We put
\begin{equation}
h_0 = \begin{pmatrix}
0 & i \\
-i & 0
\end{pmatrix},   \quad
x_0 = \frac{1}{2} \begin{pmatrix}
1 & -i \\
-i & -1
\end{pmatrix}, \quad
y_0 = \frac{1}{2} \begin{pmatrix}
1 & i \\
i & -1
\end{pmatrix}.
\end{equation}
Then the triple $\{ h_0, x_0, y_0 \} $ is a normal triple in $\slc$ with respect to the involution $\theta_0$. A normal triple in $\lieg$ is then equivalent to a Lie algebra homomorphism $\phi: \slc \to \lieg$ such that $\theta \circ \phi = \phi \circ \theta_0$.

Denote by $\sigma$ the complex conjugation of $\lieg$ with respect to the real form $\lieg_\R$. We also have the complex conjugation $\sigma_0$ on $\slc$ with respect to the real form $\slr$, which acts on the triple $\{ h_0, x_0, y_0 \} $ by:
\begin{equation}
\sigma_0 (h_0) = -h_0, \quad \sigma_0 (x_0) = y_0, \quad \sigma_0(y_0) = x_0.
\end{equation}
We introduce the following notations following \cite[\S\,6.2]{Yang}. Denote by $\Mor(\slc, \lieg)$ the set of all (complex) Lie algebra homomorphisms $\phi: \slc \to \lieg$. Define the sets
\begin{equation}
	\begin{split}
	  \Mor^\theta(\slc, \lieg) &= \{ \phi \in \Mor(\slc,\lieg) ~|~\theta \circ \phi = \phi \circ \theta_0 \},  \\
	  \Mor^\sigma(\slc, \lieg) &= \{ \phi \in \Mor(\slc,\lieg) ~|~\sigma \circ \phi = \phi \circ \sigma_0 \}, ~\text{and}  \\
	  \Mor^{\sigma,\theta}(\slc, \lieg) &=  \Mor^\sigma(\slc, \lieg) \cap \Mor^\theta(\slc, \lieg).  \\
	\end{split}
	\end{equation}
Note that $\Mor^\sigma(\slc, \lieg)$ is nothing but the set of all real Lie algebra homomorphisms $\slr \to \lieg_\R$.

\begin{proposition} \label{prop:stabG}
Suppose $\phi \in \Mor(\slc, \lieg) $. Define a standard triple $\{ H,X,Y \}$ by
\[ H = \phi(H_0), \quad X = \phi(X_0), \quad Y= \phi(Y_0).  \]
Then the following hold.
\begin{enumerate}
	\item \label{item:stabG1}
		We have a graded decomposition
			\begin{equation}\label{eq:lieg_grade}
				\lieg = \bigoplus_{k \in \ZZ} \lieg(k),
			\end{equation}
		where
		\begin{equation}
			\lieg(k) := \{ Z \in \lieg ~|~ [H, Z] = k Z \}, \quad \forall ~ k \in \ZZ. 
		\end{equation}
	\item 
	  Put 
		\begin{equation}\label{eq:parab_l}
			\liel = \lieg(0) \quad \text{and} \quad \lieu = \bigoplus_{k>0} \lieg(k), 
		\end{equation}
		then $\lieq = \liel \oplus \lieu$ is a Levi decomposition of a parabolic subalgebra of $\lieg$.
	\item 
	   According to the grading in \eqref{eq:lieg_grade}, the centralizer $\lieg^X$ has a graded decomposition
		\begin{equation}\label{eq:liegx}
			\lieg^X = \liel^X \oplus \bigoplus_{k>0} \lieg(k)^X = \liel^X \oplus \lieu^X. 
		\end{equation}
	\item 
		The subalgebra $\liel^X=\lieg^{H,X}$ is equal to $\lieg^\phi$, the centralizer in $\lieg$ of the image of $\phi$. It is a reductive subalgebra of $\lieg$ and the decomposition \eqref{eq:liegx} is a Levi decomposition of $\lieg^X$. 
\end{enumerate}
\end{proposition}

Similar results hold for normal triples.

\begin{proposition}\label{prop:stabK}
	Suppose $\phi \in \Mor^\theta(\slc, \lieg) $. Define a normal triple $\{ h_\theta, x_\theta, y_\theta \}$ by
	\[ h_\theta = \phi_\theta (h_0), \quad x_\theta = \phi_\theta (x_0), \quad y_\theta = \phi_\theta(y_0).   \]
	Then the following hold.
	\begin{enumerate}
		\item
		The parabolic subalgebra $\lieq = \liel \oplus \lieu$ constructed as in \eqref{eq:parab_l} with respect to $h_\theta$ is stable under the involution $\theta$.
		\item
		If we define
		\begin{equation}
		L_K := K^{h_\theta}, \quad U_K := \exp(\lieu \cap \liek), \quad Q_K := L_K U_K ,
		\end{equation}
		then $Q_K$ is the parabolic subgroup constructed as in \eqref{eq:parab_l} of $K$ with respect to $h_\theta$.
		\item
		We have $K^{x_\theta} \subset Q_K$ and that $K^{x_\theta}$ respects the Levi decomposition
		\begin{equation}\label{eq:Ktheta}
		K^{x_\theta} = L_K^{x_\theta} U^{x_\theta}_K. 
		\end{equation}
		\item
		The group $L_K^{x_\theta} = K^{h_\theta, x_\theta}$ is equal to $K^{\phi_\theta}$, the centralizer in $K$ of the image of $\phi_\theta$. It is a reductive algebraic subgroup of $K$. The subgroup $U_K^{x_\theta}$ is simply connected unipotent. In particular, the decomposition in \eqref{eq:Ktheta} is a Levi decomposition of $K^{x_\theta}$.
	\end{enumerate}    
\end{proposition}

Finally, by choosing a suitable $G$-(resp. $\GR$-)invariant symmetric complex (resp. real) bilinear form on $\lieg$ (resp. $\lieg_\R$), one can identify nilpotent adjoint $G$-orbits (resp. nilpotent adjoint $K$-orbits) in $\lieg$ (resp., $\liep$) with nilpotent coadjoint $G$-orbits (resp. nilpotent coadjoint $K$-orbits) in $\lieg^*$ (resp., $\liep^*$). See \cite[\S\,6]{Vogan_AV} for details. Therefore we can seamlessly switch between adjoint orbits and coadjoint orbits.

\subsection{Kostant-Sekiguchi-Vergne correspondence}

Let $\GR$ be as in \S\,\ref{subsec:nil}. Sekiguchi \cite{Sekiguchi} and Kostant (unpublished) established a bijection between the set of nilpotent $\GR$-orbits in $\lieg_\R \cong \lieg^*_\R$ and the set of nilpotent $K$-orbits in $\liep \cong \liep^*$. The  precise statement is as follows.

\begin{theorem}\label{thm:KSV}
	Let $\GR$ be a real reductive Lie group with Cartan involution $\theta$ and its corresponding maximal compact subgroup $\KR$ with complexification $K$. Let $\sigma$ be the complex conjugation on the complexification $\lieg$ of $\lieg_\R$. The the following sets are in $1-1$ correspondence.
	\begin{enumerate}[(a)]
		\item 
		  $\GR$-orbits on the cone $\N_\R$ of nilpotent elements in $\lieg_\R$;
		\item 
		  $\GR$-conjugacy classes of elements in $\Mor^\sigma(\slc, \lieg)$;
		\item 
		  $\KR$-conjugacy classes of elements in $\Mor^{\sigma,\theta}(\slc, \lieg)$;
		\item 
		  $K$-conjugacy classes of elements in $\Mor^\theta(\slc, \lieg)$;
		\item 
		  $K$-orbits on the cone $\N_\theta$ of nilpotent elements in $\liep$.
	\end{enumerate}
\end{theorem}

We remark that Vergne \cite{Vergne} showed the even stronger result that any two orbits related under this bijection are in fact diffeomorphic by using Kronheimer's instanton flow \cite{Kronheimer}. We call all the results above combined as the Kostant-Seikiguchi-Vergne (KSV) correspondence. 


\subsection{Admissible orbit data}  \label{subsec:adm}

We follow \cite{Vogan_AV} to define the so-called \emph{admissible orbit data}. Suppose that $\GR$ is a real reductive Lie group \emph{of Harish-Chandra class} in the sense of \cite[Defn, 5.21, Defn. 6.1]{Vogan_AV}. Suppose $x \in \liep^* \cong \liep $ is a (non-zero) nilpotent element. Let $\OO_\liep = K \cdot x$ be the associated nilpotent $K$-orbit in $\liep^*$. Let $K^x$ be the isotropy  subgroup of $K$ at $x$ with Lie algebra $\liek^x$. Define a character $\gamma: K^x \to \cm$ by
\begin{equation}\label{eq:gamma}
\gamma(k) : = \det (Ad(k) |_{(\liek / \liek^x)^*}).
\end{equation} 
Then the associated line bundle
\[ K \times_{K^x} \C_\gamma \to K/K^x \cong \OO_\liep  \]
is nothing but the canonical bundle of $\OO_\liep$. An irreducible representation $(\chi, V_\chi)$ of $K^x$ is said to be \emph{admissible} if the differential of $\chi$ equals $ \frac{1}{2} d \gamma$ times the identity, i.e.,
\begin{equation}\label{eq:chi}
d \chi( Z ) = \frac{1}{2} \tr ( ad(Z) |_{(\liek / \liek^x)^*}) \cdot \Id_{V_\chi} , \quad \forall ~ Z \in \liek^x.
\end{equation}
Then there is a $K$-equivariant algebraic vector bundle
\begin{equation}\label{eq:vchi}
\V_\chi = K \times_{K^x} V_{\chi} \to K/K^x \cong \OO_\liep 
\end{equation}
over the nilpotent orbit $\OO_\liep$. If such a $\chi$ exists, we say that $(x, \chi)$ is a \emph{(nilpotent) admissible orbit datum} and that the orbit $\OO_\liep$ is \emph{admissible}. Note that the notion of admissibility depends not only on the Lie algebras but also the groups.

Denote by $\codim(\partial \OO_\liep, \cl{\OO}_\liep)$ the complex codimension of the boundary $\partial \cl{\OO}_\liep = \overline{\OO}_\liep - \OO_\liep$ in the closure $\cl{\OO}_\liep$ of $\cl{\OO}_\liep$ in $\liep^*$.

\begin{definition}\label{defn:voganorbit}
	A nilpotent $K$-orbit $\OO_\liep$ is called a \emph{Vogan orbit} if  $\codim(\partial \OO_\liep, \cl{\OO}_\liep) \geq 2$.
\end{definition}

\begin{conjecture}[\cite{Vogan_AV}, Conjecture 12.1]\label{conj:vogan}
	Let $\GR$ be a real reductive Lie group of Harish-Chandra class. Suppose $(x, \chi)$ is a nilpotent admissible orbit datum such that the nilpotent $K$-orbit $\OO_\liep=K \cdot x$ is a Vogan orbit. Then there is an irreducible unitary representation $\pi(x, \chi)$ of $\GR$, such that its space of $K$-finite vectors is isomorphic to the space of algebraic sections of $\V_\chi$  as $K$-representations. 
\end{conjecture}

\subsection{$\cm$-actions on admissible vector bundles} \label{subsec:grading}

There is a natural $\cm$-action on any nilpotent $K$-orbit $\OO_\liep = K \cdot x$ which commutes with the $K$-action. It is the restriction of a $\cm$-action on the $G$-orbit $\OO=G \cdot x \subset \lieg$ defined in \cite[Lemma 1.3]{BK} and can be defined in exactly the same way. Moreover To define it, recall that by Theorem \ref{thm:JM_K} there exist $h \in \liek$, $y \in \liep$ such that the normal triple $\{ h, x, y  \}$ spans an $\slc$-subalgebra of $\lieg$ and in particular satisfies $[h, x] = 2x$. Moreover, the bijection between part (c) and (d) in Theorem \ref{thm:KSV} implies that we can always choose $h$ such that $h \in i \liek_\R$. Therefore the subgroup $C := \exp(\C h)$ in $K$ generated by $h$ is the complexification of the compact circle subgroup $\exp(i \R h) \subset K_\R$, therefore $C$ is a copy of $\cm$. Note that $C$ normalizes $K^x$, so there is a well-defined left action of $C$ on the quotient $K/K^x = \OO_\liep$ by
\[  z \cdot [k] : = [y \cdot z^{-1}], \quad \forall ~z \in C, ~ k \in K,  \]
where $[k]$ denotes the coset in $K/K^x$ represented by $k$. The homomorphism
\begin{equation}\label{eq:cms}
s: \cm \to C/(C \cap K_x), \quad \exp t \mapsto \overline{\exp(t h)},
\end{equation}
defines an action of $\cm$ on $\OO_\liep$, where $\overline{\exp(th)}$ denotes the image of $\exp(th) \in C$ in $C/(C \cap K_x)$. Since $[h,x]=2x$, this action coincides with the square of the Euler rescaling action of $\cm$ on $\overline{\OO}_\liep$ as a subvariety of the vector space $\liep$ and therefore $s$ is well-defined. In particular, we get a $\ZZ$-grading on the function ring $\C[\OO_\liep]$ compatible with the algebra structure. Moreover, the grading is nonnegative by \cite[Prop. 1.4]{BK}.

Given any admissible orbit datum $(x, \chi)$, we now show that there are also naturally defined $\cm$-actions on the vector bundle $\V_\chi$ commuting with its $K$-action, which lift the $\cm$-action on $\OO_\liep$. These actions will determine a grading on the space of global sections $\Gamma(\OO_\liep, \V_\chi)$ of $\V_\chi$, unique up to degree shift. Since the subgroup $C$ normalizes $K^x$, we can define the semi-direct product group $F:= K^x \rtimes C$. Recall that there is a semidirect decomposition $K^x = L^x_K U^x_K$ as in Proposition \ref{prop:stabK}, (4), such that $C$ commutes with $L^x_K$ and $C$ normalizes $U^x_K$. We want to extend the representation $\chi$ of $K^x$ uniquely to a representation $\wt{\chi}$ of $F$ on the same vector space $V_\chi$. For this purpose, we need to define the value of $\wt{\chi}$ on $C$. Note that the restriction of the character $\gamma$ in \eqref{eq:gamma} to the unipotent subgroup $U^x_K$ is trivial and therefore the action of $U^x_K$ on $V_\chi$ is also trivial. Moreover, the group $C$ commutes with $L^x_K$. Thus $\chi$ can be extended if and only if we choose $\wt{\chi}|_{C}$ to be a one parameter subgroup of $GL(V_\chi)$ which commutes with all elements in the subgroup $\chi(L^x_K)$. When $\chi$ is irreducible on $V_\chi$, the only possibility is that $\wt{\chi}|_{C}$ acts by a character $C \to \cm$. Conversely, any character satisfies the requirement.  

We let $F$ acts on the product $K\times V_\chi$ diagonally. This action descends to an action of $C \cong F/K^x$ on $K\times_{K^x}V_\chi = \V_\chi$ which commutes with the $K$-action on $\V_\chi$. Different choices of the character $\wt{\chi}|_{C}$ only result in a twist of the $C$-action. Note that in general the quotient map $C \to C/(C\cap K^x)$ is a nontrivial cover and the homomorphism $s: \cm \to C/(C\cap K^x)$ in \eqref{eq:cms} does not necessarily lift to a homomorphism $\cm \to C$.  One can always adjust $\wt{\chi}|_{C}$ so that the $C$-action on $ \V_\chi$ descends to an action of $C/(C\cap K^x)$, and hence $\cm$ acts on $K\times_{K^x}V_\chi = \V_\chi$ via $s: \cm \to C/(C\cap K^x)$. Alternatively, we can fix $\wt{\chi}|_C$ to be the trivial character and allow ourselves to talk about gradings with fractions, i.e., we get a $(\ZZ + \frac{1}{n})$-grading on the $\C[\OO_\liep]$-module $\Gamma(\OO_\liep, \V_\chi)$ for some nonzero integer $n$, which is compatible with the $\ZZ$-grading on $\C[\OO_\liep]$. The constructions of graded quantization of Lagrangian subvarieties in \S~\ref{sec:quan_lag} still apply in this case without any change. Lastly, note that different choices of the character $\wt{\chi}|_{C}$ only cause a degree shift on the grading of $\Gamma(\OO_\liep, \V_\chi)$.


\section{Deformation quantization}\label{sec:quan}

\subsection{Quantization of symplectic structures}\label{subsec:Fedosov}

Throughout the paper, $\kk$ is a fixed field of characteristic zero. We denote by $\pws := \kk\series$ the algebra of formal power series in the parameter $\hb$. The results in \S\,3 - 6 work for general $\kk$. Later in \S\,\ref{sec:quan_orbit} we will set $\kk$ to be $\C$ when we come back to the orbit conjecture. 

We review deformation quantization in the algebraic context from \cite{BeKaledin}. Note that the holomorphic case was previously studied by Nest and Tsygan in \cite{NestTsygan}. Fix a base scheme $S$ which is of finite type over $\kk$. An \emph{$S$-manifold} $X$ is a scheme $X/S$ of finite type and smooth over $S$. We denote by $\Omega_{X/S}^\bullet$ the relative de Rham complex. Its hypercohomology groups are the relative de Rham cohomology groups $\hdr^\bullet(X/S)$. By a \emph{symplectic form} of $X/S$, we mean a closed non-degenerate relative form $\sym \in H^0(X, \Omega_{X/S}^2)$. Such an $\sym$ induces an (relative) algebraic Possion bracket $\{ - , -\}$ on the structure sheaf $\str_X$ of $X$.  A \emph{(deformation) quantization} of $(X,\sym)$ is the data consisting of

\begin{itemize}
	\item 
	a sheaf $\Q$ of flat associative $\pws$-algebras on $X$ with multiplication $*$, complete in the $\hb$-adic topology,
	\item   
	an isomorphism $\iota: \Q / \hb \Q \to \strx$ of sheaves of commutative algebras,
\end{itemize}
such that, if $f$ and $g$ are two functions over some open subset of $X$ with arbitrary liftings $\tilde{f}$ and $\tilde{g}$ in $\Q$, the image of the commutator
\[ [\tilde{f}, \tilde{g}] = \tilde{f} * \tilde{g} - \tilde{g} * \tilde{f} \in \hb \Q  \]
in $\hb \Q / \hb^2 \Q \cong \Q/\hb\Q \cong \strx$ is equal to the Poisson bracket $\{f,g\}$. 
There is an obvious definition of equivalent quantizations. We are particularly interested in the case when $X$ is a smooth symplectic variety.

\begin{remark}
To simplify notations, we will only consider the case when $S = \Spec \kk$ most of the time and omit $S$ from the formulae. However, all results still work in the relative setting, which will be useful in the formulation of quantization with equivariant structures (e.g, \S\,\ref{subsec:graded}, \ref{subsec:moment}). We will make remarks about the relative case when necessary.
\end{remark}

For a $C^\infty$-symplectic manifold $M$, Fedosov showed that deformation quantization of $M$ always exists and can be classified by certain cohomology classes. Bezrukavnikov and Kaledin \cite{BeKaledin} reformulated Fedosov's approach for algebraic schemes/varieties in terms of the language of formal geometry (\ref{sec:formal}). Let $X$ be a smooth symplectic variety with symplectic $2-$form $\sym$. Denote by $\quan$ the set of equivalence classes of quantizations of $(X,\sym)$. There exists a canonical \emph{(noncommutative) period map}
\begin{equation} \label{eq:per}
\per: \quan \to \hb\pb \hdrx,
\end{equation}
where $\hdrx$ denotes the $\hb$-adic completion of $\htwo(X) \otimes \pws$ (\cite[Defn 4.1]{BeKaledin}). The variety $X$ is said to be \emph{admissible} if the canonical map
\begin{equation}\label{eq:adm}
\hdr^i(X) \to H^i(X,\strx)
\end{equation}
is surjective for $i=1,2$. This condition was also obtained In this case, it was shown in \cite[Thm 1.8]{BeKaledin} that the period map is \emph{injective}, i.e., the \emph{period} $\per(\Q)$ classifies the quantization $\Q$. Moreover, the image of the period map lies in $\hb\pb[\sym] +  \hdrx$, i.e., $\per(\Q)$ of any quantization $\Q$ as a formal power series has constant term $[\sym]$. Let $\Omega^{\geq 1}_X$ denote the truncated de Rham complex of $X$,
\[ 0 \to \Omega^1_X \to \Omega^2_X \to \cdots, \]
which is considered as a subcomplex of the algebraic de Rham complex $(\Omega^\bullet_X,d)$. We have a short exact sequence of graded complexes  
  \[ 0 \to \Omega^{\geq 1}_X \to \Omega^\bullet_X \to \strx \to 0.  \]
Taking cohomology gives a long exact sequence
\[ \cdots \to H^\bullet(X,\omtrx) \to \hdr^\bullet(X) \to H^\bullet(X,\strx) \to \cdots, \]
where the first canonical map is denoted by
\[ H^\bullet(X,\omtrx) \to \hdr^\bullet(X), \quad \lambda \mapsto \lambda_{\mathrm{DR}}.\] 
If $X$ is admissible, the group $H^2_F(X):=H^2(X,\omtrx)$ coincides with the kernel of the natural map $\htwo(X) \to H^2(X,\strx)$. The main result of \cite{BK} is the following:

\begin{theorem}\label{thm:per}
	Pick any splitting $P: \htwo(X) \epi H^2_F(X)$ of the inclusion $H^2_F(X) \inj \htwo(X)$. Then the map $\Q \mapsto P (\per(\Q) - \hbar\pb [\sym])$ defines a bijection $\quan \isom H^2_F(X)\series$.
\end{theorem}

We will the review the definition of the noncommutative period map $\per$ in \S\,\ref{subsec:graded}, which relies on preliminaries in \S\,\ref{subsec:hcpair}.

\subsection{Graded quantization}\label{subsec:graded} 

We recall the definition of \emph{graded quantizations} from \cite[Section 2.2]{Losev1}. Suppose a smooth symplectic variety $X$ over the field $\kk$ admits a $\gm$-action such that the symplectic form $\sym$ has degree $2$ with respect to the $\gm$-action: the form $t.\sym = t^* \sym$ obtained from $\sym$ by the push-forward via the automorphism induced by $t \in \kkm$ equals $t^2\sym$. In this case we say that $(X,\sym)$ is \emph{graded}. Note that the weight $2$ can be replaced by any \emph{positive} integer and the results of the paper are still valid.

\begin{definition}\label{defn:graded}
	A quantization $\Q$ of $\strx$ is said to be \emph{graded} if the $\gm$-action on $X$ lifts to a $\gm$-action on $\Q$ by algebra automorphisms such that $t.\hb = t^2\hb$ (i.e., $\deg \hb = 2$) and the isomorphism $\Q/\hb\Q \cong \strx$ is $\gm$-equivariant. We require isomorphisms between graded quantizations to be $\gm$-equivariant.
\end{definition}

Losev constructed a natural $\gm$-action on the set $\quan$ such that isomorphism classes of graded quantizations correspond to $\gm$-fixed points in $\quan$. Note the inverse is less obvious, i.e., that a quantization $\Q$ such that $[\Q] \in \quan$ is $\gm$-stable has a graded structure. It is also true that two graded quantizations are $\gm$-equivariantly isomorphic if and only if they are isomorphic as usual quantizations. See \cite[\S~2.2]{Losev1} for details. Thus we can simply identify the set of isomorphism classes of graded quantizations with $\quan^{\gm}$, the $\gm$-fixed set of $\quan$. Moreover, Losev showed that the period map is $\gm$-equivariant, if one defines the $\gm$-action on $\hb\pb\hdrx$ such that $\hb$ has weight $2$. Note that $\sym$ is exact by Cartan's magic formula. For the same reason, the induced action of $\gm$ on $\hdr^\bullet(X)$ is trivial. 

We also need the notion of even quantization from \cite{Losev1}, which will play important roles in the remaining part of the paper. 

\begin{definition}\label{defn:even}
	A quantization $\Q$ is said to be \emph{even}, if there is an involutory anti-automorphism $\epsilon: \Q \to \Q$ that is identity modulo $\hb$ and sends $\hb$ to $-\hb$. The involution $\epsilon$ is referred as the parity involution. We say that the pair $(\Q,\epsilon)$ is an \emph{even quantization}. If $\Q$ is a graded quantization, we require $\epsilon$ to be $\gm$-equivariant.
\end{definition}

\begin{remark}
	It is straightforward from the definition that the involution $\epsilon$, if exists, is unique up to $\cm$-equivariant automorphisms of $\Q$.
\end{remark}

The following result is due to Losev (Prop. 2.3.2, \cite{Losev1}).

\begin{proposition}\label{prop:graded}
	Let $(X,\omega)$ be a graded smooth symplectic variety. Then
	\begin{enumerate}
		\item 
		If $\Q \in \quan$ is graded, then $\per(\Q) \in \htwo(X, \kk) \subset \hb\pb\hdrx$.
		\item   
		If $\Q$ is even, then $\per(\Q) = \hb\pb[\omega] = 0$.
		\item 
		If $X$ is, in addition, admissible, then a graded quantization $\Q$ is even if and only if $\per(\Q)=0$.
	\end{enumerate}
	
\end{proposition}

If a quantization $\Q$ is even, there might be many different parity involutions $\epsilon$. Similar to the case of graded quantization, one can show that for any graded even quantization $\Q$, any two parity involutions are conjugated by a graded automorphism of $\Q$. Moreover, we will prove Proposition \ref{prop:even_auto} in \S\,\ref{subsec:moment}, which implies that such automorphism is unique. Now if $H^1(X,\strx) = H^2(X,\strx)=0$, then $X$ is admissible and $H^2_F(X) \cong \htwo(X)$. Therefore it follows from Theorem \ref{thm:per} and Proposition \ref{prop:graded} that

\begin{corollary}\label{cor:per_graded}
	Let $(X, \sym)$ be a smooth symplectic variety with a $\gm$-action as above. Suppose $H^1(X,\strx) = H^2(X,\strx)=0$. Then the period map $\per$ identifies the set of isomorphism classes of graded quantizations of $X$ with $\htwo(X) \subset \hb\pb\hdrx$. 
\end{corollary}

In fact, in most applications we have $H^1(X,\strx) = H^2(X,\strx)=0$. For the rest of the paper, we will always assume $X$ satisfies this condition unless otherwise mentioned. It is convenient to make the following definition.

\begin{definition}\label{defn:strong_adm}
	A smooth symplectic variety $(X, \sym)$ is said to be \emph{strongly admissible} if it satisfies the condition that $H^1(X,\strx) = H^2(X,\strx)=0$.
\end{definition}

\begin{remark}
	The notions of graded and even quantization still work in the relative setting. One only needs to require that $S$ admits a $\gm$-action and that the map $X \to S$ is $\gm$-equivariant. See \cite[\S\,2.2]{Losev1} for details.
\end{remark}

\subsection{Equivariant quantization}\label{subsec:equiv}

Throughout this section, we make the following assumption on the $\gm$-action  $X$:

\begin{enumerate}[label=(W)] 
\item \label{cond:wt} 
  The $\gm$-action on $\Gamma(X,\strx)$ has no negative weights.
\end{enumerate}

Let $G$ be any algebraic group over $\kk$. Suppose $G$ acts on $X$ and the action preserves the symplectic form $\sym$ and commutes with the $\gm$-action. A natural question to ask is whether the $G$-action can be lifted to an action on a given quantization $\str_\hb$ of $(X,\sym)$. More precisely, let $\beta: G\times X \to X$ be the action map, let $p_X: G\times X \to X$ be the projection onto $X$, and let $p_G: G\times X \to G$ be the projection onto $G$. Define $\beta^*\Q := \beta\pb\Q \, \widehat{\otimes}_{\pws} \, p\pb_G\series$ and similarly for $p^*_X$. A $G$-equivariant structure on $\Q$ is an isomorphism $\beta^* \Q \cong p^*_X \Q$ as sheaves of $p_G^{-1} \str_G \series$-algebras, which is compatible with the isomorphism $\beta^*\strx \cong p^*_X \strx$ given by the $G$-action on $X$, and satisfies the natural cocycle condition. We can think of such structure as a $G$-action on $\Q$, denoted as $\beta_\hb$. Note that both $\beta^* \Q$ and $p^*_X\Q$ are naturally quantization of $G\times X$ over $G$ with the relative symplectic form $p^*_X \sym$. The $\gm$-action on $G\times X$ is defined as the product of the trivial action on $G$ and the given action $\beta$ on $X$. When $\Q$ is graded or even, so are $\beta^* \Q$ and $p^*_X\Q$. In these cases, we require $\beta^* \Q \cong p^*_X \Q$ to be equivariant with respect to the $\gm$ or $\gm\times\ztwo$-action.

We first make the following key observation. Let $\der(\Q)$ be the sheaf of (continous) $\pws$-linear derivations of $\Q$ and let $\mathcal{H}$ be the sheaf of vector fields which are locally Hamiltonian (see \cite[(3.3)]{BK}). They both admit natural $\gm$-actions. Note that there is a canonical $\gm$-equivariant isomorphism between $\der(\Q) / \hb \der(\Q)$ and $\mathcal{H}$. Consider the short exact sequence of graded sheaves
\begin{equation}\label{exsq:H}
0 \to \underline{\kk} \to \str_X \to \hb\mathcal{H} \to 0, 
\end{equation}
where the map $\str_X \to \hb\mathcal{H}$ is defined by $f \mapsto \hb\{ f, - \}$, and $\underline{\kk}$ is the constant sheaf with coefficient $\kk$. Note that the map $\str_X \to \hb\mathcal{H}$ is $\gm$-equivariant since the Poisson bracket has degree $-2$. By the associated long exact sequence of sheaf cohomology and the assumption \ref{cond:wt}  about the $\gm$-action on $X$, the induced $\gm$-action on $\Gamma(X, \hb \der(\Q) / \hb^2\der(\Q)) \cong \Gamma(X, \hb\mathcal{H})$ has no negative weights. For the same reason, $\Gamma(X, \hb^{p} \der(\Q) / \hb^{p+1}\der(\Q)) \cong \Gamma(X, \hb^p\mathcal{H})$ has only positive weights for each $p \geq 2$. Then an induction argument shows that zero is the only $\gm$-invariant section in $\Gamma(X, \hb^2 \der(\Q))$.

\begin{proposition}\label{prop:even_auto}
	Assume the condition \eqref{cond:wt}. Then the identity map is the only $\gm\times\ztwo$-equivariant $\pws$-linear automorphism of an even quantization $(\Q,\epsilon)$ which descends to the identity map on $\str_X$ $mod ~\hb$ is the identity automorphism.
\end{proposition}
\begin{proof}
	Any automorphism of $\Q$ which descends to the identity map on $\strx$ has the form $\exp(d)$, where $d \in \hb \der(\Q)$ is a $\pws$-linear derivation of $\Q$. Therefore it suffices to show that the only $\gm\times\ztwo$-invariant element in $\hb \der(\Q)$ is zero. The induced $\ztwo$-action on $\hb \der(\Q) / \hb^2\der(\Q)$ by $\epsilon$ equals multiplication by $-1$ (see \cite[Lemma 3.6]{BK}). Therefore the $\gm\times\ztwo$-invariant elements in $\hb \der(\Q)$ must lie in $\hb^2 \der(\Q)$, but it has to be zero by the observation above.
\end{proof}

\begin{corollary}\label{cor:equiv}
	Let $G$ be any algebraic group. Suppose $G$ acts on a (strongly) admissible $X$ preserving $\sym$. Then this action can be lifted to a unique $G$-action on an even quantization $(\Q,\epsilon)$ such that it commutes with the $\gm\times\ztwo$-action.
\end{corollary}

\begin{proof}
	The $\gm\times\ztwo$-action on $\Q$ induces $\gm\times\ztwo$-actions on $\beta^*\Q$ and $p^*_X\Q$, making them into graded even quantizations of $G\times X$ relative to $G$. Note that $G\times X$ is also (strongly) admissible relative to $G$. Moreover, the $\gm$-action on $G\times X$ also satisfies the condition \ref{cond:wt}. Therefore by the relative version of Proposition \ref{prop:graded} and Proposition \ref{prop:even_auto}, there is a unique isomorphism $\beta^*\Q \cong p^*\Q$ which intertwines the $\gm\times\ztwo$-actions, which automatically satisfies the cocycle condition. The uniqueness of the $G$-action also follows from Proposition \ref{prop:even_auto}.
\end{proof}

\begin{remark}
	In the case when the quantization $\Q$ is not even or the $G$-action is not required to commute with the $\gm\times\ztwo$-action, one has to assume in addition that $\per(\Q)$ is fixed by $G$ for the existence of the isomorphism $\beta^*\Q \cong p^*_X\Q$. Moreover, $G$ needs to be reductive so that one can adjust the isomorphism $\beta^*\Q \cong p^*_X\Q$ to make it satisfy the cocycle condition. The argument is similar to the proof of existence of graded quantization in \cite[\S 2.2]{Losev1}. In the even case, however, this restriction on the group $G$ is unnecessary. 
\end{remark}

\begin{definition}\label{defn:symm_beta}
	We say that a $G$-action $\beta_\hb$ on an even quantization $(\Q,\epsilon)$ of $(X,\omega)$ is \emph{symmetrized} if it commutes with the $\gm\times\ztwo$-action.
\end{definition}

\subsection{Hamiltonian quantization}\label{subsec:moment}

We now consider the case when the $G$-action $\beta$ on $(X,\sym)$ is Hamiltonian with a moment map $\mu: X \to \lieg^*$. We assume that $\mu$ is $\gm$-equivariant with respect to the $\gm$-action on $\lieg^*$ of weight $2$, i.e., $\mu(t.x) = t^2 \mu(x)$, $\forall~ x \in X, t \in \gm=\kkm$, so that the copy of $\lieg$ in $\kk[\lieg^*]$ has weight $2$. Then we call the quadruple $(X,\sym, \beta, \mu)$ a \emph{(graded classical) Hamiltonian $G$-space}. The main example we are interested in is a complex nilpotent orbit $\OO$ with the $\cm$-action introduced in \S\,\ref{subsec:grading} and the moment map $\mu$ being the embedding $\OO \hookrightarrow \lieg^*$.

We form the \emph{homogeneous universal enveloping algebra} $\ugh$ of the Lie algebra $\lieg$, which is the $\kk[\hb]$-algebra generated by $\lieg$ subject to the relations $xy-yx = \hb [x,y]$. It admits the adjoint action of $G$ and a natural $\gm$-action such that $\hb$ and any element in $\lieg$ have weight $2$. The two actions commute with each other. Recall the definition of quantized (co-)moment maps (cf. \cite[\S 3.4]{BPW}). 

\begin{definition}\label{defn:quan_moment}
	A \emph{(graded) quantum Hamiltonian $G$-action} on a graded quantization $\Q$ of $(X,\sym)$ consists of 
	\begin{itemize}
		\item  
		a $G$-equivariant structure $\beta_\hb$ on $\Q$ which lifts the $G$-action on $X$ and commutes with the $\gm$-action, and
		\item 
		a $G\times\gm$-equivariant continuous $\kk[\hb]$-algebra homomorphism $\mu^*_\hb: \ugh \to \Gamma(X, \Q)$,
	\end{itemize}
	such that for any $x \in \lieg$, the adjoint action $[\hb\pb\mh (x), -]$ on $\Q$ agrees with the Lie algebra action induced by the $G$-action $\beta$ on $\Q$. The map $\mu^*_\hb$ is called a \emph{quantized (co-)moment map}. We also say that $(\Q,\beta_\hb, \mu^*_\hb)$ is an \emph{quantum Hamiltonian quantization} of the classical Hamiltonian $G$-space $(X,\sym,\beta,\mu)$.
\end{definition}

The term `quantized moment map' is justified by the fact that the composite map
\[   \kk[\lieg^*] \cong S\lieg \cong \ugh/\hb\ugh \xrightarrow{\mh|_{\hb=0}} \Gamma(X,\Q)/\hb\Gamma(X,\Q) \hookrightarrow \Gamma(X, \Q/\hb\Q)  \cong \Gamma(X,\str_X)  \]	
induces a $G\times\gm$-equivariant classical moment map $\mu: X \to \lieg^*$ as above. We say that the quantum Hamiltonian $G$-action $(\beta_\hb,\mh)$ lifts the classical quantum $G$-action $(\beta,\mu)$. Conversely, one can ask when a classical moment map can be lifted to a quantized one. We will show below that there is a nice answer when $\Q$ is even. First of all, we make the following definition after Losev (\cite[\S\,5.4]{Losev1}).

\begin{definition}\label{defn:quan_moment_symm}
	Assume $(X,\sym,\beta,\mu)$ is a graded classical Hamiltonian $G$-space and with an quantum Hamiltonian quantization $(\Q, \beta_\hb, \mu^*_\hb)$. Suppose $\Q$ is an even with the involution $\epsilon$ and the quantum $G$-action $\beta_\hb$ is symmetrized in the sense of Definition \ref{defn:symm_beta}. We say that the quantized moment map $\mu^*_\hb$ is \emph{symmetrized}, if $\epsilon(\mh(\xi)) = \mh(\xi)$ for all $\xi \in \lieg$. In this case, we say that the datum $(\Q,\epsilon, \beta_\hb, \mh)$ is a \emph{symmetrized $G$-Hamiltonian quantization} and $(\beta_\hb,\mh)$ is a \emph{symmetrized quantum Hamiltonian $G$-action}.
\end{definition}

\begin{remark}
	Note that there is a unique anti-automorphism $\varsigma$ of the algebra $\ugh$  satisfying $\varsigma(\xi) = \xi$, $\forall~\xi \in \lieg$, and $\varsigma(\hb)=-\hb$. The definition above is equivalent to requirement that $\mu^*_\hb \circ \varsigma = \epsilon \circ \mu^*_\hb$.
\end{remark}

Next, consider the sheaf $\Q[\hb\pb]$ of algebras with the $\gm$-action extended from that on $\Q$, such that $t.h^i = t^i h^i$, $\forall~t \in \gm$, $i \in \ZZ$. It is naturally a sheaf of graded Lie algebras under commutator bracket. The subsheaf $\hb\pb\Q$ is a subsheaf of graded Lie algebras. We consider the diagram of two short exact sequences of sheaves of Lie algebras
\begin{equation}\label{diag:der}
\begin{tikzcd}
0  \arrow{r}  & \hb\pb\KKs \arrow{r} \arrow[d]  & \hb\pb \Q   \arrow[r, "\varphi"]  \arrow[d, "p"]   & \der(\Q)    \arrow[d, "q"]  \arrow[r] & 0    \\
0  \arrow{r}  & \hb\pb\kks \arrow{r}   & \hb\pb \str_X    \arrow[r, "\psi"]    & \mathcal{H} \arrow[r] & 0,
\end{tikzcd}
\end{equation}
where $\kks$ and $\KKs$ are the constant sheaves over $X$ with coefficients $\kk$ and $\pws$ respectively. The two vertical maps $p$ and $q$ on the right are the canonical quotient maps modulo $\hb$. The map $\varphi: \hb\pb \Q \to \der(\Q)$ is given by $a \mapsto [a,-]$. The sequence at the bottom is just \eqref{exsq:H} multiplied by $\hb\pb$. The Lie bracket on $\hb\pb\strx$ is given by $(\hb\pb f, \hb\pb g) \mapsto \hb\pb \{ f, g \}$. The map $\psi: \hb\pb\Q \to \mathcal{H}$ is given by $\hb\pb f \mapsto \{ f, -\}$. 

A similar discussion applies to $\ugh$. Regard $\hb\pb\ugh$ as a subspace of the graded algebra $\ugh[\hb\pb]$, then it is closed under commutator bracket and therefore is naturally a Lie algebra with $\gm$-action. Note that $\hb\lieg$ is a Lie subalgebra of $\hb\pb\ugh$ which is isomorphic to the Lie algebra $\lieg$ and lies in the $\gm$-invariant part. From now on we will identify the Lie algebra $\lieg$ with the Lie subalgebra $\hb\pb\lieg$, which will prove to be a great convenience. When a $G$-equivariant structure on $\Q$ is given, we can reformulate the definition of a quantized moment map $\mh$ as a $G\times\gm$-equivariant Lie algebra homomorphism $\hb\pb\lieg  \to \Gamma(X, \hb\pb \Q)$, still denoted by $\mh$. Since the $\gm$-action on $\hb\pb\lieg$ is trivial, this is equivalent to saying that the image of $\mh$ lies in the $\gm$-invariant part of $\Gamma(X, \hb\pb \Q)$. We can also regard a $G\times\gm$-equivariant classical moment map $\mu: X \to \lieg^*$ as a Lie algebra homomorphism $\mu^*: \hb\pb\lieg \to \Gamma(X, \hb\pb\strx)$ whose image lies in the $\gm$-invariant part.

Now suppose $(\Q, \epsilon)$ is an even quantization. Then the involutive anti-automorphism $\epsilon$ induces a $\ztwo$-action on $\der(\Q)$ by Lie algebra \emph{automorphisms}. We equip $\hb\pb\Q$ with an involution $\tilde{\epsilon}$ which is transported from $\epsilon$ of $\Q$ via the tautological isomorphism $\hb\pb\Q \cong \Q$. Note that The resulting $\ztwo$-action on $\hb\pb\Q$ \emph{preserves} the Lie bracket on $\hb\pb\Q$ defined above and the induced action on $\hb\pb\pws$ is given by $\tilde{\epsilon}(h^i) = (-1)^{i+1} h^i$. The $\ztwo$-actions on sheaves at the bottom induced from the ones on the top via the quotient maps are simply trivial. Then all maps in \eqref{diag:der} are $\gm\times\ztwo$-equivariant. Then a symmetrized quantized moment map is equivalent to a $G$-equivariant Lie algebra homomorphism $\mh: \hb\pb\lieg  \to \Gamma(X, \hb\pb \Q)$ whose image lies in the $\gm\times\ztwo$-invariant part of $\Gamma(X, \hb\pb \Q)$.

\begin{proposition}\label{prop:moment}
	Suppose a graded symplectic variety $(X,\sym)$ is equipped with a Hamiltonian $G$-action $\beta$ with a $G\times\gm$-equivariant moment map $\mu$. Then any even quantization $(\Q, \epsilon)$ of $X$ admits a unique symmetrized quantum Hamiltonian $G$-action $(\beta_\hb, \mh)$ which lifts $(\beta,\mu)$.
\end{proposition}

\begin{proof}
	By Corollary \ref{cor:equiv}, the $G$-action $\beta$ on $X$ lifts uniquely to a $G$-action $\beta_\hb$ on $\Q$ that commutes with the $\gm\times\ztwo$-action. Taking cohomology of \eqref{diag:der} gives a diagram of two long exact sequences of Lie algebras
	\begin{equation}\label{diag:moment}
	\begin{tikzcd}
	0  \arrow{r}  & \Gamma(X, \hb\pb\KKs) \arrow{r} \arrow[d]  & \Gamma(X, \hb\pb \Q)   \arrow[r, "\varphi"] \arrow[d, "p"]   & \Gamma(X, \der(\Q))   \arrow[d, "q"] \arrow[r] & H^1(X,\hb\pb\KKs)  \arrow[d]  \\
	0  \arrow[r]  & \Gamma(X, \hb\pb\kks) \arrow{r}   & \Gamma(X, \hb\pb \str_X)   \arrow[r, "\psi"]    & \Gamma(X, \mathcal{H})  \arrow[r]  & H^1(X, \hb\pb\kks),
	\end{tikzcd}
	\end{equation}
such that all maps are $G\times\gm\times\ztwo$-equivariant. The differential of $\beta_\hb$ gives a Lie algebra homomorphism $d\beta_\hb: \hb\pb\lieg \to \Gamma(X, \der(\Q))$, such that its image lies in the $\gm\times\ztwo$-invariant part of $\Gamma(X, \der(\Q))$ and $q \circ d\beta_\hb = d\beta$. One has to show that $d\beta_\hb$ lifts uniquely to a $G\times\gm\times\ztwo$-equivariant Lie algebra homomorphism $\hb\pb\lieg \to \Gamma(X, \hb\pb \Q)$. Take the pullback $\hb\pb\tilde{\lieg}:= \Gamma(X, \hb\pb \str_X)  \times_{\Gamma(X, \mathcal{H})} \hb\pb\lieg$, then we obtain an extension of Lie algebras with $\gm\times\ztwo$-actions, 
	\begin{equation} \label{exsq:tlieg}
	0 \to \Gamma(X,\hb\pb\KKs) \to \hb\pb\tilde{\lieg}  \to \hb\pb\lieg \to 0,
	\end{equation}
	since the $\gm\times\ztwo$-invariant part of $H^1(X,\underline{\kk}) \subset H^1(X,\hb\pb\KKs)$ is zero. Note that the $\gm\times\ztwo$-action on each term in\eqref{exsq:tlieg} is rational. Moreover, the $\gm\times\ztwo$-action on $\hb\pb\lieg$ is trivial, while the $\gm\times\ztwo$-invariant part of $\Gamma(X,\hb\pb\KKs)$ is zero. Therefore the extension has a canonical splitting that identifies $\hb\pb\lieg$ with the $\gm\times\ztwo$-invariant part of $\hb\pb\tilde{\lieg}$, which is also $G$-equivariant. Hence there is a unique $G\times\gm\times\ztwo$-equivariant Lie algebra homomorphism $\mh: \hb\pb\lieg \to \Gamma(X, \hb\pb \Q)$ such that $\varphi \circ \mh = d\beta_\hb$. This is the desired moment map.
	
	It remains to check that $\mu^*_\hb$ induces $\mu^*: S\lieg \to \Gamma(X, \strx)$ or $\mu^*: \hb\pb\lieg \to \Gamma(X, \hb\pb\strx)$. Diagram chasing in \eqref{diag:moment} gives
	\begin{equation}\label{eq:Phi}
	\psi \circ p \circ \mh =  q \circ \varphi \circ \mh = q \circ d\beta_\hb = d\beta = \psi \circ \mu^*.  
	\end{equation}
	Note that the image of $\hb\pb\lieg$ in $\mu^*$ sits inside the $\gm$-invariant part of $\Gamma(X, \hb\pb \str_X)$ and hence maps injectively into $\Gamma(X,\mathcal{H})$ via $\psi$, since $\Gamma(X, \hb\pb\kks)$ is of weight $-2$. Therefore $p \circ \Phi = \mu^*$ by \eqref{eq:Phi}.
\end{proof}

\begin{remark}\label{rmk:ham}
	The existence of quantized moment map have been known since long time ago, at least in the $C^\infty$-setting (e.g., \cite{GuttRawnsley, Xu}). However, the traditional argument requires additional assumptions on the topology of the manifold or the group (e.g., connected, semisimple, etc.), which are the same assumptions ensuring the existence of classical moment maps. In some of the most important yet common examples, however, the classical moment maps do exist but these assumptions fail. For instance, a cotangent bundle of a smooth $G$-variety with a general algebraic group $G$ does not necessarily satisfy these conditions, but in this case one can explicitly write down both the classical and quantized moment maps. A natural question is whether there are some natural mild assumptions which ensure that a given classical moment map can be lifted (uniquely) to a quantized moment map. Our results seem to give a satisfactory answer. The cotangent bundle example, for instance, satisfies the condition \ref{cond:wt} and hence is covered by Proposition \ref{prop:moment}. Proposition \ref{prop:moment} can also be generalized to a relative version, which then applies to the case of twisted cotangent bundles. Another essential advantage of our approach is that the group $G$ can be rather general, e.g., can be disconnected and does not have to be reductive, which is quite convenient for the study of representation theory of disconnected Lie groups. 
\end{remark}

\section{Quantization of Lagrangian subvarieties}\label{sec:quan_lag}

\subsection{Quantization of line bundles}\label{subsec:quan_lag}

Let $(X,\sym)$ be a smooth symplectic variety with a quantization $\Q$ of the structure sheaf $\strx$. Suppose $Y$ is a smooth Lagrangian subvariety inside $X$ with the embedding denoted by $i_Y : Y \inj X$. Let $L$ be a line bundle over $Y$ and set $\LL=(i_Y)_* L$ to be the direct image sheaf as $\str_X$-module. A natural question is if there exists a $\Q$-module $\Lhb$, flat over $\pws$, with an isomorphism $\Lhb/\hb\Lhb \cong \LL$ of $\strx$-modules. We call such $\Lhb$ as a \emph{quantization of $\LL$} or $L$ or $(Y,L)$. We also say that $\Lhb$ is a \emph{rank one quantization of $Y$}. A complete answer in the general setting was given in \cite[Thm 1.4]{BGKP}. In this section, we will summarize their results in the form which can be promoted to the graded and equivariant setting.

Let $\IY$ be the ideal subsheaf of $\str_X$ associated to $Y$ and let $\JY$ and $\JY'$ be the preimage of the ideal $\IY$ and $\IY^2$ resp. under the projection $\Q \epi \strx$. We write $\JY^2 = \JY * \JY \subset \Q$, which is a subsheaf of two-sided ideals of $\Q$. Then $\hinv\JY$, $\hinv\JY'$ and $\JY^2$ are all subsheaves of Lie ideals in the sheaf $\hb\pb\Q$ of Lie algebras define in \S\,\ref{subsec:moment}. Consider the sheaf 
\begin{equation}\label{eq:tat}
\tatp : = \hinv \JY/ \hinv\JY^2,
\end{equation}
then it is a sheaf Lie algebras supported on $Y$. The inclusions $\JJ_Y^2 \subset \JJ_Y' \subset \JJ_Y$ give a short exact sequence of sheaves of Lie algebras
\begin{equation}\label{exsq:J}
0 \to \hinv\JY' / \hinv\JY^2 \to \hinv\JY/ \hinv\JY^2 \to \hinv\JY / \hinv\JY' \to 0.
\end{equation} 
It was shown in \cite[Lemma 5.4]{BGKP} that there are canonical isomorphisms
\begin{equation} \label{eq:J_tat}
\hinv \JJ_Y / \hinv\JJ'_Y \cong \T_Y, \quad \hinv \JJ_Y' / \hinv \JJ_Y^2 \cong \stry,
\end{equation}
 where the first one comes from the Lie algebra action $\hinv\JY\times \hinv\JY' \to \hinv\JY'$ by commutator and the second isomorphism. Therefore we have the short exact sequence
\begin{equation} \label{exsq:tatp}
0 \to \stry \to \tatp \to \T_Y \to 0.
\end{equation}
We endow $\tatp $ with the $\stry$-module structure given by 
  \[  \stry\times\tatp \to \tatp, \quad (f, \partial) \mapsto  f * \partial, \]
or equivalently, the map $\Q/\JY \times \hinv\JY/ \hinv\JY^2 \to \hinv\JY/ \hinv\JY^2$ that descends from the left multiplication of $\Q$ on $\hinv\JY$. We simply write the multiplication as $f\partial$ or $f \cdot \partial$. Then \eqref{exsq:tatp} becomes a sequence of $\stry$-modules and $\tatp$ is a Picard algebroid in the sense of \cite[Defn. 2.1.3]{BB}.

On the other hand, there is a second $\stry$-module structure on $\JY/\JY^2$ defined as follows. Define the \emph{symmetrized product} on $\Q$,
  \begin{equation} \label{eq:symmprod}
     a \bullet b := \frac{1}{2}(a * b + b * a),
  \end{equation}
  which is commutative but not associative. The symmetric product descends to the map 
  \[ \stry\times\hinv\JY / \hinv\JY^2 \to \hinv\JY / \hinv\JY^2, \quad (f, \partial) \mapsto  f \bullet \partial\] 
since $\JY$ is a sheaf of two-sided ideals. One can check that $\bullet$ is compatible with the Lie brackets on $\hinv\JY / \hinv\JY^2$, hence gives a second Lie algebroid structure on $\hinv\JY / \hinv\JY^2$. Moreover, the morphisms in \eqref{exsq:tatp} are still $\stry$-linear when $\hinv\JY / \hinv\JY^2$ is equipped with the second $\stry$-module structure, therefore we get a second Picard algebroid, denoted by $\tat$, which is the same as $\tatp=\hinv\JY / \hinv\JY^2$ as a sheaf of Lie algebras, but equipped with a different $\stry$-module structure. We denote the identity map by 
\begin{equation}\label{eq:tat2tatp}
s: \tat \to \tatp.
\end{equation}
It is graded, compatible with Lie brackets and satisfy the relation
\begin{equation} \label{eq:slinear}
s(f \bullet \partial) - f s(\partial) = \frac{1}{2} \wt{\partial}(f), 
\end{equation}
where $f$ is a locally defined function and $\wt{\partial}$ denotes the image of a local section $\partial$ of $\tat$ in $\TT_Y$. As in \cite[\S\,1]{BC}, this gives an isomorphism of Picard algebroids
\begin{equation}\label{eq:tatp2tatky}
v: \tatp \isom \tat + \frac{1}{2}\TT(K_Y).
\end{equation}
Note that in \cite[\S\,1]{BC}, $\tatp$ was directly defined as \eqref{eq:tatp2tatky}. Our definition is equivalent to theirs. In the special case when $X$ is the cotangent bundle of $Y$, our definition specializes to that of \cite[\S 2.4]{BB}. 

\begin{remark}
	In \cite{BGKP} the sheaf $\JY/\JY^2$ was used instead of $\hinv\JY/\hinv\JY^2$, etc. Of course our sheaves are tautologically isomorphic to theirs, but it is slightly more convenient when we consider the graded situation.
\end{remark}

In \cite{BGKP}, a class $\aty \in H^2(Y,\Omega_Y\gone)$ was defined as the Atiyah class $c_1(\tat)$ of the Picard algebroid $\tat$ (for details, see \cite{BB}, Lemma 2.1.6 and the paragraph below). This class is closely related to the period of $\Q$, which is of the form
\[ \per(\Q) = \hb\pb[\sym] + [\sym_1] + \hb[\sym_2] + \hb^2[\sym_3] + \cdots \in \hb\pb\hdrx, \]
where $[\sym_i]$ stands for the de Rham cohomology class of some closed $2$-form $\sym_i(\Q) \in \htwo(X)$. It was shown in \cite[Lemma 5.8]{BGKP} that the De Rham cohomology class corresponding to $\aty$ satisfies
\begin{equation} \label{eqn:aty}
  \aty_{\mathrm{DR}} = i^*_Y (\sym_1(\Q)), 
\end{equation}
where $i^*_Y: \htwo(X) \to \htwo(Y)$ is the restriction map of De Rham cohomology groups induced by the embedding $i_Y: Y \inj X$.

Now for any line bundle $L$ on $Y$, recall from \cite[\S 2.1.12]{BB} that there is an associated Picard algebroid $\TT(L)$, which is defined as the sheaf of first order differential operators acting on (local sections of) $L$. The Atiyah class $c_1(L) = c_1(\TT(L)) \in H^2(\omtry))$ of $\TT(L)$ is a lift of the usual first Chern class $c_1(L)_{\mathrm{DR}} \in \htwo(Y)$.

Let $K_Y = \Omega^{\dim Y}_Y$ be the canonical line bundle of $Y$. The following result in \cite[Theorem 1.4]{BGKP} gives the necessary and sufficient condition for the existence of formal quantization of $\LL$ and classifies all such quantizations. Here we denote by $\F((\stry\series)^\times)$ the group of isomorphism classes of $(\stry\series)^\times$-torsors on $Y$ with a flat algebraic connection.

\begin{theorem}\label{thm:lag}
	\begin{enumerate}
		\item 
		The line bundle $L$ admits a quantization if and only if the following two conditions hold:
		\begin{enumerate}
			\item \label{cond:c1}
			$ c_1(L) = \frac{1}{2} c_1(K_Y) + \aty$ holds in $H^2(\omtry)$;
			\item 
			$i^*_Y \sym_i(\Q) = 0$ holds in $\htwo(Y)$, $\forall ~ i \geq 2$.
		\end{enumerate}	
		\item 
		If the set $\qyone$ of isomorphism classes of quantizations of line bundles on $Y$ is non-empty, then this set admits a free and transitive action of the group $\F((\stry\series)^\times)$ of isomorphism classes of $(\stry\series)^\times$-torsors on $Y$ with a flat algebraic connection.  	  
	\end{enumerate}
\end{theorem}

Note that, by the isomorphism \eqref{eq:tatp2tatky}, condition \eqref{cond:c1} in Theorem \ref{thm:lag} is equivalent to the existence of an isomorphism of Picard algebroids between $\TT(L)$ and $\tatp$.

\subsection{Graded quantization of line bundles}\label{subsec:graded_lag}

Now let us come back to the setting in \S~\ref{subsec:quan_lag} and consider the graded situation. Suppose $(\X,\sym)$ is graded such that the Lagrangian subvariety $Y \subset X$ is preserved by the $\gm$-action. In this case we say that $Y$ is a \emph{graded Lagrangian subvariety} of $(X,\sym)$. From now on we always assume that we are in the graded setting. 

\begin{definition}\label{defn:graded_line}
	Suppose $\Q$ is a graded quantization of $(X,\sym)$. A rank one quantization $\Lhb$ of $(Y,L)$ is said to be \emph{graded} is a usual quantization equppied with a $\gm$-action such that
	\begin{itemize}
		\item 
		  the isomorphism $\Lhb/\hb\Lhb \cong (i_Y)_*L$ is $\gm$-equivariant;
		\item 
		  he $\gm$-action satisfies $t.\hb = t^2 \hb$ and $t.(um) = (t.u)(t.m)$, for any $t \in \gm = \kkm$, local section $u$ of $\Q$ and local section $m$ of $\Lhb$. 
	\end{itemize}
	Two graded quantization of $Y$ are said to be \emph{$\gm$-equivariantly isomorphic} or simply \emph{equivalent} if there exists a $\gm$-equivariant isomorphism between them. We denote by $\qyonek$ the set of all equivalence classes of rank one graded quantization of $Y$.
\end{definition}

The goal of the section is to state the necessary and sufficient condition for the existence of graded quantization of line bundles. The proof will be postponed to \S\,\ref{subsec:graded_lag_proof}. Note that, in the graded setting, all the three sheaves in \eqref{exsq:J} admit natural $\gm$-actions inherited from the $\gm$-action on $\hb\pb\Q$ and all the connecting maps are $\gm$-equivariant. Moreover, the isomorphisms $d$ in \eqref{eq:J_tat} are $\gm$-equivariant with respect to the natural $\gm$-actions on $\T_Y$ and $\stry$. Similarly, all the natural maps appeared in \S\,\ref{subsec:quan_lag} concerning $\tat$ and $\tatp$ are $\gm$-equivariant. Therefore $\tat$ and $\tatp$ are \emph{graded Picard algebroids}, i.e., Picard algebroids equipped with $\gm$-actions which are compatible with all the structure maps. We require isomorphisms between graded Picard algebroids to be $\gm$-equivariant.

We denote by $\F_{\gm}(\stry^\times)$ the group of all equivalence classes of $\gm$-equivariant $\stry^\times$-torsors with $\gm$-invariant algebraic flat connections. It has a subgroup $\Conn$, the set of all equivalence classes of $\gm$-equivariant algebraic flat connections of the trivial $\stry^\times$-torsor (or, equivalently, the trivial bundle $\str_Y$) equipped with the canonical $\gm$-action. The main result of this subsection is the following.

\begin{theorem}\label{thm:quan_lag_graded}
	Suppose $(X,\sym)$ is a graded smooth symplectic variety with a graded quantization $\Q$. Let $Y$ be a graded Lagrangian subvariety of $X$ and let $L$ be a $\gm$-equivariant line bundle over $Y$. Then 
	\begin{enumerate}
		\item 
		  $L$ admits a graded quantization if and only if there exits a $\gm$-equivariant isomorphism of graded Picard algebroids,
		  \begin{equation}\label{eq:cond_sqrt}
		  \vu^+: \tatp \to \TT(L).
		  \end{equation}
		\item 
		  The set of equivalence classes of all graded quantization of a fixed $L$, if nonempty, can be identified with the set of all such $\vu^+$ in \eqref{eq:cond_sqrt}, which is a torsor over $\Conn$. 
		\item 	  
			The set $\qyonek$ of all equivalence classes of all rank one graded quantization over $Y$, if nonempty, is a torsor over the group $\F_{\gm}(\stry^\times)$.
	\end{enumerate}
	
\end{theorem}  

\begin{remark}\label{rmk:grading}
	Theorem \ref{thm:quan_lag_graded} is only sensitive to the adjoint $\gm$-action on $\TT(L)$ induced by the $\gm$-action on $L$. The set of $\gm$-actions on $L$ (when nonempty) which induce the same $\gm$-action on $\TT(L)$ is a torsor over the group $\Hom(\gm, \gm)$ of characters of $\gm$. 
\end{remark}

\begin{remark}\label{rmk:quan_exist}
	Note that the graded quantization $\Q$ has period $\per(\Q) = \hb\pb[\sym] + [\sym_1]$ by Proposition \ref{prop:graded}. We can then deduce that the condition (b) in part (1) of Theorem \ref{thm:lag} is vacuous, while the condition (a) in part (1) of Theorem \ref{thm:lag} is implied by \eqref{eq:cond_sqrt}.  Therefore under the assumptions of Theorem \ref{thm:quan_lag_graded}, usual ungraded quantization of $L$ always exists. 
\end{remark}

\subsection{Graded quantization of vector bundles}\label{subsec:quan_lag_vec}

We now consider the general case in which the line bundle $L$ is replaced by a vector bundle $E$ of rank $r$ over $Y$. Just as in the case of line bundles, a \emph{quantization of} $E$, or the sheaf $\EE=(i_Y)_* E$ of $\str_X$-modules, is a $\Q$-module $\Ehb$, flat over $\pws$, with an isomorphism $\Ehb/\hb\Ehb \cong \EE$ of $\strx$-modules. We denote by $\qyr$ the set of all isomorphism classes of all quantization on $Y$ of all possible vector bundles $E$ of rank $r$. There is a generalization  \cite[Thm\,1.1]{BC} of Theorem \ref{thm:lag} for vector bundles.

\begin{theorem}\label{thm:lag_vec}
	\begin{enumerate}
		\item 
		A vector bundle $E$ of rank $r$ over a smooth Lagrangian subvariety $Y$ admits a quantization if and only if the following conditions hold:
		\begin{enumerate}
			\item \label{cond:lag_vec1}
			$i^*_Y \sym_i(\Q) = 0$ holds in $\htwo(Y)$, $\forall ~ i \geq 2$;
			\item  \label{cond:lag_vec2}
			the projectivization $\PE$ admits a flat algebraic connection;
			\item \label{cond:lag_vec3}
			$\frac{1}{r} c_1(E) = \frac{1}{2} c_1(K_Y) + \aty$ holds in $\omtry$.
		\end{enumerate}	
		\item 
		If nonempty, the set $\qyr$  admits a free action of the group $\F((\stry\series)^\times)$. The set of orbits for this action can be identified with the set of all isomorphism classes of $\PGLh$-bundles with a flat algebraic connection. 	  
	\end{enumerate}
\end{theorem}

Here $\PGLh = PGL(r, \pws)$ is the $\pws$-valued points of the algebraic group $\PGLr$. We denote by $\F(\PGLh)$ the set of all isomorphism classes of $\PGLh$-bundles with a flat algebraic connection and denote by 
  \begin{equation}\label{eq:tau}
    \tau: \qyr \epi \F(\PGLh)
  \end{equation} 
the quotient map by the $\F((\stry\series)^\times)$-action in part (2) of Theorem \ref{thm:lag_vec}.

The aim of this section is to promote Theorem \ref{thm:lag_vec} to the graded setting and generalize Theorem \ref{thm:quan_lag_graded} to the case of vector bundles. Assume we are in the situation of \S\,\ref{subsec:graded}. Suppose there is a $\gm$-equivariant structure on $E$. A quantization $\Ehb$ of $E$ is \emph{graded} if it satisfies exactly the same conditions in Definition \ref{defn:graded_line}. Again two graded quantizations are equivalent if there exists a $\gm$-equivariant isomorphism between them. Let $\qyrk$ denote the set of all equivalence classes of graded rank $r$ quantization of $Y$. 


We recall the discussions in \cite[\S 1]{BC}. Denote by $\TT(E)$ the Atiyah Lie algebroid of $E$ (\cite{Atiyah}), then there is an exact sequence of  Lie algebroids over $Y$,
\begin{equation}
0 \to \send_{\str_Y}(E) \to \TT(E) \to \T_Y \to 0, 
\end{equation}
where $\send_{\str_Y}(E)$ denotes the sheaf of $\str_Y$-algebras of $\str_Y$-linear endomorphisms of $E$. Suppose we have a usual (ungraded) quantization $\Ehb$ of $E$. Then the $\Q$-module structure of $\Ehb$, or rather $\Ehb/\hb^2\Ehb$ as a module over $\Q/\hb^2\Q$, gives rise to a morphism of Atiyah algebroids 
\begin{equation}\label{eq:vup}
\vu^+: \tatp \to \TT(E).
\end{equation}
In other words, the vector bundle $E$ is a module over the Lie algebroid $\tatp$.

The datum of the morphism $\vu^+: \tatp \to \TT(E)$ can be decomposed into two pieces. The trace morphism $\send_{\stry}(E) \to \stry$ induces a Lie algebroid morphism $\tr_E: \TT(E) \to \TT(\det(E))$ from $\TT(E)$ to the Picard algebroid of the determinant line bundle $\det(E)$ of $E$. Note that the map $\tr_E$ is scalar multiplication by $r=rk(E)$ on the subbundle $\stry$ and identity on the quotient bundle $\T_Y$. Therefore the composite map 
\begin{equation}\label{eq:detvup}
\det\vu^+ := \tr_E \circ \vu^+: \tatp \to \TT(\det(E))
\end{equation}
is an isomorphism of Lie algebroids which has the same property as $\tr_E$ when restricted to $\stry \subset \tatp$. Such an isomorphism is equivalent to an isomorphism 
  \begin{equation} \label{eq:detvupr}
  	\frac{1}{r}\det\vu^+: \tatp \isom \frac{1}{r} \TT(\det(E)) 
  \end{equation}
of Picard algebroids (which restricts to the identity map on $\stry$). Thus $\det\vu^+$ gives
\[  c_1(E) = c_1(\det(E)) = r \cdot c(\tatp) = r \cdot \aty \in \omtry, \]
which is condition \eqref{cond:lag_vec3} in Theorem \ref{thm:lag_vec}. Furthermore, the Atiyah algebroid $\TT(\PE)$ of the  associated $\PGLr$-torsor $\P_{\PE}$ of $\PE$ can be obtained by taking the quotient of $\TT(E)$ by the subbundle $\stry \subset \send_{\stry}(E)$. We hence obtain  a natural embedding $\TT(E) \inj \TT(\det(E)) \oplus \TT(\PE)$ of sheaves of Lie algebras. Composing $\vu^+$ with the quotient map $\TT(E) \epi \TT(\PE)$ provides a morphism $\T_Y \to \TT(\PE)$ of Lie algebroids lifting the identity map on $\T_Y$, which is exactly a flat algebraic connection on $\PE$. 

Conversely, the morphism $\vu^+$ can be recovered from a morphism $\det(\vu^+): \tatp \to \TT(\det(E))$ satisfying the condition above and a flat algebraic connection on $\PE$. Namely, they can be combined to give a commutative diagram
\begin{equation}\label{diag:tatp}
\begin{tikzcd}
0  \arrow{r}  & \stry \arrow{r} \arrow[d]  & \tatp \arrow[r] \arrow[d]   & \T_Y  \arrow{r} \arrow[d] & 0  \\
0  \arrow{r}  & \stry \oplus \send^0(E) \arrow{r}  & \TT(\det(E)) \oplus \TT(\PE) \arrow{r}   &  \T_Y \oplus \T_Y  \arrow{r} & 0,  
\end{tikzcd}
\end{equation}
where $ \send^0(E) := \send_{\stry}(E)/\stry$. The leftmost vertical map $\stry \to \stry \oplus \send^0(E)$ is given by $f \mapsto (r \cdot f, 0)$. The rightmost vertical map $\T_Y \to \T_Y \oplus \T_Y$ is the diagonal map $\partial \mapsto (\partial,\partial)$. The vertical map in the middle is the sum of the map $\det(\vu^+): \tatp \to \TT(\det(E))$ and the composition of the projection $\tatp \to \T_Y$ with the splitting $\T_Y \to \TT(\PE)$ given by the flat connection. It was argued in \cite[\S 1]{BC} that the image of the middle vertical map is exactly $\TT(E) \subset \TT(\det(E)) \oplus \TT(\PE)$.

Now suppose we are in the graded situation and $\Ehb$ is a graded quantization of $E$. Then there are natural $\gm$-actions on $\tatp$ and $\TT(E)$. Moreover the morphisms $\vu^+$ and $\det\vu^+$ are $\gm$-equivariant. By the discussions above, this is equivalent to the condition that both the isomorphism $\frac{1}{r}\det\vu^+$ is $\gm$-equivariant and the flat algebraic connection on $\PE$ are $\gm$-invariant. Furthermore, the flat connection lifts the infinitesimal Lie algebra action of $\lie(\gm)=\kk$ on $Y$ to a Lie algebra $\kk$-action on $\PE$. On the other hand, the $\gm$-action on $\PE$ differentiates to another Lie algebra $\kk$-action on $\PE$. To compare them, we make the following definition.

\begin{definition}
	If the Lie algebra $\kk$-action defined via the flat connection of $\PE$ as above coincides with the Lie algebra $\kk$-action induced by the graded structure of $\PE$, we say that the $\gm$-equivariant structure or the flat connection on $\PE$ is \emph{strong}. 
\end{definition}

The main result of this section is the following theorem, which basically says that $\Ehb$ can be reconstructed from $\vu^+$, or equivalently, $\frac{1}{r} \det\vu^+$ and the flat algebraic connection on $\PE$, assuming it is strong. The proof will be postponed to \S\,\ref{subsec:quan_lag_vec_proof}.

\begin{theorem}\label{thm:lag_vec_graded}
	Let $(X,\sym)$ be a graded smooth symplectic variety with a graded quantization $\Q$ and let $Y$ be a graded Lagrangian subvariety of $X$. Suppose $E$ is a graded vector bundle or rank $k$ over $Y$ such that $\PE$ is equipped with a strong flat connection. Then 
	\begin{enumerate}
		\item 
		  There is a unique $\gm$-equivariant flat $\PGLh$-torsor $\Ppgl$ which lifts $\PE$ with its strong flat connection.
		\item 
		   The set of (isomorphism classes of) graded quantizations of $E$ in the preimage  of $[\Ppgl] \in \F(\PGLh)$ under the projection $\tau$ \eqref{eq:tau} is nonempty if and only if there exists a $\gm$-equivariant isomorphism $\tatp \to \frac{1}{r} \TT(\det(E))$ of graded Picard algebroids. In this case, this set is a torsor over $\Conn$ by restriction of the action of 
		   $\F_{\gm}(\stry^\times)$. Moreover, it is isomorphic to the $\Conn$-torsor of all $\gm$-equivariant isomorphisms $\tatp \to \frac{1}{r} \TT(\det(E))$ of graded Picard algebroids.
	\end{enumerate}

\end{theorem}

\section{Formal geometry and quantization}\label{sec:formal}

In \S\,\ref{subsec:hcpair} and \S\,\ref{subsec:hctorsor}, we first recall basics of formal geometry, which are taken from \cite{BK, BGKP, BC}. Then in the remaining parts we provide the proofs of main results from \S\,\ref{sec:quan_lag}.

\subsection{Harish-Chandra pairs}\label{subsec:hcpair}

In this section, we will recall the notion of $\emph{Harish-Chandra pair}$ and important examples related to quantization problems. Note that the same notion is also used in representation theory of finite dimensional reductive groups, but the examples considered in this section are mostly infinite dimensional. For applications in graded quantizations, we also introduce the notion of \emph{graded Harish-Chandra pair}.

\begin{definition}
	A \emph{Harish-Chandra pair (HC)} (over $\kk$) is a pair $\gh$, where $G$ is a (pro)algebraic group and $\lieh$ is a (pro)finite dimensional Lie algebra equipped with a $G$-action and with a $G$-equivariant embedding $\lieg \inj \lieh$ of Lie algebras such that the differential of the $G$-action on $\lieh$ coincides with the adjoint action of $\lieg$ on $\lieh$ via the embedding $\lieg \inj \lieh$. A \emph{module $V$ over the Harish-Chandra pair $\gh$} is a representation $V$ of the Lie algebra $\lieh$ whose restriction to $\lieg \subset \lieh$ integrates into an algebraic representation of $G$. Moreover, if both $G$ and $\lieh$ are equipped with continuous $\gm$-actions such that the embedding $\lieg \inj \lieh$ is $\gm$-equivariant, then we say that $\gh$ is a \emph{graded Harish-Chandra pair}. If a module $V$ over $\gh$ admits a compatible $\gm$-action, then we say that $V$ is a \emph{graded module over $\gh$}.
\end{definition}

We consider the following important examples of HC pairs. Fix $n \geq 1$ and a $2n$-dimensional vector space $\vv$ over $\kk$ equipped with a linear symplectic form $\sym_0 \in \wedge^2 \vv^*$. We choose a system of linear coordinate functions $x_1, \ldots, x_n, y_1, \ldots, y_n$ on $\vv$ such that
\[ \{x_i, x_j\} = 0 = \{y_i, y_j\}, \quad \{y_j,x_i\} = \delta_{ij}, \quad \forall~ i,j=1,\ldots,n, \]
where $\{\cdot,\cdot\}$ is the Poisson bracket associated to $\sym$. 

Let $\D=\kk \lb x_1, \ldots ,x_n, y_1, \ldots, y_n, \hb \rb$ be the completed homogeneous Weyl agebra over $\pws = \kk\series$ subject to the commutation relations
\begin{equation} \label{eq:CCR}
[x_i,x_j] = 0 = [y_i, y_j], \quad [y_j,x_i] = \delta_{ij} \hb, \quad [\hb,x_i]=[\hb,y_j]=0, \quad \forall~ i,j.
\end{equation}
Then $\A := \D/\hb\D \cong \kk \lb x_1, \ldots ,x_n, y_1, \ldots, y_n \rb$ is the (topological) algebra of functions on the formal completion $\widehat{\vv}$ of $\vv$ at the origin equipped with the Poisson bracket $\{\cdot,\cdot\}$. We set
\begin{equation}\label{eq:grading_D}
\deg x_i = \deg y_j = 1, \quad \deg \hb = 2,
\end{equation} 
so that $\A$ and $\D$ are graded algebras (or equivalently have $\gm$-actions) and all canonical homomorphisms, structures and constructions below are compatible with this grading. We denote by $\D^i$ the $i$-th graded component of $\D$. Let $\autd$ be the group of $\pws$-linear automorphisms of the algebra $\D$ and let $\der(\D)$ be the Lie algebra of $\pws$-linear derivations of $\D$. Then they form a graded HC pair $\autder$. They also have the natural $\gm$-actions by conjugations induced by the $\gm$-action on $\D$. 

Let $\xx$ be the Lagrangian subspace of $\vv$ spanned by $\{ y_1,\ldots,y_n \}$. Then $\M := \D/\D\xx$ is a left $\D$-module such that $\M/\hb\M \cong \A/\A\xx$ as a $\A$-module. This is the `formal local model' of quantization of line bundles. Let $1_\M := 1 \mod \D\xx$ denote the generator of $\M$. The composite map 
\[ \kk \llbracket x_1, \dots, x_n, \hb \rrbracket \inj \D \epi \D/\D\xx = \M   \]
is an isomorphism of $\kk \llbracket x_1, \dots, x_n, \hb \rrbracket$-modules. Under this isomorphism, the action of $y_i$ on $\M$ goes to the operator $\hb\partial_i = \hb\partial_{x_i}$ on $\kk \llbracket x_1, \dots, x_n, \hb \rrbracket$ by \eqref{eq:CCR}.

In order to obtain a graded version of Theorem \ref{thm:lag}, we need to introduce a grading, or equivalently a (continuous) $\gm$-action $(t,m) \mapsto t.m$, on $\M$ compatible with the grading \eqref{eq:grading_D} of $\D$, in the sense that   
\begin{equation}\label{eq:grading_M}
t.\hb = t^2\hb, \quad t.(um) = (t.u)(t.m), \quad  \forall~ t \in \kkm,~u \in \D,~m \in \M.
\end{equation}   
It is clear that such a $\gm$-action is uniquely determined by its value on the generator $1_\M$. We can define the $\gm$-action by regarding $\M$ as a quotient of $\D$ such that $t.1_\M = 1_\M$, $\forall~t \in \kkm$. We denote this action by $s^0 : \gm\times\M \to \M, (t, m) \mapsto s^0_t(m)$. 

\begin{remark}\eqref{rmk:grading}
	The $\gm$-action satisfying \eqref{eq:grading_M} is not unique. For instance, any character $t \mapsto t^n$ of $\gm=\kkm$ gives rise to a $\gm$-action $s^n: \gm\times\M \to \M$ on $\M$ such that
	\[  s^n_t(m) = t^n \cdot s^0_t(m).  \]
	Up to automorphisms of $\M$ as $\D$-modules, the collection $\{ s^n \}_{n \in \ZZ}$ give all equivalence classes of $\gm$-actions on $\M$ satisfying \eqref{eq:grading_M}. The choice of the grading on $\M$ will not affect the discussions afterwards. This is compatible with the fact that the grading on admissible bundles on nilpotent coadjoint orbits is only unique up to degree shift (see \S\,\ref{subsec:grading}). So we will stick with the $\gm$-action $s^0$. 
\end{remark}

Let $\JJ \subset \D$ be the preimage of the ideal $\A\xx \subset \A$ under the projection $\D \epi \D/\hb\D = \A$. Then $\JJ$ is a two-sided ideal of $\D$. Let $\autdj$, resp. $\derdj$, be the subset of $\autd$, resp. $\der(\D)$, consisting of maps preserving $\JJ$. The pair $\autderj$ is a HC subpair of $\autder$, which is preserved by the $\gm$-action and hence graded. 

Now we define another HC pair which is directly related to quantization of Lagrangian subvarieties. Let $\derdm$ be the Lie algebra of derivations of the pair $(\D,\M)$, which are of the form $(\dd, \dm)$ where $\dd \in \der(\D)$ and $\dm: \M \to \M$ is a continuous $\pws$-linear map such that
\begin{equation}\label{eq:dm}
\dm(u \cdot m) = \dd(u) m + u (\dm(m)), \quad \forall~ u \in \D, ~ m \in \M.
\end{equation}
Similarly, we define $\autdm$ to be the group of automorphisms of the pair $(\D,\M)$, which are of the form $(\fd,\fm)$ where $\fd \in \autd$ and $\fm: \M \to \M$ is a continuous $\pws$-linear bijection such that
\begin{equation}
\fm(u \cdot m) = \fd(u)\fm(m).
\end{equation}
Then $\autderdm$ forms a HC pair and $\M$ is a module over this pair. If we choose the grading on $\M$ described after \ref{eq:grading_M}, we also get a $\gm$-action on $\autdm$ and $\derdm$ by conjugation. Note that the collection of $\gm$-actions $s^n$ on $\M$ in Remark \ref{rmk:grading}  all lead to the same $\gm$-action on the pair $\autderdm$. Thus the pair $\autderdm$ is \emph{canonically} a graded HC pair and $\M$ is a graded module over it.

Let $\m$ be the preimage of the unique maximal ideal of $\A$ under the projection $\D \epi \A$. Then $\m$ is the unique maximal ideal of $\D$ and there is a canonical $\gm$-equivariant isomorphism $\M / \m \M \cong \kk$, where the $\gm$-action on $\kk$ is trivial. For any $(\fd,\fm) \in \autderdm$, since $\fd \in \autd$ always preserves $\m$, the bijection $\fm: \M \to \M$ descends to a $\kk$-linear automorphism of $\M/\m \M \cong \kk$, which is the same as an element in $\kkm$. This gives a canonically defined $\gm$-equivariant group homomorphism
\begin{equation}\label{eq:kappa}
	\kappa: \autdm \to \kkm,
\end{equation}
where the $\gm$-action on the codomain is trivial. The natural embedding $\kkm \inj \autdm$ is a section of this projection. 

The map $f \mapsto (\Id_\D, f \cdot \Id_\M)$ provides a natural central embedding $ \epaut: \pws^\times \inj \autdm$. Similarly there is a central Lie algebra embdding $\epder: \pws \inj \derdm$ given by $a \mapsto (0, a \cdot \Id_\M)$. This gives an injective morphism 
\[ \ep = \lab \epaut, \epder \rab: \hcK \inj \autderdm \] 
of graded HC pairs. On the other hand, the obvious forgetful map gives a morphism 
\[ F : \autderdm \to \autder \] 
of graded HC pairs. It can be shown that the image of $F$ is exactly $\autderj$ and we have an exact sequence of \emph{graded} HC pairs(\cite[Cor. 3.8] {BGKP}),
\begin{equation}\label{exsq:autderdm}
1 \to \hcK \xrightarrow{\ep} \autderdm \xrightarrow{F} \autderj \to 1.
\end{equation}

Similarly for the pair $\autderdm$, we define the HC pair
\[ \bautderdm := \lab \autdm/\epaut(\kkm), \derdm/\epder(\kk)  \rab.  \]
Because of $\kappa: \autdm \to \kkm$ in \eqref{eq:kappa}, we have a decomposition of groups
  \begin{equation}\label{eq:autdm_split}
  	\autdm \cong \kkm\times\bautdm.
  \end{equation}

Consider the algebra $\D[\hb\pb]$ with the $\gm$-action extended from that on $\D$, such that $t.h^i = t^i h^i$, $\forall~t \in \gm$, $i \in \ZZ$. The subspace $\hinvj \subset \D[\hb\pb]$ is naturally a graded Lie algebra under commutator bracket. There is a graded Lie algebra homomorphism $\varphi_\D: \hinvj \to \der(\D)$ given by $a \mapsto [a, -]$. Note that $\JJ\M \subset \hb\M$ since $\JJ=\ann(\M/\hb\M)$. Therefore we have a well-defined map $\pdm: \M \to \M$ such that $\pdm(a)$ is given by $m \mapsto am$ for any $a \in \hinvj$ and $m \in \M$. It is straightforward to check that both $\pd(a)$ and $\pdm(a)$ satisfy the equation \eqref{eq:dm} and hence the pair $(\pd(a), \pdm(a))$ gives an element of $\derdm$. Therefore the assignment $a \mapsto (\pdm(a), \pd(a)) $ yields a \emph{graded} Lie algebra homomorphism $\pddm: \hinvj \to \derdm$. In fact, it is an isomorphism by \cite[Lemma 3.6]{BGKP}.


For the high rank case, consider $\D$-module $\M_r:=(\M)^{\oplus r}$ to be the direct sum of $r$-copies of $\M$, in particular $\M_1 = \M$. The natural grading on $\M$ induces a grading on $\M_r$ so that $\M_r$ is a graded $\D$-module. Similar to the rank one case, one can consider the group $\autdmr$ of automorphisms and the Lie algebra $\derdmr$ of derivations of the pair $(\D,\M_r)$, which together form a HC pair $\autderdmr$. Denote by $\GLh = GL(r,\pws)$ and $\PGLh=PGL(r,\pws)$ the $\pws$-valued points of the group $\GLr$ and $\PGLr$ respectively and by $\glh$ and $\pglh$ their Lie algebras.  We then have a short exact sequence of HC pairs (\cite[Lemma\,4.4]{BC}),
\begin{equation}\label{exsq:autderdmr}
1 \to \glgl \to \autderdmr \to \autderj \to 1,
\end{equation}
which is the sequence \eqref{exsq:autderdm} when $r=1$. Like the rank one case, the graded structure on $\M_r$ induces a graded structure on $\autderdmr$, so that maps in \eqref{exsq:autderdmr} respect the graded structures. Here the $\gm$-action on $\glgl$ is induced by the $\gm$-action on $\pws$.

There is a commutative diagram of graded Lie algebras
\begin{equation}\label{diag:derdmr}
\begin{tikzcd}
0  \arrow{r}  & \pws \arrow{r} \arrow[d]  & \hb\pb\JJ \arrow[r] \arrow[d]   & \derdj  \arrow{r} \arrow[d, equal] & 0  \\
0  \arrow{r}  & \glh    \arrow{r}  & \derdmr  \arrow{r}   &  \derdj  \arrow{r} & 0  
\end{tikzcd}
\end{equation}
where the middle vertical map is defined in exactly the same way as the isomorphism $\pddm: \hb\pb\JJ \isom \derdm$. There is a natural splitting $\glh \cong \pws \oplus \pglh$ of graded Lie algebras given by the trace map $\tr: \glh \to \pws$. Combining this splitting with diagram \eqref{diag:derdmr} give a splitting of graded Lie algebras
\begin{equation}\label{eq:derdmr}
\derdmr \cong [\hb\pb\JJ \oplus \glh]/\pws \cong  \hb\pb\JJ \oplus \pglh.
\end{equation}
At the level of groups there is a corresponding diagram
\begin{equation}\label{diag:autdmr}
\begin{tikzcd}
1  \arrow{r}  & \pws^\times \arrow{r} \arrow[d]  & \aut(\D,\M) \arrow[r] \arrow[d]   & \derdj  \arrow{r} \arrow[d, equal] & 1  \\
1  \arrow{r}  & \GLh    \arrow{r}  & \autdmr  \arrow{r}   &  \derdj  \arrow{r} & 1  
\end{tikzcd}
\end{equation}
but the splitting at the Lie algebra level does not pass to groups in constant terms (with respect to the order of $\hb$), since the group $\GLr$ does not split into $\PGLr\times\kkm$. On the other hand, we have the splitting $\autdm \cong \kkm\times\bautdm$ by \eqref{eq:autdm_split}. Therefore we have a splitting of graded groups
\begin{equation}\label{eq:autdmr}
\autdmr \cong  \bautdm\times[\GLr \ltimes \exp(\hb \cdot \pglh)],
\end{equation}
where $\exp(\hb \cdot \pglh)$ is the prounipotent kernel of the evaluation map $\PGLh \to \PGLr$ modulo $\hb$. The two splittings \eqref{eq:derdmr} and \eqref{eq:autdmr} agree modulo $\lab \kkm, \kk \rab$ (see \cite[Lemma 6]{BC}),
\begin{equation}\label{eq:autderdmr_quot}
\autderdmr / \lab \kkm, \kk \rab \cong \bautderdm\times\pglpgl.
\end{equation}

\subsection{Harish-Chandra torsors} \label{subsec:hctorsor}

For the remaining part of this section, we sketch basic definitions and constructions related to the proof of Theorem \ref{thm:per}, Proposition \ref{prop:graded} and Corollary \ref{cor:per_graded}. An essential notion in the approach of \cite{BK} and \cite{BGKP} is that of a \emph{Harish-Chandra (HC) torsor} over a Harish-Chandra pair $\gh$, which is defined as follows. Suppose $\P$ is a $G$-torsor over a variety $X$ with the structure map $\pi:\P \to X$. Let $\lieg_\P:=\pi_*(\str_\P \hten \lieg)^G$ and $\lieh_\P := \pi_*(\str_\P \hten \lieh)^G$ be the Lie algebra bundles on $X$ associated to the $G$-modules $\lieg$ and $\lieh$ respectively. Let $\EE_\P:=\pi_*(\T_\P)^G$ be the sheaf of $G$-invariant vector fields on $\P$. Then we have the short exact sequence (the \emph{Atiyah extension}, cf. \cite{Atiyah})
\begin{equation} \label{eq:atiyah}
	 0 \to \lieg_\P \xrightarrow{\iota_\P} \EE_\P \to \T_X \to 0
\end{equation}
of bundles over $X$. An \emph{$\lieh$-valued connection} on $\P$ is a $G$-equivariant bundle map $\theta_\P: \EE_\P \to \lieh_\P$ such that the composition $\theta_\P \circ \iota_\P: \lieg_\P \to \lieh_\P$ coincides with the embedding $\lieg_\P \inj \lieh_\P$ induced by $\lieg \inj \lieh$. An $\lieh$-valued connection is said to be \emph{flat} if the $\lieh$-valued $1$-form $\Theta:=\pi^*(\theta_\P)$ on $\P$ satisfies the Maurer-Cartan equation $d\Theta + \frac{1}{2}\Theta \wedge \Theta = 0$. 

\begin{definition}
	A \emph{Harish-Chandra (HC) $\gh$-torsor} over $X$ is a pair $\lab \P, \theta_\P \rab$ of a $G$-torsor $\P$ over $X$ and a flat $\lieh$-valued connection $\theta_\P: \EE_\P \to \lieh_\P$ on $\P$. It is said to be \emph{transitive} if the connection map $\theta_\P : \EE_\P \to \lieh_\P$ is an isomorphism. In this case, $\theta_\P$ is equivalent to a Lie algebra morphism $\lieh \to \Gamma(\P, \T_\P)$ which induces an isomorphism of vector bundles $\lieh \otimes \str_\P \cong \T_\P$. 
\end{definition}


Given a module $V$ over a HC pair $\gh$ and a Harish-chandra $\gh$-torsor $\P$ over $X$, the \emph{localization  $\loc(\P,V)$ of $V$ with respect to $\P$} is defined to be the associated bundle $\P \times_G V$ of the $G$-torsor $\P$ with fiber $V$, which carries a natural flat connection induced by $\theta_\P$ (see \cite[\S\,2.3]{BK} for details).

Suppose $\str_X$ is a quantization of the smooth symplectic variety $(X,\sym)$. Recall that $\IY$ is the ideal subsheaf of $Y$ and $\JY$ is the preimage of $\IY$ under the projection $\Q \epi \str_X$. For any $y \in Y$,  let $\hstr_{\hb,y}$ denote the completion of $\Q$ at $y$. Correspondingly, let $\hJJ_y \subset \hstr_{\hb,y}$ be the completion of the ideal sheaf $\JY$  at $y$. Set $\PJ$ to be the set of pairs $(y, \eta)$ where $\eta: \hstr_{\hb,y} \isom \D$ is a continuous $\pws$-algebra isomorphism such that $\eta(\hJJ_y) = \JJ$. Then $\PJ$ is a torsor over the HC pair $\autderj$. Moreover, $\PJ$ is a transitive torsor since the conormal bundle to $Y$ is naturally identified with its tangent bundle due to the Lagrangian condition (cf. \cite[Lemma 4.3]{BC}). 


By \cite[Lemma 6.1]{BGKP}, a choice of line bundle $L$ on $Y$ and its quantization $\Lhb$ is equivalent to the choice of a lift of the torsor $\PJ$ to a transitive HC $\autderdm$-torsor $\Pdm$. Given a line bundle $L$ and its quantization $\Lhb$, set $\LL_{\hb,y} := \hstr_{\hb,y} \otimes_{\Q} \Lhb$. Then it is shown in the proof of \cite[Lemma 6.1]{BGKP} that the $\D$-module $\eta^*\LL_{\hb,y}$, obtained from $\LL_{\hb,y}$ by transporting the module structure via $\eta$, is isomorphic to $\M$. Various choices of an isomorphism $\eta^*\LL_{\hb,y} \cong \M$ for all $y \in Y$ give a transitive HC torsor $\Pdm$ over $\autderdm$, which is the required lift of the torsor $\PJ$. Conversely, the quantization $\Lhb$ can be recovered as the (Zariski) sheaf of flat sections of $\loc(\Pdm, \M)$. 

\begin{remark}
	Note that the nonemptyness of the torsors $\PJ$ and $\Pdm$ is not tautological. The case of vector bundles in \S\,5.4 is similar. The proof was provided in \cite[Lemma 4.2]{BC}.
\end{remark}

To promote the previous discussions into the graded setting, we need the notion of \emph{graded HC torsors}.

\begin{definition} \label{defn:graded_HC_pair}
	A \emph{graded Harish-Chandra (HC) torsor} over a graded Harish-Chandra pair $\gh$ is a usual Harish-Chandra $\gh$-torsor $\P$ equipped with a $\gm$-action which lifts the $\gm$-action on $X$, such that 
	\begin{enumerate}
		\item it is compatible with the $G$-action and the $\gm$-action on $G$, and
		\item 
		  the connection map $\theta_\P : \EE_\P \to \lieh_\P$ is $\gm$-equivariant with respect to the induced $\gm$-actions on $\EE_\P$ and $\lieh_\P$.
	\end{enumerate}
 We say that a graded Harish-Chandra torsor is \emph{transitive} if it is transitive as a usual Harish-Chandra torsor.
\end{definition}

\begin{remark} \label{rmk:graded_HC_pair}
	Equivalently, a graded structure on a $\gh$-torsor $\P$ can be described as follows. We identify $\P$ with its sheaf of local sections. For any $t \in \kkm$, one can define a new $\gh$-torsor $\tPP$ to be the sheaf $t_*(\P)$ on $Y$, where $t_*$ stands for the sheaf-theoretic push-forward with respect to the automorphism of $X$ induced by $t$. The $\gh$-torsor structure of $\tPP$ is defined by twisting the $\gh$-torsor structure of $\P$ via the $\gm$-action on $G$:  $gs := (t.g)s$, where $g \in G$, $s$ is any local section of $\tPP$, $t.g$ stands for the $\gm$-action on $G$ and the product on the right hand side is taken in $\P$. Note that $\EE_{\tPP}$ and $\lieh_{\tPP}$ are canonically isomorphic to $t_*(\EE_\P)$ and $t_*(\lieh_\P)$ respectively. Therefore we can define the connection $\theta_{\tPP}:= t_*(\theta_\P)$ for $\tPP$ and $\tPP$ is a HC $\gh$-torsor. A graded structure of $\P$ is encoded as a collection of $\gh$-torsor isomorphisms $\phi_t: \P \to \tPP$ for all $t \in \kkm$, satisfying the cocycle condition $\phi_{t_1 t_2} = (t_1)_* (\phi_{t_2}) \phi_{t_1}$, $\forall t_1, t_2 \in \kkm$.
\end{remark}

Now suppose $\Lhb$ is graded, we can define a $\gm$-action on the set of isomorphisms $\eta^*\LL_{\hb,y} \cong \M$  by conjugation of the $\gm$-action on $\Lhb$ and the $\gm$-action on $\M$ by the equation \eqref{eq:grading_M}. Therefore $\Pdm$ has a $\gm$-equivariant structure which is compatible with the $\gm$-equivariant structure on $\PJ$ via the natural projection $\Pdm \epi \PJ$. Conversely, if a $\gm$-equivariant lift $\Pdm \epi \PJ$, then $\loc(\Pdm, \M)$ is graded and so is its sheaf of flat sections. 

\begin{lemma}\label{lemma:torsor_lag_graded}
	Let $(X,\sym)$ be a graded smooth symplectic variety with a graded quantization $\Q$ and let $Y$ be a graded Lagrangian subvariety of $X$. Then a graded rank one quantization $\Lhb$ of $Y$ is equivalent to the choice of a lift of the $\gm$-equivariant torsor $\PJ$ to a $\gm$-equivariant transitive Harish-Chandra $\autderdm$-torsor $\Pdm$ such that the natural projection $\Pdm \epi \PJ$ is $\gm$-equivariant. 
\end{lemma}

Similar to Lemma \ref{lemma:torsor_lag_graded}, we have the following graded version of \cite[Prop.\,4.4]{BC}.

\begin{lemma}\label{lemma:torsor_lag_vec}
	A choice of a graded vector bundle $E$ of rank $r$ on $Y$ and its graded quantization $\Ehb$ is equivalent to a choice of a lift of the $\gm$-equivariant Harish-Chandra $\autderj$-torsor $\PJ$ to a $\gm$-equivariant transitive Harish-Chandra $\autderdmr$-torsor $\Pmr$ such that the natural projection $\Pmr \epi \PJ$ is $\gm$-equivariant. 
\end{lemma}

\subsection{Graded quantization of line bundles: proof of Theorem \ref{thm:quan_lag_graded}}\label{subsec:graded_lag_proof}

By Lemma \ref{lemma:torsor_lag_graded}, a choice of graded quantization $\Lhb$ of a graded line bundle $L$ can be divided into a two-step choice of:

\begin{itemize}
	\item 
	A graded lift of the $\autderj$-torsor $\PJ$ to a graded transitive torsor $\bPdm$ over $\bautderdm$.
	\item 
	A graded lift of $\bPdm$ to a graded transitive $\autderdm$-torsor $\Pdm$ such that there is a graded isomorphism $L \cong \kk \otimes_\kappa \Pdm$, where the right hand side is the associated line bundle of $\Pdm$ via the group homomorphism $\kappa: \autdm \to \kkm$ as defined in \eqref{eq:kappa}.
\end{itemize}

The following lemma describes when the second part of the choice exists. The proof is exactly the same as that of \cite[Lemma 6.5]{BGKP}. One only needs to check that every step of the argument is compatible with the graded structures.

\begin{lemma}\label{lemma:bPdm_lift}
	The set of isomorphism classes of graded lifting $\Pdm$ of the Harish-Chandra $\bautderdm$-torsor $\bPdm$ as above is in bijection to the set of $\gm$-equivariant isomorphisms $\vu^+: \tatp \to \TT(L)$ of graded Picard algebroids. Moreover, the bijection is an isomorphism of $\Conn$-torsors.
\end{lemma}

The only thing we need to explain is the torsor structure on the former set. It is given by $\Pdm \mapsto (\Pdm \times_Y \stry^\times) /\kkm$, where $\stry^\times$ is equipped with any $\gm$-invariant flat connection.

Next we show that the first part of the choice above also exists and is unique. This will complete the proof of Theorem \ref{thm:quan_lag_graded}.

\begin{proposition}\label{prop:bPdm_lift}
	There is a unique (up to isomorphism) graded transitive $\bautderdm$-torsor $\bPdm$ which is a graded lift of the $\autderj$-torsor $\PJ$ such that the natural projection $\Pdm \epi \PJ$ is $\gm$-equivariant.
\end{proposition}

Let $\Lambda$ denote the set of isomorphism classes of all (ungraded) transitive $\bautderdm$-torsors $\bPdm$ equipped with a $\autderj$-torsor isomorphism $\bPdm \otimes_ {\bautdm} \autdj \cong \PJ$, i.e., $\bPdm$ is a lift of $\PJ$ to a $\bautderdm$-torsor. By Remark \ref{rmk:quan_exist} as well as \cite[Lemma 6.4]{BGKP}, the set $\Lambda$ is nonempty. We denote by $\Pi$ the group of isomorphism classes of torsors over $1 + \hb \str_Y \series \cong (\str_Y \series)^\times / \str_Y^\times$ with a flat algebraic connection (but the flat connection is not part of the data). Then by \cite[\S\,6.1]{BGKP}, the set $\Lambda$ is a $\Pi$-torsor such that the group action is given by $\bPdm \mapsto (\bPdm \times_Y Q) / \pws^\times$, where $\bPdm$ is any lift of $\PJ$ and $Q$ is any flat $(1 + \hb \str_Y \series)$-torsor.

Suppose that the graded structure of $\PJ$ is encoded as a collection of $\autderj$-torsor isomorphisms $\phi_t: \PJ \to \tPJ$ for all $t \in \kkm$, satisfying the cocycle condition in Remark \ref{rmk:graded_HC_pair}. There is a natural $\gm$-action on $\Lambda$ defined as follows. Pick any transitive $\bautderdm$-torsor $\bPdm$ which is a lift of $\PJ$. For any $t \in \kkm$, one can define a new transitive $\bautderdm$-torsor $\tP$ just as in Remark \ref{rmk:graded_HC_pair}. It is naturally a lift of $\tPJ$ and hence also a lift of $\PJ$ via the isomorphism $\phi_t: \PJ \to \tPJ$. Therefore $[\tP]$ is a well-defined element in $\Lambda$ and we define the $\gm$-action on $\Lambda$ by $t.[\bPdm] := [\tP]$. In exactly the same manner, we can define a $\gm$-action on the group $\Pi$ by group automorphisms, which makes $\Lambda$ a graded torsor over the graded group $\Pi$.

It is clear that the isomorphism class of a graded transitive $\bautderdm$-torsor determines a $\gm$-stable point in $\Lambda$. The converse is also true. To prove this, consider the group $\autpdm$ of all $\bautderdm$-torsor automorphisms of $\bPdm$ which intertwines the projection $\bPdm \epi \PJ$. The following lemma is immediate from the definition of the HC torsors in question.

\begin{lemma}\label{lemma:autpdm}
	There is a canonical $\gm$-action on $\autpdm$ by group automorphisms induced by the $\gm$-action on the group $\autdm$. Moreover, we have a canonical $\gm$-equivariant group isomorphism
	\begin{equation}\label{eq:autpdm}
	\autpdm \cong \hdr^0(Y, 1 + \hb \pws),
	\end{equation} 
	where $\hdr^0(Y, 1 + \hb \pws)$ is regarded as a multiplicative abelian group equipped with the  $\gm$-action induced by the isomorphism $1 + \hb \pws \cong \pws^\times / \kkm$. In particular, the only $\gm$-equivariant automorphism of $\bPdm$ in $\autpdm$ is the identity.
\end{lemma}

\begin{proposition}\label{prop:kstable}
	A lift $\bPdm$ of $\PJ$ to a transitive $\bautderdm$-torsor admits a graded structure such that the natural projection $\bPdm \epi \PJ$ is $\gm$-equivariant, if and only if the corresponding element $[\bPdm]$ in $\Lambda$ is $\gm$-stable. In this case, the graded structure is unique up to automorphism.
\end{proposition}

\begin{proof}
	The ``only if'' part is obvious. To prove the ``if'' part, we only need to construct a $\gm$-action on $\bPdm$ when $[\bPdm]$ is $\gm$-stable in $\Lambda$. Since $[\tP]=t.[\bPdm]=[\bPdm]$, we can pick a $\bautderdm$-torsor isomorphism $\psi_t: \bPdm \to \tP$ for each $t \in \kkm$, which intertwines the projections $\bPdm \epi \PJ$ and $\tP \epi \PJ$. The isomorphisms $\psi_t$ may not satisfy the cocycle condition in Remark \ref{rmk:graded_HC_pair}. We can define a map $a: \kkm\times\kkm \to \autpdm$ by 
	\[ \psi_{t_1 t_2} = (t_1)_* (\psi_{t_2}) \psi_{t_1} a(t_1,t_2), \quad \forall t_1, t_2 \in \kkm. \] 
	Since $\autpdm$ is abelian, one can check that $a$ is a $2$-cocycle, i.e.,
	\[  [(t_1).a(t_2, t_3) ] a(t_1, t_2 t_3) = a(t_1t_2,t_3) a(t_1,t_2), \]
	where $(t_1).a(t_2, t_3)$ stands for the $\gm$-action on $\autpdm$. Since the $\gm$-action $\autpdm$ is rational, the $2$-cocycle $a$ is a coboundary and we can therefore modify $\psi_t$ so that it satisfies the cocycle condition. This means that $\bPdm$ is graded. A similar argument shows that the graded structure is unique up to automorphism of $\bPdm$.	
\end{proof}

To complete the proof of Proposition \ref{prop:bPdm_lift}, it remains to prove the following:

\begin{proposition}\label{prop:stable}
	There exists a unique $\gm$-stable element in $\Lambda$.
\end{proposition}

By \cite[Prop. 2.7]{BK} and \cite[\S\,6.1]{BGKP}, the set $\Lambda$ is a torsor over the (multiplicative) abelian group 
  \[ \Pi \cong \hone(Y, 1 + \hb \pws) \cong  \hone(Y, \pws^\times / \kkm ).\] 
Before proving Proposition \ref{prop:stable}, we need the following result on graded torsors.

\begin{lemma}\label{lemma:torsor}
	Let $V$ be a vector space over $\kk$ equipped with a $\gm$-action given by $t. v = t^d \cdot v$, where $d$ is any nonzero integer and the right hand side of the equation is the scalar multiplication in $V$. Assume that there exists an element $a \in \gm=\kkm$ such that $a^d \neq 1$. Suppose $W$ is a graded torsor (or affine space) over $V$ (regarded as an additive abelian group). Then $W$ is in fact a trivial torsor, i.e., $W$ is isomorphic to $V$ canonically as graded $V$-torsors. In particular, there is a unique $\gm$-stable point in $W$.
\end{lemma}

\begin{proof}
	Let $w_1 - w_2 \in V$ denotes the difference of two elements of $W$. Define a map $f: W \to V$ by $f(w) = (a^d - 1)^{-1} \cdot (a.w - w)$. Then one can check that $f$ is a $\gm$-equivariant isomorphism of $V$-torsors. The only $\gm$-stable point in $V$ is the zero element by the assumption of $a$. So $W$ has a unique $\gm$-stable point. Note that $f$ is independent of the cohice of $a$.
\end{proof}

\begin{proof}[Proof of Proposition \ref{prop:stable}]
	Our field $\kk$ is of characteristic zero, therefore satisfies the condition in Lemma \ref{lemma:torsor}. The uniqueness of $\gm$-stable point in $\Lambda$ follows from the uniqueness of $\gm$-stable point in $\hone(Y, 1 + \hb \pws)$. Now $\Lambda$ is a projective limit of torsors $\Lambda_n$ over the groups $\hone(Y, (1+ \hb \pws )/ ( 1+ \hb^n \pws))$, $\forall~n \geq 1$. Each connecting map $\Lambda_{n+1} \epi \Lambda_{n}$ is a torsor over the abelian group $\hone(Y, \hb^n\kk)$ ($\hb^n\kk$ is regarded as an additive group), which is equipped with a $\gm$-action of weight $2n$. By induction on $n$, after finding the $\gm$-stable point at each step, we can then apply Lemma \ref{lemma:torsor} to lift it to a unique $\gm$-stable point at the next level. The result of the proposition follows.
\end{proof}

\subsection{Graded quantization of vector bundles: proof of Theorem \ref{thm:lag_vec_graded}}\label{subsec:quan_lag_vec_proof}

The proof of Theorem \ref{thm:lag_vec_graded} goes in the same way as that of Theorem \ref{thm:lag_vec} which is \cite[Theorem 1.1]{BC}, with appropriate adjustments to incorporate the $\gm$-equivariant structures on the HC pairs and torsors as in the proof of Theorem \ref{thm:quan_lag_graded}. We only sketch the key ingredients here. We adopt the same notations in \S\,\ref{subsec:hcpair}.

By Lemma \ref{lemma:torsor_lag_vec}, we need to construct a graded transitive HC torsor $\Pmr$ over the HC pair $\autderdmr$ which lifts the graded HC torsor $\PJ$ over $\autderj$. As discussed in \cite[\S\,5]{BC}, the lift can be divided into a chain of lifts of torsors
\[ \Pmr = \P \epi \P_2 \epi \P_1 \epi \P_0 = \PJ    \]
corresponding to a chain of surjections
\[ \gh \epi \ghtwo \epi \ghone \epi \ghzero   \]
of graded HC pairs, where
\[ \gh = \autderdmr , \quad \ghzero = \autderj, \]
\[ \ghtwo = \gh /  \lab \kkm, \kk \rab  \cong \bautderdm\times\pglpgl  \quad (\text{by \eqref{eq:autderdmr_quot}}),\]
and
\[ \ghone = \autderj\times\pglpgl. \]
By \cite[Theorem\,5.1, (1)]{BC}, a lift $\P_1 $ of $\P_0 = \PJ$ to an ungraded transitive $\ghone$-torsor is of the form 
  \begin{equation} \label{eq:P1}
     \P_1 = \PJ\times\P_{\PGLh}, 
   \end{equation}
where $\P_{\PGLh}$ is a flat $\PGLh$-torsor lifting the flat $\PGLr$-torsor $\P_{\PE}$ associated to $\PE$. In general, there can be many different such lifts of $\P_{\PE}$ equipped with various graded structures. Among them there is a special one, 
  \begin{equation} \label{eq:Ppgl}
    \Ppgl := \P_{\PE} \times_{\PGLr} \PGLh,
  \end{equation}
the pushforward torsor $\P_{\PE}$ via the natural embedding $\PGLr \to \PGLh$ which is equipped with the graded structure induced from that of $\P_{\PE}$ and $\PGLh$. The $\gm$-invariant flat connection of $\P_{\PE}$ also extends uniquely to a $\gm$-invariant flat connection of $\Ppgl$. In fact, when the flat connection on $\PE$ is strong, we can show that $\Ppgl$ with the canonical flat connection is the only possible lift. To show this we need the following lemma, whose proof is straightforward by the definition of strong flat connection and Cartan's magic formula. 

\begin{lemma}\label{lemma:deRham_strong}
	For a strong flat connection of $\PE$ with the strong $\gm$-equivariant structure, denote by $\nabla$ the induced flat connection on the associated vector bundle $\pend(E) = \send_{\stry}(E)/\stry$. Then the induced $\gm$-action on the de Rham cohomology groups $H^\bullet_{\nabla}(Y, \pend(E))$ is trivial.
\end{lemma}

\begin{proposition}\label{prop:ghzero_lift}
	Given a strong flat connection on $\PE$, $\Ppgl$ is the unique (up to isomorphism) graded flat $\PGLh$-torsor which lifts the graded flat $\PGLr$-torsor $\P_{\PE}$.
\end{proposition}

\begin{proof}
	The proof is by a similar inductive argument with respect to the order of $\hb$ as in the proof of Proposition \ref{prop:bPdm_lift}. The only difference is that $\PGLh$ is nonabelian. However, the relevant torsor of all possible lifts of a flat $PGL(r, \pws/\hb^n\pws)$-torsor to a flat $PGL(r, \pws/\hb^{n+1}\pws)$-torsors at each step is a torsor over the abelian group $H^1_\nabla(Y, \hb^n\pend(E))$, $n \geq 1$. By strongness of the flat connection and Lemma \ref{lemma:deRham_strong}, the $\gm$-action on $H^1_\nabla(Y, \hb^n\pend(E))$ has positive weight.  Therefore Lemma \ref{lemma:torsor} implies that there is a unique $\gm$-stable point in the torsor at each step. The resulting flat $\PGLh$-torsor is exactly $\Ppgl$. 
\end{proof}

Given a choice of $\P_1$, any lift $\P_2$ of $\P_1$ to a graded transitive $\ghtwo$-torsor is of the form $\P_2 = \bPdm\times\P_{\PGLh}$, where $\bPdm$ is a graded transitive $\bautderdm$-torsor which lifts $\P_0=\PJ$. By Proposition \ref{prop:bPdm_lift}, such $\bPdm$ always exists and is unique up to isomorphism. So there is no obstruction at this step.

The last step is to lift the graded transitive $\ghtwo$-torsor $\P_2$ to a graded transitive $\gh$-torsor $\P = \Pmr$. In fact, we have the isomorphism of \emph{$\autdmr$-torsors}
  \begin{equation}\label{eq:Pmr}
    \Pmr \cong \P_E \times_{\P_{\PP(E)}} \P_2
  \end{equation}
corresponding to the canonical isomorphism of groups
\[ \autdmr \cong \GLr \times_{\PGLr} \autdmr/\kkm  \]
due to \eqref{eq:autdmr}. The $\gm$-action on $\P_2$ and the $\gm$-action on $\P_{\PP(E)}$ extends uniquely to a $\gm$-action on $\Pmr$ by $\autdmr$-torsor automorphisms. By part (3) of Theorem 5.1, \cite{BC}, the extension of the $\autdmr$-torsor structure on $\Pmr$ to the structure of a transitive HC torsor over $\gh$ is equivalent to the choice of an isomorphism $ \tatp \isom \frac{1}{r} \TT(\det(E))$ of Picard algebroids, which was denoted as $\frac{1}{r}\det\vu^+$ in \eqref{eq:detvupr}. Therefore to require that the $\gm$-action on $\Pmr$ to be compatible with the HC $\gh$-torsor structure is equivalent to requiring $\frac{1}{r}\det\vu^+$ to be $\gm$-equivariant. Therefore we have the following graded analogue of the part (3) of Theorem 5.1, \cite{BC}, which completes the proof of Theorem \ref{thm:lag_vec_graded}.

\begin{lemma}\label{lemma:ghtwo_lift}
	Given a lift of a graded transitive $\ghtwo$-torsor $\P_2$ to a ungraded transitive $\gh$-torsor $\Pmr$, the graded structure of $\P_2$ lifts to a graded structure on $\Pmr$ which induces the given graded structure on $E$, if and only if the associated isomorphism $ \frac{1}{r} \det \vu^+: \tatp \isom \frac{1}{r} \TT(\det(E))$ of graded Picard algebroids in \eqref{eq:detvupr} is $\gm$-equivariant. Such lifting, if exists, is unique.
\end{lemma}

\section{Quantization of Lagrangian subvarieties with group actions}\label{sec:lag_group}

\subsection{Equivariant quantization of Lagrangian subvarieties}\label{subsec:equiv_lag}

We consider quantizations in the presence of group actions. Suppose we are in the setting of \S\,\ref{subsec:equiv}. We assume that the $G$-action $\beta$ lifts to a $G$-action $\beta_\hb$ on the graded quantization $\Q$ which commutes with the $\cm$-action. For instance, such a lift always exists under the assumptions of Corollary \ref{cor:equiv}. Let $Y$ be a graded Lagrangian subvariety of $X$. Now assume there is an algebraic subgroup $K$ of $G$ with Lie algebra $\liek$, such that the restriction $\beta|_K$ of the $G$-action $\beta$ to $K$-action preserves the ideal subsheaf $\IY \subset \str_X$. We will simply say that $Y$ is preserved by the $\beta|_K$.

Suppose $E$ is a graded vector bundle of rank $r$ over $Y$ satisfying the condition in Theorem \ref{thm:lag_vec_graded}, so that it admits a unique graded quantization $\Ehb$. We are interested in the case when $\Ehb$ admits a $K$-action which is compatible with $\beta_\hb|_K$, the restriction of the quantum $G$-action $\beta_\hb$ to $K$-action.

\begin{definition}\label{defn:Kquan}
	A \emph{$K$-equivariant structure}, or simply a $K$-action, on a graded quantization $\Ehb$ of $(Y, E)$ is a $K$-action $\aeh$ on $\Ehb$, such that
	\begin{enumerate}[label=(K\arabic*)]
		\item \label{cond:Kquan1}
		$\aeh$ commutes with the $\gm$-action and the $\pws$-module structure.
		\item  \label{cond:Kquan2}
		The module structure map $\Q \otimes \Ehb \to \Ehb$ is $K$-equivariant, where the $K$-action on $\Q$ is given by $\beta_\hb|_K$ and the $K$-action on $\Q$ is given by $\alpha_\hb$.
	\end{enumerate}
\end{definition}

Note that if such an $\aeh$ exists, then it descends to a $K$-action $\alpha$ on $E$ compatible with the $K$-action $\beta|_K$ and the $\gm$-equivariant structure. Conversely, when $E$ is equipped with such a $K$-action $\alpha$ , it is natural to ask if this action lifts to a $K$-equivariant structure on $\Ehb$. Note that the $K$-action $\beta_\hb|_K$ preserves the ideal $\JY \subset \Q$ and descends to a $K$-action on $\tat$ via the isomorphisms \eqref{eq:J_tat}, which is still denoted by $\beta_\hb|_K$. The $K$-action on $Y$ induces a $K$-equivariant structure on the canonical bundle $K_Y$ which commutes with the natural $\gm$-action on $K_Y$. This leads to a $K$-equivariant structure on the graded Picard algebroid $\tat + \frac{1}{2} \TT(K_Y)$. It can be shown that the isomorphism $v: \tatp \isom \tat + \frac{1}{2}\TT(K_Y)$ in \eqref{eq:tatp2tatky} is $K$-equivariant.

\begin{proposition}\label{prop:quan_lag_equiv}
	Under the assumptions of Theorem \ref{thm:lag_vec_graded}, suppose $E$ is a graded vector bundle over $Y$ with a strong $\gm$-action and $\Ehb$ is a graded quantization of $E$. Let $\alpha$ be any $K$-action on $E$ such that it commutes with the $\gm$-action on $E$. Then
	\begin{enumerate}
      \item
	 	The $K$-action $\alpha$ on $E$ can be lifted to a $K$-equivariant structure $\aeh$ on $\Ehb$ if and only if the associated morphism of graded Picard algebroids $ \vu^+: \tatp \to \TT(E) $ is $K$-equivariant. Here the $K$-action on the right hand side is the adjoint action induced by $\alpha$, where the $K$-action on the $\tatp$ is defined as above. 
	  \item 
	    The sufficient and necessary condition in part (1) is equivalent to the following requirements that,
		\begin{enumerate}
			\item \label{cond:quan_lag_equiv1}
				the isomorphism $\frac{1}{r} \det\vu^+: \tatp \to \frac{1}{r}  \TT(\det(E))$ in \eqref{eq:detvupr} is $K$-equivariant;
			\item \label{cond:quan_lag_equiv2}
				the flat connection on $\PE$ as in Theorem \ref{thm:lag_vec_graded} is $K$-equivariant with respect to the $K$-action on $\PE$ induced from that on $E$.
		\end{enumerate} 
     \item 
		If the condition in part (1) or the conditions in part (2) are fulfilled, then the $K$-equivariant structure $\aeh$ is unique (up to graded automorphisms of $\Ehb$ as $\Q$-module intertwining the projection $\Ehb \epi E$).
  \end{enumerate}
\end{proposition}

\begin{proof}
	 Part (2) follows from \eqref{diag:tatp} and discussions there. We only need to show the conditions in part (2) imply the conclusion in part (2). In other words, we need to lift the $K$-action on $\PJ$ induced by $\beta_\hb|_K$ and the $K$-action $\alpha$ on $E$ to a $K$-action on the transitive graded $\autderdmr$-torsor $\Pmr$, such that it commutes with the $\gm$-action on $\Pmr$. By the discussions in \S\,\ref{subsec:quan_lag_vec}, this amounts to lifting the $K$-action from $\PJ$ and $E$ to $\P_1$, $\P_2$ and $\P=\Pmr$ successively. We also need to show that all the lifts are unqiue to verify part (3).
	
	By Proposition \ref{prop:ghzero_lift}, passing from $\PJ$ to $\P_1$ is equivalent to lifting the $K$-action on the graded flat $\PGLr$-torsor associated to $\PE$ satisfying requirement \eqref{cond:quan_lag_equiv2} to the graded flat $\PGLh$-torsor $\Ppgl$ defined in \eqref{eq:Ppgl}.  Since $\Ppgl = \P_{\PE} \times_{\PGLr} \PGLh$, it is immediate that the $K$-action on $\P_{\PE}$ extends uniquely to a $K$-action on $\Ppgl$ which commutes with the $\PGLh$-action. Therefore we obtain the desired $K$-action on $\P_1 = \PJ\times\P_{\PGLh}$.
	
	
	
	Next, we need to show that the $K$-action on $\P_1$ lifts to a unqiue $K$-action on the $\ghtwo$-torsor $\P_2 = \bPdm\times\P_{\PGLh}$ which commutes with the $\gm$-action. This amounts to lifting the $K$-action on $\PJ$ to a $K$-action on $\bPdm$. This follows from a similar cocycle argument as in the proof of Proposition \ref{prop:kstable}. The main difference is that we \emph{do not} require $K$ to be reductive in this case. The compatibility with the $\gm$-action already enforces the cocycle condition automatically, since the only $\gm$-equivariant automorphism of the HC torsor $\bPdm$ which intertwines the projection $\bPdm \epi \PJ$ is the identity map by Lemma \ref{lemma:autpdm}. For the same reason, the lifted $K$-action is unique. Therefore we obtain the desired $K$-action on $\P_2$. This $K$-action together with the $K$-action on $E$ provides a $K$-action on $\Pmr$ by $\autdmr$-torsor automorphisms, via the isomorphism \eqref{eq:Pmr}. Finally, the requirement \eqref{cond:quan_lag_equiv1} ensures that the $K$-action is also compatible with the $\gh$-torsor structure.
	
\end{proof}

\begin{remark}\label{rmk:pair}
  For applications in orbit method, we need a slightly more general setting in which $K$ is not necessarily a subgroup of $G$. Suppose there is a group homomorphism $J: K \to G$ which induces an inclusion of Lie algebras $j: \liek \inj \lieg$. Then the adjoint action of $K$ on $\lieg$ via $J$ is compatible with $j$ and $(G,K)$ is a \emph{pair} in the sense of \cite[Definition 5.12]{Vogan_AV}. In this case, Definition \ref{defn:Kquan} and Proposition \ref{prop:quan_lag_equiv} are still valid without any change. We only need to redefine $\beta|_K$ and $\beta_\hb|_K$ as the $K$-actions via $J: K \to G$ without changing the notations. This also applies \S\, \ref{subsec:moment_lag} below.
\end{remark}

\subsection{Hamiltonian quantization of Lagriangian subvarieties}\label{subsec:moment_lag}

Suppose we are in the setting of \S\,\ref{subsec:equiv_lag} and Proposition \ref{prop:quan_lag_equiv}. Note that there can be many different $K$-action $\alpha$ on $E$ which induce the same adjoint action on $\TT(E)$. We now put an extra restriction on the choice of $\alpha$ in the presence of a quantized moment map (see Definition \ref{defn:quan_lag_ham} below). Consider the situation where $(X,\sym)$ admits a Hamiltonian $G$-action $\beta$ that commutes with the $\gm$-action as in \S\,\ref{subsec:moment}. Let $\Q$ be a graded quantization of $X$. Now suppose there exists a quantum Hamiltonian $G$-action $\beta_\hb$ on $\Q$ together with a quantized moment map $\mu^*_\hb: \ugh \to \Gamma(X, \Q)$ which lifts the classical moment map $\mu: X \to \lieg^*$. For instance, this true under the conditions of Proposition \ref{prop:moment}. Let $Y$ be a graded Lagrangian subvariety as before. We make the following assumptions on $Y$, $\mu$ and $\beta$:
\begin{enumerate}[label=(M)]
	\item \label{cond:M}
	The image of $\liek$ under $\mu^*: S\lieg \to \Gamma(X, \strx)$ lies in $\Gamma(X, \IY)$.
\end{enumerate}

\begin{remark}
	The condition \ref{cond:M} together with the definition of moment map and the Lagrangian condition force $\IY$ to be preserved by the infinitesimal action 
	$d\beta|_K$ of $\liek$. However, we do not assume $K$ is connected. Therefore the assumption that $\beta|_K$ preserves $\II_Y$ is not redundant.
\end{remark}
Recall that in \ref{subsec:moment} we identify $\lieg$ with $\hb\pb\lieg$. Now the Lie algebra $\hb\pb\lieg$ acts on $\Ehb$ via $ \mu^*_\hb: \hb\pb\lieg \to \Gamma(X,\Q)$ and the $\Q$-module structure of $\Ehb$. This action restricts to a Lie algebra action of $\liek$ on $\Ehb$, which lifts the infinitesimal action of $\liek$ on $Y$. Therefore $\liek$ also acts on $(i_Y)_*E \cong \Ehb/\hb\Ehb$ or simply $E$. To unload notations, we will denote both actions on $\Ehb$ and $E$ by $\mhk := \mu^*_\hb |_\liek$. For the remaining part of this section, we shall compare the two $\liek$-actions $d\aeh$ and $\mhk$ on $\Ehb$.

\begin{definition}\label{defn:quan_lag_ham}
	We say that $(\Ehb, \aeh, \mu^*_\hb)$ is a \emph{(graded) Hamiltonian quantization} of $(Y, E)$ if, besides \ref{cond:Kquan1} and \ref{cond:Kquan2} in \ref{defn:Kquan}, the following additional condition is satisfied:
	\begin{enumerate}[label=(K\arabic*)]
		\setcounter{enumi}{2}
		\item  \label{cond:Kquan3}
		The two $\liek$-actions $d\aeh$ and $\mhk$ on $\Ehb$ coincide.
	\end{enumerate}
\end{definition}

We want to find natural criteria for the condition \ref{cond:Kquan3}. Let $\hJY$ be the completion of $\JY$ along $Y$. Let $\der(\Q)_{\JY}$ be the sheaf of derivations of $\Q$ which preserves $\JY$ and let $\widehat{\der}(\Q)_{\JY}$ be its completion along $Y$. Set $\derehb$ to be the sheaf of derivations of $\Ehb$ as $\Q$-module. Then there is a commutative diagram of graded sheaves of Lie algebras over $Y$
\begin{equation}\label{diag:derehb}
\begin{tikzcd}
0  \arrow{r}  & \underline{\pws}  \arrow{r} \arrow[d, hookrightarrow]  & \hb\pb\hJY \arrow[r] \arrow[d, hookrightarrow]   & \widehat{\der}(\Q)_{\JY} \arrow{r} \arrow[d, equal] & 0  \\
0  \arrow{r}  & \underline{\pws} \oplus \pend_\nabla(E)    \arrow{r}  & \derehb  \arrow{r}   &  \derdj  \arrow{r} & 0,
\end{tikzcd}
\end{equation}
which is the localization of the diagram \eqref{diag:derdmr} with respect to the HC torsor $\PJ$. Here $\KKs$ is the constant sheaf with coefficient $\pws$ over $Y$ and $\pend_\nabla(E)$ is the sheaf of flat sections of $\pend(E)$ with respect to $\nabla$. For any $Z \in \liek$, $\mhk(Z)$ is a section of $\hb\pb\hJY$. Therefore a necessary condition for \ref{cond:Kquan3} to hold is that $d\aeh(Z)$, which is \emph{a priori} a section of $\derehb$, is a section of $\hb\pb\hJY$ for any $Z \in \liek$. We can regard $d\aeh(Z)$ as an $\autdmr$-invariant vector field on the torsor $\Pmr$. Its projection to the $\PGLh$-torsor $\P_{\PGLh}$, which appears in \eqref{eq:P1}, is the vector field $\xi_Z$ generated by $Z$ via the $\liek$-action induced by $d\aeh$.  By the decomposition \eqref{eq:derdmr}, $d\aeh(Z)$ is a section of $\hb\pb\hJY$ if and only if $\xi_Z$ is horizontal with respect to the flat connection of $\P_{\PGLh}$. In this situation, we say that the action on the flat torsor is \emph{horizontal}. In particular, the induced $\liek$-action on $\P(E)$ is also horizontal with respect to the flat connection of $\P(E)$. Conversely, starting with a horizontal $K$-action on $\P(E)$, the unique extended $K$-action on the special $\PGLh$-torsor $\Ppgl$ is also horizontal with respect to the flat connection on $\Ppgl$.

Now assume that the $K$-action on $\P_{\PGLh}$ is horizontal, so that both $d\aeh(Z)$ and $\mhk(Z)$ lie in the $\gm$-invariant part of $\Gamma(Y, \hb\pb\hJY)$, for any $Z \in \liek$.  Moreover, both their images under the projection $\hb\pb\hJY \to \widehat{\der}(\Q)_{\JY}$ coincide with $d\beta_\hb(Z)$. Thus the map $\rho:=d\aeh - \mhk$ has images in the $\gm$-invariant part $\hdr^0(Y,\kk)$ of $\hdr^0(Y,\pws)$. By Definition \ref{defn:Kquan} of $\aeh$ and Definition \ref{defn:quan_moment} of quantized moment map, it is standard to check that $\rho : \liek \to \kk$ defines a Lie algebra homomorphism. 

Note that we have a commutative diagram of short exact sequences of graded sheaves of Lie algebras
\begin{equation}\label{diag:rho}
\begin{tikzcd}
0  \arrow{r}  & \underline{\pws} \arrow{r} \arrow[d]  & \hb\pb\hJY \arrow[r] \arrow[d, "\varphi_Y"]   & \widehat{\der}(\Q)_{\JY}  \arrow{r} \arrow[d] & 0  \\
0  \arrow{r}  & \stry \arrow{r}  & \tatp \arrow[r, "\delta"]  &  \T_Y  \arrow{r}   & 0  ,
\end{tikzcd}
\end{equation}
where the middle vertical map is the quotient map modulo $\hb\pb\hJY^2$ as in \eqref{eq:J_tat}. The leftmost vertical map is the composition of the natural projection $\underline{\pws} \to \underline{\kk}$ and the embedding $\underline{\kk} \hookrightarrow \stry$, which is injective when restricted to the $\gm$-invariant part $\kk \subset \pws$. Therefore $\rho$ is in fact determined by the descent of $d\aeh$ and $\mhk$ to $\tatp$ via the map $\varphi_Y$.  The former can be identified with a strong $\liek$-action on $\frac{1}{r} \TT(\det(E))$ via $\det\Upsilon^+: \tatp \isom \frac{1}{r}  \TT(\det(E))$. The latter is still denoted by $\mhk$. 

We want to find natural conditions under which $\rho=0$. Of course this can be ensured by assuming $\liek=[\liek,\liek]$. However, we are mainly interested in the case when $K$ is the complexification of a maximal compact subgroup $\KR$ of a real reductive Lie group $\GR$, so that $\liek$ might have a nontrivial center. We restrict ourselves to the case when $\Q$ has a even structure, i.e., there is a $\gm$-equivariant involutory anti-automorphism $\epsilon$ of $\Q$ which sends $\hb$ to $-\hb$. By Proposition \ref{prop:moment}, we can assume that the quantum Hamiltonian $G$-action $\beta_\hb$ on $\Q$ and the quantized moment map $\mh$ are symmetrized with respect to $\epsilon$. 

By Proposition \ref{prop:graded}, we have $\per(\Q)=\sym_1(\Q)=0$. Thus equation \eqref{eqn:aty} implies that the image of $\aty$ in $\htwo(Y)$ vanishes. In fact, we will show below that $\tat$ is isomorphic to the trivial Picard algebroid in this case. Since the involution $\epsilon$ on $\Q$ descends to the identity map on $\strx$, it preserves the ideal sheaves $\JY$, $\JY'$ and $\JY^2$ and therefore the involution $\tilde{\epsilon}$ on $\hb\pb\Q$ (introduced in \S\,\ref{subsec:moment}) also preserves $\hinv\JY$, $\hinv\JY'$ and $\hinv\JY^2$. Therefore $\tilde{\epsilon}$ descends to an involution on $\hinv\JY / \hinv\JY^2$ of $\gm$-equivariant Lie algebra \emph{automorphisms}, which is still denoted as $\tilde{\epsilon}$. Moreover, $\tilde{\epsilon}$ is $\stry$-linear with respect to the symmetrized product $\bullet$ defined in \eqref{eq:symmprod}. The $-1$-eigensubsheaf of $\tilde{\epsilon}$ is exactly $\str_Y \cong \hb\str_Y$. Therefore $\tat$ splits into a direct sum $\str_Y \oplus \T_Y$.  In particular, the class $\aty$ vanishes. Since $\tilde{\epsilon}$ is $\gm$-equivariant, the splitting also preserves the $\gm$-equivariant structure. In summary, we have the following lemma.

\begin{lemma}\label{lemma:atc_even}
	Suppose $(\Q,\epsilon)$ is an even graded quantization of $(X,\sym)$. Then $\tat$ is isomorphic (as graded Picard algebroids) to the trivial Picard algebroid $\str_Y \oplus \T_Y$ with the obvious $\gm$-action. 
\end{lemma}

We can also regard $\tilde{\epsilon}$ as an involution on $\tatp$ via the isomorphism $s: \tat \to \tatp$, but it is not linear with respect to the $\str_Y$-module structure on $\tatp$. Set $\star := - \tilde{\epsilon}$, then $\star$ is a $\gm$-equivariant involution on $\tatp$ satisfying
\begin{equation}\label{cond:alg_inv}
\star([\partial_1,\partial_2]) = -[\star\partial_1, \star\partial_2], \quad \star(f \partial) = f (\star\partial) - \wt{\partial}(f), \quad \star(1)= 1, \quad \text{and} \quad \wt{\star\partial} = -\wt{\partial},
\end{equation}
for any local sections $\partial_1, \partial_2, \partial$ of $\tatp$ and local sections $f$ of $\stry$, where $\wt{\partial}$ denotes the image of $\partial$ in $\T_Y$. Here we are using the $\stry$-module structure of $\tatp$. In fact, the existence of such $\star$ characterizes the Picard algebroid $\frac{1}{2} \TT(K_Y)$. By \cite[\S 2.4]{BB}, there is also a canonical involution $\APLstar$ on the Picard algebroid $\frac{1}{2} \TT(K_Y)$ satisfying the same property \eqref{cond:alg_inv} as $\star$. We have the following reformulation of Lemma \ref{lemma:atc_even}.

\begin{lemma}\label{lemma:atc_sqrt}
	Suppose $(\Q,\epsilon)$ is an even graded quantization of $(X,\sym)$. Then there exists a unique $\gm$-equivariant isomorphism $\tatp \cong \frac{1}{2} \TT(K_Y)$ of graded Picard algebroids which intertwines $\star$ with the canonical involution $\APLstar$ on $\frac{1}{2} \TT(K_Y)$.
\end{lemma}

 The symmetrized condition on $\mh$ gives
\begin{equation}\label{eq:inv_k}
\star [\mhk(Z)] = -\mhk(Z) \quad \text{in} ~ \tatp  , \quad \forall ~  Z \in \liek \cong \hb\pb\liek \subset \hb\pb\lieg.
\end{equation}
By Lemma \ref{lemma:atc_sqrt}, there exists a unique $\gm$-equivariant isomorphism $\tatp \cong \frac{1}{2} \TT(K_Y)$ of graded Picard algebroids which intertwines $\star$ and the involution $\APLstar$ on $\frac{1}{2} \TT(K_Y)$. Moreover, the canonical strong $\liek$-action on $\frac{1}{2} \TT(K_Y)$ induced by $\beta|_K$ (restricted to $Y$) also satisfies an equality similar to \eqref{eq:inv_k}. This implies that the isomorphism intertwines the strong $\liek$-actions. 

For any graded Picard algebroid $\TT$ equipped with a strong $\liek$-action that commutes with the $\gm$-action,  set $\aut_\liek(\TT)$ to be the group of Picard algebroid automorphisms of $\TT$ that commute with the strong $\liek$-action  and the $\gm$-action. Note that, with respect to the natural strong $\liek$-actions on $\TT(K_Y)$ and $\frac{1}{2}\TT(K_Y)$ induced by $\beta|_K$, we have a canonical group isomorphism $\aut_\liek(\TT(K_Y)) \cong \aut_\liek(\frac{1}{2}\TT(K_Y))$ via the canonical isomorphism $\frac{1}{2}\TT(K_Y) + \frac{1}{2}\TT(K_Y) \cong \TT(K_Y)$ of Picard algebroids. We finally arrive at the following main result of this section.

\begin{theorem}\label{thm:lag_equiv}
	Suppose $(\Q,\epsilon,\beta_\hb,\mh)$ is a symmetrized $G$-Hamiltonian quantization of a graded symplectic variety $(X,\sym, \beta, \mu)$. Let $E$ be a graded vector bundle of rank $r$ over $Y$ equipped with a compatible $K$-action $\alpha$,  such that $\PP(E)$ is equipped with a strong flat connection. Then there exists a graded quantization $\Ehb$ of $E$ equipped with a $K$-equivariant structure $\aeh$, such that $(\Ehb, \aeh, \mh)$ is a Hamiltonian quantization of $(Y,E)$, if and only if the following conditions are satisfied:
	\begin{enumerate}
		\item 
		the $K$-action on $\PP(E)$ is horizontal with respect to the strong flat connection;
		\item 
		there exists an isomorphism $\frac{1}{2} \TT(K_Y) \isom \frac{1}{r} \TT(\det(E))$ of graded Picard algebroids which intertwines the natural strong $\liek$-action on $\frac{1}{2} \TT(K_Y)$ induced by $\beta$ and the strong $\liek$-action on $\frac{1}{r} \TT(\det(E))$ induced by $\alpha$.
	\end{enumerate}
	In this case, the set of isomorphism classes of $(\Ehb, \aeh)$ with the fixed $E$ is in bijection with the set of isomorphisms in (2), which is a torsor over the group $\aut_\liek(\TT(K_Y))$.
\end{theorem}

\begin{remark}
	All the results so far can be generalized to families of Lagrangian subvarieties without any change in the arguments. We leave the details to interested readers.
\end{remark}

\section{Quantization of nilpotent orbits}\label{sec:quan_orbit}

\subsection{Quantization of complex orbits}\label{subsec:quan_c}

From now on, we set $\kk= \C$ and $G$ to be a complex semisimple with Lie algebra $\lieg$. Therefore $\gm = \cm$. Any coadjoint orbit $\OO$ in $\lieg^*$ is naturally a smooth symplectic variety equipped with a $G$-invariant algebraic symplectic form $\sym$. Suppose $\OO$ is nilpotent, then it has a $\cm$-action such that $t. \sym = t^2 \sym$ by \cite{BK}(see also \S\,\ref{subsec:grading}). So we are in the setting of \S\,\ref{subsec:graded}. However, $\OO$ is far from satisfying the admissibility condition \eqref{eq:adm}, so we cannot directly apply Theorem \ref{thm:per} and Proposition \ref{prop:graded} to construct quantization of $(\OO,\sym)$. 

The key observation made by Losev in \cite{Losev2} is that $\OO$ can be embedded as an open subvariety into some smooth symplectic variety that is strongly admissible in the sense of Definition \ref{defn:strong_adm}. Namely, take $X = \Spec(\C[\OO])$. This is an affine Poission variety which is the normalization of the closure $\cl{\OO}$ of $\OO$ in $\lieg^*$ (\cite{BK}). It was shown in \cite{Panyushev} and \cite{Hinich} that $X$ is a conical symplectic singularity in the sense of Beauville (\cite{Beauville1}). We have the following result about general symplectic singularities (\cite[Prop. 2.3]{Losev2})
\begin{proposition}
	Let $X$ be a symplectic singularity. Then there is a birational projective morphism $\rho: \wt{X} \to X$ such that $\wt{X}$ satisfies the following properties:
	\begin{enumerate}[(a)]
		\item 
		$\wt{X}$ is an irreducible normal Poisson variety which has symplectic singularieties.
		\item 
		$\wt{X}$ is $\QQ$-factorial.
		\item   
		$\wt{X}$ has terminal singularities.
		\item 
		If additionally $X$ is conical, then $\wt{X}$ admits a $\cm$-action such that $\rho$ is $\cm$-equivariant.
	\end{enumerate}
\end{proposition}
The form $\rho^*\sym$ extends to a symplectic form on the smooth part $\xreg$ of $\wt{X}$, still denoted as $\sym$. Such a $\wt{X}$ is called a \emph{$\QQ$-factorial terminalization} of $X$. Under the assumption (a), Namikawa \cite{Namikawa1} showed that (c) is equivalent to the condition that $\codim(\wt{X} - \xreg, \wt{X}) \geq 4$. Then the proof of \cite[Lemma 12]{Namikawa2} shows that

\begin{proposition}\label{prop:adm}
	If $X$ is affine, then $H^i(\xreg, \str_{\wt{X}})=0$ for $i=1, 2$.
\end{proposition}

Therefore $(\xreg,\sym)$ is strongly admissible as a smooth symplectic variety in the sense of \eqref{defn:strong_adm}. Then one can consider graded quantizations of $(\xreg,\sym)$ using Corollary \ref{cor:per_graded}. Moreover, if $i: \xreg \hookrightarrow \wt{X}$ denotes the embedding, \cite[Proposition 3.4]{BPW} shows that the functors $i_*$ and $i^*$ give a bijection between isomorphism classes of graded quantization of $\xreg$ and $\wt{X}$.  Therefore graded quantizations of $\wt{X}$ are classified by $H^2(\xreg,\C)$. It is shown in \cite[\S 3.2]{Losev2} that any such quantization $\tQ$ gives a filtered quantization $\A$ of the Poisson algebra $\C[X]$ by taking $\cm$-finite global sections and setting $\hb = 1$. Conversely any filtered quantization of $\C[X]$ arises in this way. 

Back to the case when $X=\Spec \C[\OO]$. We will focus on the unique (up to isomorphisms) even quantization $(\tQ, \epsilon)$ of $\xreg$ since they are related to Vogan's Conjecture \ref{conj:vogan}. Proposition \ref{prop:moment} ensures a symmetrized quantum Hamiltonian $G$-action on $\Q$ equipped with the quantized moment map $\mu^*_\hb: \ugh \to \Gamma(\xreg, \tQ)$. The evaluation of $\mu^*_\hb$ at $\hb = 1$ gives an algebra homomorphism $\Phi: \Ug \to \A$, which makes $\A$ into a Dixmier algebra. Set $\JJ_{\OO} := \ker \Phi$, then \cite[Lemma 3.7]{Losev2} shows that $\JJ_{\OO}$ is a multiplicity-free complete prime primitive ideal of $\Ug$, such that its associated variety is $\cl{\OO}$. In particular, $\A$ is a Harish-Chandra bimodule of $G$ with infinitesimal character, denoted by $\ld_\OO$.

\begin{remark}
	Note that Losev did not require his quantized moment map in \cite{Losev2} to be symmetrized, but it is important for us to have a symmetrized one for the discussions in \S\,\ref{subsec:quan_r}.
\end{remark}

The restriction $\Q:=\tQ|_{\OO}$ of $\tQ$ to the open subset $\OO$ is again an even quantization of $(\OO,\sym)$, equipped with the restriction of $\epsilon$ which is still denoted by the same symbol. Moreover, the composition of the symmetrized quantized moment map $\mu^*_\hb$ for $\tQ$ and the restriction map $\Gamma(\xreg, \tQ) \to \Gamma(\OO,\Q)$ gives a symmetrized quantized moment map for $\Q$. When $\OO$ is a \emph{birationally rigid orbit} in the sense of \cite[\S 4.1]{Losev2}, it was shown in \cite[Prop. 4.4]{Losev2} that $\wt{X} = X$ and $H^2(X^{reg}, \C) = 0$. Therefore there is one unique quantization $\Q$ of $X^{reg}$ which is even.

\begin{remark}
	For a non-birationally rigid orbits, or even more generally, any affine variety with conical symplectic singuarities, Losev \cite{Losev2} has classified all the quantization of its algebra of regular functions as a Poisson algebra and showed that the parameter space is isomorphic to the parameter space of Poisson deformations. We will focus on the birational rigid case, but it is also possible to apply the family versions of our results in \S\,\ref{sec:lag_group} to the general case, where a general nilpotent orbit fits into a family of (non-nilpotent) coadjoint orbits.
\end{remark}

\subsection{Quantization of $K$-orbits}\label{subsec:quan_r}

Suppose that $\GR$ is a real reductive Lie group of Harish-Chandra class. In particular, the adjoint action of $K$ on $\lieg$ is inner. Let $G = G_{ad}$ be the adjoint group of $\lieg$, then $G$ is a connected semsimple complex Lie group. Moreover, the adjoint action of $K$ gives a group homomorphism $J: K \to G$ which satisfies the assumption in Remark \ref{rmk:pair}. Let $\OO_\liep = K \cdot x \subset \liep^* \cong \liep$ be a nilpotent $K$-orbit. Then $\OO := G \cdot x$ is the unique $G$-orbit in $\lieg^* \cong \lieg$ containing $\OO_\liep$ due to our assumption on the group $\GR$. By \cite[Prop.\,5.16, Cor.\,5.20]{Vogan_AV}, $\OO_\liep$ is a closed smooth Lagrangian subvariety of $\OO$. Let $(\Q,\epsilon)$ be the even quantization of $\OO$ as constructed in \S\,\ref{subsec:quan_c}, which is equipped with a unique symmetrized quantum Hamiltonian $G$-action $\beta_\hb$ and a symmetrized quantized moment map $\mu^*_\hb: \ugh \to \Gamma(X, \Q)$. Recall that Vogan's Conjecture \ref{conj:vogan} concerns a vector bundle of the form $\V_\chi$ determined by any irreducible representation $(\chi, V_\chi)$ of the stablizer $K^x$ satisfying \eqref{eq:chi} via \eqref{eq:vchi}. By \S\,\ref{subsec:grading}, the vector bundle $\V_\chi$ can be equipped with a $\cm$-action which commutes with the $K$-action. The $\cm$-action is only unique up to a degree shift, but note that different choices lead to the same $\cm$-action on the Picard algebroid $ \TT(\det(\V_\chi))$, i.e., the one induced from the canonical $\cm$-action on the canonical bundle $K_{\OO_\liep}$. By Remark \ref{rmk:grading}, this ambiguity does not affect our discussion of graded quantization of $\V_\chi$. Again we use $r$ to denote the rank of $\V_\chi$.

\begin{remark}
	Comparing the definition of admissible orbit data \eqref{eq:chi} and condition (2) of Theorem \ref{thm:lag_equiv}, one immediately sees the similarity. It should not be a surprise. Vogan proved that admissible orbit data are the right candidates for classical limits of representations (\cite[Thm. 8.7]{Vogan_AV}) by analyzing the first-order information (cf. \cite[\S\,11]{Vogan_AV}), which is essentially the same as the arguments in the proof of Theorem \ref{thm:lag_equiv}.
\end{remark}

\begin{proposition}\label{prop:vchi}
	The vector bundles $\V_\chi$ in Conjecture \ref{conj:vogan} are exactly the irreducible $K$-equivariant graded vector bundles satisfying the conditions of Theorem \ref{thm:lag_equiv} with respect to $(\Q,\epsilon)$. More precisely, 
	\begin{enumerate}
		\item 
		the projectivization $\PP(\V_\chi)$ is equipped with a unique strong flat connection, with respect to which the induced $K$-action is horizontal, and
		\item 
		there is a unique isomorphism $\frac{1}{2} \TT(K_Y) \isom \frac{1}{r} \TT(\det(\V_\chi))$ of graded Picard algebroids which intertwines the natural strong $\liek$-action on $\frac{1}{2} \TT(K_Y)$ and the induced strong $\liek$-action on $\frac{1}{r} \TT(\det(\V_\chi))$.
	\end{enumerate}
	Conversely, any irreducible $K$-equivariant graded vector bundle satisfying the conditions of Theorem \ref{thm:lag_equiv} is of the form $\V_\chi$ for some $\chi$. Therefore the set of equivalence classes of $\chi$ satisfying \eqref{eq:chi} is in bijection with the set of isomorphism classes of irreducible $K$-equivariant graded vector bundles satisfying the conditions of Theorem \ref{thm:lag_equiv}.
\end{proposition}

\begin{proof}
	The proof is almost straightforward. Given an admissible $K$-equivariant vector bundle $\V_\chi$ over the $K$-orbit $\OO_\liep$, \eqref{eq:chi} implies that the projectivization $\PP(\V_\chi)$ is equipped with a unique flat connection with respect to which the $K$-action is horizontal. The infinitesimal action induced by $\cm$-action on $\V_\chi$ satisfies the strongness condition at the base point $x \in \OO_\liep$, then by $K$-equivariance it is strong everywhere on $\OO_\liep$. Again by \eqref{eq:chi} there is a graded isomorphism $\frac{1}{r} \TT(\det(\V_\chi)) \isom  \frac{1}{2} \TT(K_Y)$ intertwining the strong $\liek$-actions. Since the $K$-action on $\OO_\liep$ is transitive, such an isomorphism is unique.
	
	Conversely, if $E$ is a $K$-equivariant vector bundle satisfying the assumptions of Theorem \ref{thm:lag_equiv}, the condition (1) implies that the stablizer group $K^x$ acts on the fiber $E_x$ by scalars and condition (2) implies that the scalars is exactly given by \eqref{eq:chi}.
\end{proof}

By Proposition \ref{prop:vchi}, we can apply Theorem \ref{thm:lag_equiv} to $\Q$ and the vector bundle $E = \V_\chi$ over $\OO_\liep$, and conclude that there exists a unique Hamiltonian quantization $\Ehb$ of $(\OO_\liep, \V_\chi)$. Let $\M_\hb$ be the $\cm$-finite part of $\Gamma(\OO_\liep, \Ehb)$. Then by Definition \ref{defn:quan_lag_ham} of Hamiltonian quantization, $\M_\hb$ is a module over $\ugh$ with a compatible $K$-action such that $\M:=\M_\hb|_{\hb = 1}$ is a $(\lieg,K)$-module. Now assume that $\OO_\liep$ is a Vogan orbit in the sense of Definition \ref{defn:voganorbit}, i.e., $\codim(\partial \OO_\liep, \cl{\OO}_\liep) \geq 2$. Then $\C[\OO_\liep]$ is a finitely generated algebra and therefore by \cite[Lemma 3.2.1]{Losev4}, $\Gamma(\OO_\liep, E)$ is a finitely generated $\C[\OO_\liep]$-module. As a consequence, the grading on $\Gamma(\OO_\liep, E)$ is bounded from below since the grading on $\C[\OO_\liep]$ is nonnegative. Since the action of $\cm$ on $\Gamma(\OO_\liep, E)$ is locally finite, any eigenvector of the $\cm$-action on $\Gamma(\OO_\liep, E)$ can be lifted to an eigenvector in $\Gamma(\OO_\liep, \Ehb)$ with the same eigenvalue. Thus we have canonical maps
\begin{equation}\label{eq:Kembed}
\M_\hb / \hb\M_\hb \cong \Gamma(\OO_\liep, \Ehb) / \hb \Gamma(\OO_\liep, \Ehb) \hookrightarrow \Gamma(\OO_\liep, E). 
\end{equation}
Therefore we deduce that $\M_0 := \M_\hb / \hb\M_\hb$ is a finitely generated $\C[\OO_\liep]$-submodule of $\Gamma(\OO_\liep, E)$. This implies that $\M_0$ is a finitely generated module over both $\C[\OO]$ and $S\lieg$. Now $\M$ has a natural good filtration whose associated graded module is isomorphic to $\M_0$, we conclude that $\M$ is a finitely generated module over both $\A$ and $\Ug$. Moreover, the $K$-action on $\Gamma(\OO_\liep, E)$ is locally finite and by Frobenius reciprocity, the multiplicity of each irreducible finite dimensional representation of $K$ appearing in $\Gamma(\OO_\liep, E)$ is finite. Then by \eqref{eq:Kembed} again we see that $\M$ is admissible (see \cite[\S\,3.3]{Wallach}). Thus we have

\begin{theorem}\label{thm:hc1}
	The $(\lieg,K)$-module $\M$ constructed above from a Vogan orbit is an admissible Harish-Chandra module of $\GR$ with a canonical good filtration (defined up to degree shift), such that its associated variety is contained in $\cl{\OO}_\liep$.
\end{theorem}

\begin{remark}
	The conclusion that $\M$ is a Harish-Chandra module is in fact valid for any orbit, which is not necessarily a Vogan orbit. Indeed, $\M$ is admissible since the $K$-action on $\Gamma(\OO_\liep,E)$ is admissible. On the other hand, $\M$ is a module over $\A$ and therefore has an infinitesimal character. Thus $\M$ is Harish-Chandra. On the other hand, when $\codim(\partial \OO_\liep, \cl{\OO}_\liep) =1$, the associated variety of $\M$ can be bigger than $\cl{\OO}_\liep$. See the example below. In this case we do not get a canonical good filtration. 
\end{remark}

\begin{example}
	Let $\GR=\SLR$ and $\KR=SO(2,\R)$. The nilpotent cone in $\liep^*$ is the union of the zero orbit and the two principal nilpotent $K$-orbits, denoted by $\OO_\RN{1}$ and $\OO_\RN{2}$. The orbits $\OO_\RN{1}$ and $\OO_\RN{2}$ are Lagrangian subvarieties of the nilpotent cone in $\lieg^*=\slc^*$. The stablizer at any point of the principal orbits is $\ZZ/2\ZZ$, so there are exactly two choices of admissible line bundles on each of them. Picking either $\OO_\RN{1}$ or $\OO_\RN{2}$, the quantization of  the trivial $K$-equivariant line bundle gives the (Harish-Chandra module of) the irreducible spherical principal series representation of $\SLR$ with zero infinitesimal character. The corresponding associated variety is the entire nilpotent cone in $\liep^*$. The quantization of the nontrivial line bundle gives the reducible principal series representation of $\SLR$, which is the sum of two limits of discrete series representations. The associated variety of this reducible representation is again the entire nilpotent cone, but its irreducible components have associated varieties $\cl{\OO}_\RN{1}$ and $\cl{\OO}_\RN{2}$ respectively.
\end{example}

\subsection{Nonvanishing result}

In general, we can only conclude that the associated variety $\V(M)$ of $M$ constructed in Proposition \ref{thm:hc1} satisfies $\V(M) \subset \cl{\OO}_\liep$. The equality is preferable. We prove the following result which is an analogue of \cite[Lemma 4.5]{Losev3}.

\begin{proposition}\label{prop:nonzero}
	Suppose that $\codim(\partial \OO_\liep, \cl{\OO}_\liep) \geq 3$. Then the associated variety of the Harish-Chandra module $M$ is $\cl{\OO}_\liep$. If in addition the primitive ideal $\JJ_{\OO}$  associated to $\OO$ (see \S\,\ref{subsec:quan_c}) is maximal, then $M$ is irreducible.
\end{proposition}

\begin{proof}
	The proof is similar to that of \cite[Lemma 4.5]{Losev3}. Let $\ughh := \ugh \hten_{\kk[\hb]} \pws$ be the $\hb$-adic completion of $\ugh$. Then the quantized moment map $\mh$ can also be regarded as an algebra homomorphism $\ughh \to \Gamma(\OO,\Q)$. Taking sheaf cohomology of the short exact sequence 
	\[ 0 \to \Ehb \xrightarrow{\cdot \hb} \Ehb \to E \to 0 \]
	gives the exact sequence of $\ughh$-modules,
	\begin{equation} \label{exsq:gamma}
	0 \to \Gamma(\Ehb)/\hb\Gamma(\Ehb) \to \Gamma(E) \to H^1(\Ehb) \xrightarrow{\cdot \hb} H^1(\Ehb) \to H^1(E),
	\end{equation}
	where all cohomology are taken over $\OO_\liep$. By Step 2 of the proof of \cite[Lemma 4.5]{Losev3}, $H^1(E)$ is finitely generated over $\C[\OO_\liep]$ and hence also finitely generated over $\C[\lieg^*]=S\lieg$ with support in $\partial\OO_\liep$. The cohomology group $H^1(\Ehb)$ can be computed using \v{C}ech complex, which is $\hb$-complete and separated. By the same argument at the end of the proof of \cite[Lemma 5.6.3]{GordonLosev}, we can conclude that $H^1(\Ehb)$ is finitely generated over the Noetherian algebra $\ughh$.
	
	We claim that the kernel of multiplication map by $\hb$ in $H^1(\Ehb)$ is finitely generated over $S\lieg$ and supported on $\partial\OO_\liep$. Once the claim is proven, the first and second conclusions follow immediately. To prove the claim, we follow the argument in Step 6 of \cite[Lemma 4.5]{Losev3}. For the benefit of the reader and ourselves, 
we sketch the argument for this exercise of linear algebra here. To be honest, we are not confident that we could recall this simple argument in ten years or even in one week. For any $k \geq 0$, denote by $\Ker \hb^k$ (resp. $\Image \hb^k$) the kernel (resp. the image) of the multiplication map by $\hb^k$ in $H^1(\Ehb)$. Since $H^1(\Ehb)$ is finitely generated over $\ughh$, the ascending chain $\{ \Ker \hb^k \}_{k \geq 1}$ of submodules terminates. Assume that $\Ker \hb^k = \Ker\hb^N$ for some $N >0$ and any $k \geq N$. Then the descending chain $\{ \Ker\hb \cap \Image \hb^k\}_{k \geq 0}$ of submodules also terminates such that $\Ker\hb \cap \Image \hb^N = 0$. Now for each $k \geq 0$, we have isomorphisms of $S\lieg$-modules
	  \[  \frac{\Ker \hb^{k+1} + \Image \hb}{\Ker\hb^k + \Image \hb} \xrightarrow{~~\sim~~} \frac{\Ker \hb^{k+1}}{(\Ker \hb^{k+1} \cap \Image \hb) + \Ker \hb^k} \xrightarrow{~~\cdot  \hb^k~~} \frac{\Ker \hb \cap \Image \hb^k}{\Ker \hb \cap \Image \hb^{k+1}}.  \]
	Therefore $\{ \Ker\hb \cap \Image \hb^k\}_{k \geq 0}$ gives a finite filtration of $\Ker\hb$ whose subquotients are subquotients of $H^1(\Ehb)/\hb H^1(\Ehb)$, which is a submodule of the finitely generated $S\lieg$-module $H^1(E)$ supported on $\partial\OO_\liep$. The claim follows.
	
\end{proof}

\begin{definition} \label{defn:tall}
	We say that a nilpotent $K$-orbit $\OO_\liep$ is a \emph{tall orbit} if it satisfies the assumption $\codim(\partial \OO_\liep, \cl{\OO}_\liep) \geq 3$ in Proposition \ref{prop:nonzero}. Otherwise $\OO_\liep$ is a \emph{short orbit}.
\end{definition}

\begin{remark}
	Note that our construction only ensures an inclusion $\M_\hb / \hb\M_\hb \hookrightarrow \Gamma(\OO_\liep, E)$, while Vogan's Conjecture \ref{conj:vogan} predicts that the inclusion is in fact a bijection. In the case when $\GR$ is a complex group, our results also apply and give quantization of flat vector bundles on $\GR$-orbits. For a tall orbit, the $K$-types (or $G$-types) of the Harish-Chandra (bi)module do coincide with that of the global sections of the vector bundle by \cite{Mason-Brown}. The reason is that $H^1(\OO_\liep, E) = 0$ in this case.
\end{remark}

\subsection{Admissible orbits of exceptional groups} \label{subsec:list}

 Below we list the admissible nilpotent $K$-orbits of all noncompact real forms of exceptional groups, which are included in birationally rigid complex nilpotent orbits, and specify whether they are tall or short in the sense of \ref{defn:tall}. Here by admissible we mean that the orbits are admissible under the simply connected cover groups (In \cite{Noel1,Noel2} the terminology is \emph{sc-admissible}). The closure diagrams of all nilpotent $K$-orbits can be found in \cite{Dk1, Dk6, Dk5, Dk2, Dk3, Dk4, Dk7, Dk8}. We adopt the numbering of orbits from these papers and give references to specific papers for each real group. The list of admissible orbits can be found in \cite{Noel1, Noel2, Nevins}. A number with an asterisk means that the ambient complex orbit is not special (in the sense of Lusztig \cite{Lusztig}).

\begin{enumerate}[1.]
	\item 
	  {\bf GI} $= G_{2(2)}$: only two admissible orbits of (complex) dimensions $3$ and $4$ are included in birationally rigid orbits (in fact rigid orbits in this case). Only the one of dimension $3$ (the minimal orbit) is tall.
	\item 
	  {\bf FI} $= F_{4(4)}$: \cite{Dk1}
	  \begin{itemize}
	  	\item 
	  	  tall orbits: 1, 2*, 3*, 4, 5
	  	\item 
	  	  short orbits: 13
	  \end{itemize}
	\item 
	 {\bf FII} $= F_{4(-20)}$: \cite{Dk1}
	    \begin{itemize}
	    	\item 
	    	tall orbits: 11
	    	\item 
	    	short orbits: none
	    \end{itemize}
   \item
      {\bf EI} $=E_{6(6)}$: \cite{Dk2}
       \begin{itemize}
      	\item 
      		tall orbits: class 1, 3*
      	\item 
      		short orbits: class 11
      \end{itemize}
    \item
    {\bf EII} $=E_{6(2)}$: \cite{Dk2}
    \begin{itemize}
    	\item 
    	tall orbits: 1, 4*, 5*
    	\item 
    	short orbits: 17
    \end{itemize}
    \item
    {\bf EIII} $=E_{6(-14)}$: \cite{Dk2}
    \begin{itemize}
    	\item 
    	tall orbits: 1, 2
    	\item 
    	short orbits: none
    \end{itemize}
  \item 
   {\bf EIV} $=E_{6(-26)}$: \cite{Dk2}. No admissible orbit is included in a birationally rigid complex $E_6$-orbit.
 
 \item
 {\bf EV} $=E_{7(7)}$: \cite{Dk5}
 \begin{itemize}
 	\item 
 	tall orbits: 1, 2, 8*, 9*, 10, 11, 12, 13, 14, 15, 24
 	\item 
 	short orbits: 50
 \end{itemize}

  \item
  {\bf EVI} $=E_{7(-5)}$: \cite{Dk3, Dk4}
  \begin{itemize}
  	\item 
  	tall orbits: 1, 2, 3, 5*, 9, 10, 11, 16*, 18*
  	\item 
  	short orbits: 17*, 27
  \end{itemize}
    
   \item
    {\bf EVII} $=E_{7(-25)}$: \cite{Dk3, Dk4}
    \begin{itemize}
    	\item 
    	tall orbits: 1, 2, 3, 4, 5, 11, 12
    	\item 
    	short orbits: none
    \end{itemize}
   
   \item
   {\bf EVIII} $=E_{8(8)}$: \cite{Dk7, Dk8}
   	\begin{itemize}
   		\item 
   		tall orbits: 1, 2, 6*, 7, 8, 9, 10, 12*, 13*, 17*, 22*
   		\item 
   		short orbits: 23*, 24, 25, 31*, 32*, 38, 42, 43, 44
   	\end{itemize}
   
   \item
    {\bf EIX} $=E_{8(-24)}$: \cite{Dk6}
    \begin{itemize}
    	\item 
    	tall orbits: 1, 2, 3, 4*, 5*, 9, 10, 11, 15, 16*, 26
    	\item 
    	short orbits: 17*
    \end{itemize}
\end{enumerate}

\subsection{Speculations}\label{subsec:spec}
\begin{enumerate}[1.]
\item 
  For real classical groups, the tall orbits are very rare. Therefore Proposition \ref{prop:nonzero} is not enough. The analysis of the case of codimension 1 or 2 singularities will be studied in the future.

\item 
  By applying Saito's theory of mixed Hodge modules to Beilinson-Bernstein realization of Harish-Chandra modules using $\D$-modules \cite{BB}, Schmid and Vilonen \cite{SchmidVilonen} suggested that there is a natural good filtration, so-called \emph{Hodge filtration}, on each irreducible Harish-Chandra module. They conjectured that this filtration should be closely related to unitarity of the Harish-Chandra (\cite[Conj 1.10]{SchmidVilonen}). On the other hand, the Harish-Chandra module $\M$ constructed in our Theorem \ref{thm:hc1} also admits a natural filtration induced from the $\C$-action the orbit datum. We suspect that there is a close relationship between these two geometrically defined filtrations. It is possible that they only differ by a degree shift, at least when the ambient  nilpotent $G$-orbit $\OO$ is birationally rigid.
\end{enumerate}

\vskip 5em
\bibliographystyle{abbrv}
\bibliography{orbit.bib}

\begin{thebibliography}{10}

\bibitem{ABV}
J.~Adams, D.~Barbasch, and D.~A. Vogan, Jr.
\newblock {\em The {L}anglands classification and irreducible characters for
  real reductive groups}, volume 104 of {\em Progress in Mathematics}.
\newblock Birkh\"auser Boston, Inc., Boston, MA, 1992.

\bibitem{Arthur1}
J.~Arthur.
\newblock On some problems suggested by the trace formula.
\newblock In {\em Lie group representations, {II} ({C}ollege {P}ark, {M}d.,
  1982/1983)}, volume 1041 of {\em Lecture Notes in Math.}, pages 1--49.
  Springer, Berlin, 1984.

\bibitem{Arthur2}
J.~Arthur.
\newblock Unipotent automorphic representations: conjectures.
\newblock {\em Ast\'{e}risque}, (171-172):13--71, 1989.
\newblock Orbites unipotentes et repr\'{e}sentations, II.

\bibitem{Atiyah}
M.~F. Atiyah.
\newblock Complex analytic connections in fibre bundles.
\newblock {\em Trans. Amer. Math. Soc.}, 85:181--207, 1957.

\bibitem{AK}
L.~Auslander and B.~Kostant.
\newblock Polarization and unitary representations of solvable {L}ie groups.
\newblock {\em Invent. Math.}, 14:255--354, 1971.

\bibitem{BC}
V.~Baranovsky and T.~Chen.
\newblock Quantization of vector bundles on lagrangian subvarieties.
\newblock {\em International Mathematics Research Notices}, page rnx230, 2017.

\bibitem{BGKP}
V.~Baranovsky, V.~Ginzburg, D.~Kaledin, and J.~Pecharich.
\newblock Quantization of line bundles on lagrangian subvarieties.
\newblock {\em Selecta Math. (N.S.)}, 22(1):1--25, 2016.

\bibitem{Barbasch1}
D.~Barbasch.
\newblock Unipotent representations and the dual pair correspondence.
\newblock In {\em Representation theory, number theory, and invariant theory},
  volume 323 of {\em Progr. Math.}, pages 47--85. Birkh\"auser/Springer, Cham,
  2017.

\bibitem{BarbaschVogan}
D.~Barbasch and D.~A. Vogan, Jr.
\newblock Unipotent representations of complex semisimple groups.
\newblock {\em Ann. of Math. (2)}, 121(1):41--110, 1985.

\bibitem{BFFLS2}
F.~Bayen, M.~Flato, C.~Fronsdal, A.~Lichnerowicz, and D.~Sternheimer.
\newblock Deformation theory and quantization. {I}. {D}eformations of
  symplectic structures.
\newblock {\em Ann. Physics}, 111(1):61--110, 1978.

\bibitem{BFFLS1}
F.~Bayen, M.~Flato, C.~Fronsdal, A.~Lichnerowicz, and D.~Sternheimer.
\newblock Deformation theory and quantization. {II}. {P}hysical applications.
\newblock {\em Ann. Physics}, 111(1):111--151, 1978.

\bibitem{Beauville2}
A.~Beauville.
\newblock Fano contact manifolds and nilpotent orbits.
\newblock {\em Comment. Math. Helv.}, 73(4):566--583, 1998.

\bibitem{Beauville1}
A.~Beauville.
\newblock Symplectic singularities.
\newblock {\em Invent. Math.}, 139(3):541--549, 2000.

\bibitem{BB}
A.~Be{\u\i}linson and J.~Bernstein.
\newblock A proof of {J}antzen conjectures.
\newblock In {\em I. {M}. {G}elfand {S}eminar}, volume~16 of {\em Adv. Soviet
  Math.}, pages 1--50. Amer. Math. Soc., Providence, RI, 1993.

\bibitem{BeKaledin}
R.~Bezrukavnikov and D.~Kaledin.
\newblock Fedosov quantization in algebraic context.
\newblock {\em Mosc. Math. J.}, 4(3):559--592, 782, 2004.

\bibitem{BCHM}
C.~Birkar, P.~Cascini, C.~D. Hacon, and J.~McKernan.
\newblock Existence of minimal models for varieties of log general type.
\newblock {\em J. Amer. Math. Soc.}, 23(2):405--468, 2010.

\bibitem{BPW}
T.~Braden, N.~Proudfoot, and B.~Webster.
\newblock Quantizations of conical symplectic resolutions {I}: local and global
  structure.
\newblock {\em arXiv:1208.3863 [math.RT]}, 2017.

\bibitem{BK}
R.~Brylinski and B.~Kostant.
\newblock Nilpotent orbits, normality and {H}amiltonian group actions.
\newblock {\em J. Amer. Math. Soc.}, 7(2):269--298, 1994.

\bibitem{CM}
D.~H. Collingwood and W.~M. McGovern.
\newblock {\em Nilpotent orbits in semisimple {L}ie algebras}.
\newblock Van Nostrand Reinhold Mathematics Series. Van Nostrand Reinhold Co.,
  New York, 1993.

\bibitem{DeWilde}
M.~De~Wilde and P.~B.~A. Lecomte.
\newblock Existence of star-products and of formal deformations of the
  {P}oisson {L}ie algebra of arbitrary symplectic manifolds.
\newblock {\em Lett. Math. Phys.}, 7(6):487--496, 1983.

\bibitem{Deligne}
P.~Deligne.
\newblock D\'{e}formations de l'alg\`ebre des fonctions d'une vari\'{e}t\'{e}
  symplectique: comparaison entre {F}edosov et {D}e {W}ilde, {L}ecomte.
\newblock {\em Selecta Math. (N.S.)}, 1(4):667--697, 1995.

\bibitem{Dk1}
D.~v. \DJ{okovi\'c}.
\newblock The closure diagrams for nilpotent orbits of real forms of {$F_4$}
  and {$G_2$}.
\newblock {\em J. Lie Theory}, 10(2):491--510, 2000.

\bibitem{Dk6}
D.~v. \DJ{okovi\'c}.
\newblock The closure diagram for nilpotent orbits of the real form {EIX} of
  {$E_8$}.
\newblock {\em Asian J. Math.}, 5(3):561--584, 2001.

\bibitem{Dk5}
D.~v. \DJ{okovi\'c}.
\newblock The closure diagram for nilpotent orbits of the split real form of
  {$E_7$}.
\newblock {\em Represent. Theory}, 5:284--316, 2001.

\bibitem{Dk2}
D.~v. \DJ{okovi\'c}.
\newblock The closure diagrams for nilpotent orbits of real forms of {$E_6$}.
\newblock {\em J. Lie Theory}, 11(2):381--413, 2001.

\bibitem{Dk3}
D.~v. \DJ{okovi\'c}.
\newblock The closure diagrams for nilpotent orbits of the real forms {E} {VI}
  and {E} {VII} of {$E_7$}.
\newblock {\em Represent. Theory}, 5:17--42, 2001.

\bibitem{Dk4}
D.~v. \DJ{okovi\'c}.
\newblock Correction to: ``{T}he closure diagrams for nilpotent orbits of the
  real forms {E} {VI} and {E} {VII} of {$E_7$}'' [{R}epresent. {T}heory {\bf 5}
  (2001), 17--42; {MR}1826427 (2002b:22033)].
\newblock {\em Represent. Theory}, 5:503, 2001.

\bibitem{Dk7}
D.~v. \DJ{okovi\'c}.
\newblock The closure diagram for nilpotent orbits of the split real form of
  {$E_8$}.
\newblock {\em Cent. Eur. J. Math.}, 1(4):573--643, 2003.

\bibitem{Dk8}
D.~v. \DJ{okovi\'c}.
\newblock Corrections for ``{T}he closure diagram for nilpotent orbits of the
  split real form of {$E_8$}'' [{C}ent. {E}ur. {J}. {M}ath. {\bf 1} (2003), no.
  4, 573--643 (electronic); mr2040654].
\newblock {\em Cent. Eur. J. Math.}, 3(3):578--579, 2005.

\bibitem{Duflo}
M.~Duflo.
\newblock Th\'{e}orie de {M}ackey pour les groupes de {L}ie alg\'{e}briques.
\newblock {\em Acta Math.}, 149(3-4):153--213, 1982.

\bibitem{Fedosov}
B.~V. Fedosov.
\newblock A simple geometrical construction of deformation quantization.
\newblock {\em J. Differential Geom.}, 40(2):213--238, 1994.

\bibitem{Fu}
B.~Fu.
\newblock On {$\Bbb Q$}-factorial terminalizations of nilpotent orbits.
\newblock {\em J. Math. Pures Appl. (9)}, 93(6):623--635, 2010.

\bibitem{FJLS}
B.~Fu, D.~Juteau, P.~Levy, and E.~Sommers.
\newblock Generic singularities of nilpotent orbit closures.
\newblock {\em Adv. Math.}, 305:1--77, 2017.

\bibitem{GordonLosev}
I.~G. Gordon and I.~Losev.
\newblock On category {$\mathscr{O}$} for cyclotomic rational {C}herednik
  algebras.
\newblock {\em J. Eur. Math. Soc. (JEMS)}, 16(5):1017--1079, 2014.

\bibitem{GuttRawnsley}
S.~Gutt and J.~Rawnsley.
\newblock Natural star products on symplectic manifolds and quantum moment
  maps.
\newblock {\em Lett. Math. Phys.}, 66(1-2):123--139, 2003.

\bibitem{He}
H.~He.
\newblock Unipotent representations and quantum induction.
\newblock {\em arXiv:math/0210372}, 11 2002.

\bibitem{Hinich}
V.~Hinich.
\newblock On the singularities of nilpotent orbits.
\newblock {\em Israel J. Math.}, 73(3):297--308, 1991.

\bibitem{Howe}
R.~Howe.
\newblock Transcending classical invariant theory.
\newblock {\em J. Amer. Math. Soc.}, 2(3):535--552, 1989.

\bibitem{HuangLi}
J.-S. Huang and J.-S. Li.
\newblock Unipotent representations attached to spherical nilpotent orbits.
\newblock {\em Amer. J. Math.}, 121(3):497--517, 1999.

\bibitem{Kirillov_nil}
A.~A. Kirillov.
\newblock Unitary representations of nilpotent {L}ie groups.
\newblock {\em Russian Mathematical Surveys}, 17(4):53, 1962.

\bibitem{Kirillov_quan}
A.~A. Kirillov.
\newblock Geometric quantization [ {MR}0842909 (87k:58104)].
\newblock In {\em Dynamical systems, {IV}}, volume~4 of {\em Encyclopaedia
  Math. Sci.}, pages 139--176. Springer, Berlin, 2001.

\bibitem{Kontsevich}
M.~Kontsevich.
\newblock Deformation quantization of {P}oisson manifolds.
\newblock {\em Lett. Math. Phys.}, 66(3):157--216, 2003.

\bibitem{KostantRallis}
B.~Kostant and S.~Rallis.
\newblock On orbits associated with symmetric spaces.
\newblock {\em Bull. Amer. Math. Soc.}, 75:879--883, 1969.

\bibitem{Kronheimer}
P.~B. Kronheimer.
\newblock A hyper-{K}\"ahlerian structure on coadjoint orbits of a semisimple
  complex group.
\newblock {\em J. London Math. Soc. (2)}, 42(2):193--208, 1990.

\bibitem{Li}
J.-S. Li.
\newblock Singular unitary representations of classical groups.
\newblock {\em Invent. Math.}, 97(2):237--255, 1989.

\bibitem{LokeMa}
H.~Y. Loke and J.~Ma.
\newblock Invariants and {$K$}-spectrums of local theta lifts.
\newblock {\em Compos. Math.}, 151(1):179--206, 2015.

\bibitem{Losev4}
I.~Losev.
\newblock Finite-dimensional representations of {$W$}-algebras.
\newblock {\em Duke Math. J.}, 159(1):99--143, 07 2011.

\bibitem{Losev1}
I.~Losev.
\newblock Isomorphisms of quantizations via quantization of resolutions.
\newblock {\em Adv. Math.}, 231(3-4):1216--1270, 2012.

\bibitem{Losev3}
I.~Losev.
\newblock Quantizations of regular functions on nilpotent orbits.
\newblock {\em Preprint, arXiv:1505.08048}, 2015.

\bibitem{Losev2}
I.~Losev.
\newblock Deformations of symplectic singularities and orbit method for
  semisimple lie algebras.
\newblock {\em Preprint, arXiv:1605.00592}, 2016.

\bibitem{Lusztig}
G.~Lusztig.
\newblock A class of irreducible representations of a {W}eyl group.
\newblock {\em Nederl. Akad. Wetensch. Indag. Math.}, 41(3):323--335, 1979.

\bibitem{MaSunZhu}
J.-j. Ma, S.~Bingyong, and Z.~Chen-Bo.
\newblock Unipotent representations of real classical groups.
\newblock {\em arXiv:1712.05552}, 2017.

\bibitem{Mason-Brown}
L.~Mason-Brown.
\newblock Unipotent representations and microlocalization.
\newblock {\em arXiv:1805.12038}, 6 2018.

\bibitem{Namikawa1}
Y.~Namikawa.
\newblock A note on symplectic singularities.
\newblock {\em arXiv:math/0101028}, 02 2001.

\bibitem{Namikawa2}
Y.~Namikawa.
\newblock Flops and {P}oisson deformations of symplectic varieties.
\newblock {\em Publ. Res. Inst. Math. Sci.}, 44(2):259--314, 2008.

\bibitem{Namikawa3}
Y.~Namikawa.
\newblock Induced nilpotent orbits and birational geometry.
\newblock {\em Adv. Math.}, 222(2):547--564, 2009.

\bibitem{NestTsygan}
R.~Nest and B.~Tsygan.
\newblock Deformations of symplectic {L}ie algebroids, deformations of
  holomorphic symplectic structures, and index theorems.
\newblock {\em Asian J. Math.}, 5(4):599--635, 2001.

\bibitem{Nevins}
M.~Nevins.
\newblock Admissible nilpotent orbits of real and {$p$}-adic split exceptional
  groups.
\newblock {\em Represent. Theory}, 6:160--189, 2002.

\bibitem{Noel1}
A.~G. No{\"e}l.
\newblock Classification of admissible nilpotent orbits in simple exceptional
  real {L}ie algebras of inner type.
\newblock {\em Represent. Theory}, 5:455--493, 2001.

\bibitem{Noel2}
A.~G. No{\"e}l.
\newblock Classification of admissible nilpotent orbits in simple real {L}ie
  algebras {$E_{6(6)}$} and {$E_{6(-26)}$}.
\newblock {\em Represent. Theory}, 5:494--502, 2001.

\bibitem{Panyushev}
D.~I. Panyushev.
\newblock Rationality of singularities and the {G}orenstein property of
  nilpotent orbits.
\newblock {\em Funktsional. Anal. i Prilozhen.}, 25(3):76--78, 1991.

\bibitem{SchmidVilonen}
W.~Schmid and K.~Vilonen.
\newblock Hodge theory and unitary representations of reductive {L}ie groups.
\newblock In {\em Frontiers of mathematical sciences}, pages 397--420. Int.
  Press, Somerville, MA, 2011.

\bibitem{Schwartz}
J.~O. Schwartz.
\newblock {\em The determination of the admissible nilpotent orbits in real
  classical groups}.
\newblock ProQuest LLC, Ann Arbor, MI, 1987.
\newblock Thesis (Ph.D.)--Massachusetts Institute of Technology.

\bibitem{Sekiguchi}
J.~Sekiguchi.
\newblock Remarks on real nilpotent orbits of a symmetric pair.
\newblock {\em J. Math. Soc. Japan}, 39(1):127--138, 1987.

\bibitem{Trapa}
P.~E. Trapa.
\newblock Special unipotent representations and the {H}owe correspondence.
\newblock In {\em Functional analysis {VIII}}, volume~47 of {\em Various Publ.
  Ser. (Aarhus)}, pages 210--229. Aarhus Univ., Aarhus, 2004.

\bibitem{Vergne}
M.~Vergne.
\newblock Instantons et correspondance de {K}ostant-{S}ekiguchi.
\newblock {\em C. R. Acad. Sci. Paris S\'er. I Math.}, 320(8):901--906, 1995.

\bibitem{Vogan_Dixmier}
D.~A. Vogan, Jr.
\newblock Dixmier algebras, sheets, and representation theory.
\newblock In {\em Operator algebras, unitary representations, enveloping
  algebras, and invariant theory ({P}aris, 1989)}, volume~92 of {\em Progr.
  Math.}, pages 333--395. Birkh\"auser Boston, Boston, MA, 1990.

\bibitem{Vogan_AV}
D.~A. Vogan, Jr.
\newblock Associated varieties and unipotent representations.
\newblock In {\em Harmonic analysis on reductive groups ({B}runswick, {ME},
  1989)}, volume 101 of {\em Progr. Math.}, pages 315--388. Birkh\"auser
  Boston, Boston, MA, 1991.

\bibitem{Wallach}
N.~R. Wallach.
\newblock {\em Real reductive groups. {I}}, volume 132 of {\em Pure and Applied
  Mathematics}.
\newblock Academic Press, Inc., Boston, MA, 1988.

\bibitem{Xu}
P.~Xu.
\newblock Fedosov {$*$}-products and quantum momentum maps.
\newblock {\em Comm. Math. Phys.}, 197(1):167--197, 1998.

\bibitem{Yang}
J.-H. Yang.
\newblock The method of orbits for real {L}ie groups.
\newblock {\em Kyungpook Math. J.}, 42(2):199--272, 2002.

\end{thebibliography}

\end{document}